\documentclass[11pt,english]{amsart}
\usepackage[T1]{fontenc}
\usepackage[latin9]{inputenc}
\usepackage[a4paper]{geometry}
\usepackage{color}
\usepackage{babel}
\usepackage{amstext}
\usepackage{amsthm}
\usepackage{amssymb}
\usepackage{stmaryrd}
\usepackage{url}
\usepackage{hyperref}
\usepackage{multirow}

\hypersetup{
    colorlinks,
    linkcolor={red!50!black},
    citecolor={blue!50!black},
    urlcolor={blue!80!black}
}

\usepackage{caption} 
\usepackage{autobreak}
\usepackage{float}

\usepackage{dynkin-diagrams}

\usepackage{tikz}
\usetikzlibrary{arrows.meta}%
\usetikzlibrary{bending}%

\makeatletter

\numberwithin{equation}{section}
\numberwithin{figure}{section}

\theoremstyle{plain}
\newtheorem{thm}{\protect\theoremname}[section]
\newtheorem{prop}[thm]{\protect\propositionname}
\newtheorem{cor}[thm]{\protect\corollaryname}
\newtheorem{conj}[thm]{Conjecture}

\newtheorem{theintro}{Theorem}

\theoremstyle{plain}
\newtheorem{lem}[thm]{\protect\lemmaname}

\theoremstyle{definition}

\newtheorem{defn}[thm]{\protect\definitionname}

\theoremstyle{remark}
\newtheorem{rem}[thm]{\protect\remarkname}

\usepackage{babel}

\usepackage{verbatim}
\usepackage{xcolor}
\usepackage[all,cmtip]{xy}
\usepackage{bm}
\usepackage{etoolbox}
\usepackage{dynkin-diagrams}

\@namedef{subjclassname@2020}{%
  \textup{2020} Mathematics Subject Classification}

\title{Stringy invariants for abelian character varieties}


\author[C. Florentino]{Carlos Florentino}
 \address{Departamento de Matem\'atica, Faculdade de Ci\^encias, Univ.\ de Lisboa, Edf.\ C6, Campo Grande 1749-016 Lisboa, Portugal}
  \email{caflorentino@fc.ul.pt}

\author[\'A. Gonz\'alez-Prieto]{\'Angel Gonz\'alez-Prieto}
 \address{Departamento de
  \'Algebra, Geometr\'ia y Topolog\'ia, Facultad de Ciencias Matem\'aticas, Universidad Complutense de Madrid, Plaza Ciencias 3, 28040 Madrid, Spain}
  \address{Instituto de Ciencias Matem\'aticas (CSIC-UAM-UCM-UC3M), C/ Nicol\'as Cabrera 13-15, 28049 Madrid, Spain}
  \email{angelgonzalezprieto@ucm.es}
  
\author[A. Zamora]{Alfonso Zamora}
 \address{Departamento de Matem\'atica Aplicada a las TIC, ETSI Inform\'aticos, Universidad Polit\'ecnica de Madrid, Campus de Montegancedo, 28660 Madrid, Spain}
  \email{alfonso.zamora@upm.es}

\DeclareMathOperator{\image}{Im}

\newcommand{\ZZ}{\mathbb{Z}}
\newcommand{\NN}{\mathbb{N}}
\newcommand{\CC}{\mathbb{C}}

\newcommand{\QQ}{\mathbb{Q}}

\newcommand{\Spec}{\operatorname{Spec}}

\newcommand{\coker}{\operatorname{coker}}
\newcommand{\codim}{\operatorname{codim}}
\newcommand{\Hilb}{\operatorname{Hilb}}
\newcommand{\Hom}{\operatorname{Hom}}

\newcommand{\Pic}{\operatorname{Pic}}
\newcommand{\Sch}{\operatorname{Sch}}

\newcommand{\im}{\operatorname{im}}

\newcommand{\rk}{\operatorname{rk}}

\newcommand{\tr}{\operatorname{tr}}
\newcommand{\GL}{\operatorname{GL}}
\newcommand{\SL}{\operatorname{SL}}
\newcommand{\SO}{\operatorname{SO}}

\newcommand{\PGL}{\operatorname{PGL}}

\newcommand{\G}{\operatorname{G}}
\newcommand{\F}{\operatorname{F}}
\newcommand{\E}{\operatorname{E}}
\newcommand{\Sp}{\operatorname{Sp}}

\newcommand{\age}{\operatorname{age}}
\newcommand{\str}{\textup{str}}
\newcommand{\Ab}{\textup{Ab}}
\newcommand{\Conj}{\textup{Conj}}

\makeatother

\providecommand{\corollaryname}{Corollary}
\providecommand{\examplename}{Example}
\providecommand{\propositionname}{Proposition}
\providecommand{\remarkname}{Remark}
\providecommand{\theoremname}{Theorem}
\providecommand{\lemmaname}{Lemma}
\providecommand{\definitionname}{Definition}

\begin{document}
\begin{abstract}
We compute all stringy invariants, encoding orbifold cohomology numbers, of (the normalization of) the identity component of $G$-character varieties of free abelian groups and of moduli spaces of $G$-Higgs bundles on abelian varieties, for all complex connected reductive groups $G$. As an application, we provide a direct proof of a topological mirror symmetry statement: the equality of the stringy invariants for Langlands dual groups. In the framework of Ginzburg--Kaledin local resolutions, we discuss the meaning of these stringy invariants, showing that symplectic resolutions of these moduli spaces can only exist in the cases of groups of Dynkin type $A$, $B$, or $C$. When such resolutions exist, the computed stringy Hodge numbers agree with the actual Hodge numbers of the resolution.
\end{abstract}

\maketitle

\section{Introduction}

Let $G$ be a complex connected reductive group with maximal torus $T \subseteq G$ and associated Weyl group $W$, with its natural action on $T$. In this work, we study the stringy invariants of the quotient variety $T^r/W$, where $W$ acts diagonally.

This variety plays a fundamental role in moduli space theory due to its close relation with character varieties of abelian varieties. Concretely, given a smooth compact manifold $M$, its $G$-character variety is the (affine) geometric invariant theory quotient
$$
    \mathfrak{X}_M(G) = \Hom(\pi_1(M), G) \sslash G,
$$
where $G$ acts on $\Hom(\pi_1(M), G)$ by conjugation. The geometry of these spaces is very involved and has been the object of intense research in the last decades, due to both its relation with moduli spaces of Higgs bundles through the celebrated non-abelian Hodge correspondence \cite{simpson1992higgs}, and its fundamental role in knot theory and geometric topology \cite{culler1983varieties,Si}.

In the particular case that $M = A$ is a complex abelian variety (or, more generally, a complex torus) of complex dimension $d$, its fundamental group is $\pi_1(A) = \ZZ^{2d}$ and thus the corresponding character variety is $\mathfrak{X}_A(G) = \mathcal{C}_{2d}(G) \sslash G$, where $\mathcal{C}_{2d}(G)$ is the commuting variety whose points are tuples of $2d$ pairwise commuting elements of $G$. In this case, there is a natural injective morphism $\psi: T^{2d}/W \to \mathfrak{X}_A(G)$,
which becomes a normalization of the connected component of the identity representation. In general, $\psi$ is not an isomorphism \cite{florentino2021singular,BFN}, but it induces an isomorphism of mixed Hodge structures \cite{florentino2024mixed} with the connected component of the identity. For this reason, studying the geometry of $T^{2d}/W$ is equivalent to studying the geometry of the distinguished connected component of the identity in $\mathfrak{X}_A(G)$. More generally, for $r \geq 1$ not necessarily even, the space $\mathfrak{X}_r(G):=\mathcal{C}_{r}(G) \sslash G$ can be identified with the character variety of representations of the free abelian group $\ZZ^r$ into $G$, and the natural morphism 
\[T^r/W \to \mathfrak{X}_r(G) = \mathcal{C}_{r}(G) \sslash G\]
identifies $T^r/W$ with the normalization of the connected component of the identity of $\mathfrak{X}_r(G)$.
This important morphism can be viewed as a higher dimensional commuting Lie group Chevalley map (the well-known identification between $\mathfrak t/W$ and $\mathfrak g \sslash G$, with $\mathfrak g$ and $\mathfrak t$ the Lie algebras of $G,T$,  respectively) and generalizes the case of a $d$-dimensional complex abelian variety $A$, as $\mathfrak{X}_{2d}(G) = \mathfrak{X}_A(G)$.

\subsection*{Stringy invariants.} Despite the seemingly simple structure of $T^r/W$, its singularities and geometry can be very intricate for higher $r$, as it corresponds to the associated character variety. For this reason, more relevant than its usual (co)homology (in the Euclidean topology), is its Chen--Ruan cohomology \cite{CR}, also known as \emph{orbifold cohomology}, given by
\begin{equation}\label{eq:CR-intro}
H^\bullet_{\text{CR}}(T^r/W; \mathbb{C}) := \bigoplus_{[w] \in \text{Conj}(W)} H^{\bullet}((T^r)^w / C(w); \mathbb{C})[2 \, \age(w)]\; .
\end{equation}
Here, $\text{Conj}(W)$ is the set of conjugacy classes of $W$, $w \in [w]$ is any representative of the class, $(T^r)^w$ is the locus fixed by $w\in W$ and $C(w)$ is the centralizer of $w$ in $W$, acting as the residual symmetry group on the fixed locus. The shift $2\,\age(w)$ in the cohomology complex is the so-called \emph{fermionic shift} \cite{fantechi2003orbifold} that captures the codimension of $(T^r)^w$ inside $T^r$ (see Section \ref{subsec:age-shift}). This orbifold cohomology inherits a natural mixed Hodge structure, whose associated mixed Hodge polynomial is known as the \emph{stringy mixed Hodge polynomial} defined by:
$$
\mu^{\str}(T^r/W)(t, u, v) = \sum_{k, p, q} h_{k,p,q}^\str(T^r/W)\, t^k\,u^p\,v^q,
$$
where $h_{k,p,q}^\str(T^r/W) = \dim \left(H^k_{\text{CR}}(T^r/W; \mathbb{C})\right)^{p,q}$ are the \emph{stringy Hodge numbers}. Since fermionic shifts are not generally integers, Chen-Ruan cohomology is not necessarily integer graded, but our computations will show that $\mu^{\str}(T^r/W)$ is actually a polynomial in $\ZZ[t,(uv)^{1/2}]$ with non-negative coefficients.

The central result of this work is an explicit formula for $\mu^{\str}(T^r/W)$ purely in terms of the action of $W$ on the character lattice $\Lambda$ of $G$. Let us briefly introduce some notation; for more details, please see Section \ref{sec:stringy-mhpoly}. Let $I_{\Lambda}$ be the identity map on $\Lambda$ and, given $w \in W$, decompose the cokernel of $I_{\Lambda}-w: \Lambda \to \Lambda$ into its free and torsion parts:
$$
    \coker(I_{\Lambda}-w) = \Phi_w \oplus D_w\; ,
$$
where $\Phi_w$ is the free part and $D_w$ the torsion part. Since the action of any other element  $g\in C(w)$ (the centralizer of $w$) descends to $\coker(I_{\Lambda}-w)$, it induces morphisms on the free and torsion parts, denoted by
$$
    \phi_w(g): \Phi_w \to \Phi_w\quad , \quad \tau_w(g): D_w \to D_w\; .
$$
Denote by $n_w(g) = |\coker(I_{D_w}-\tau_w(g))|$ the cardinality of the cokernel (a finite group) of the map 
$$I_{D_w}-\tau_w(g): D_w \to D_w\; .$$ 
With these notions at hand, we have the following main result.

\begin{theintro}[Theorem \ref{thm:algorithm_CV}]\label{thm:main-intro}
    The stringy mixed Hodge polynomial of $T^r/W$ is given by
    \[
        \mu^{\str}(T^r/W)(t,u,v)=\sum_{[w] \in \textup{Conj}(W)}
        \frac{(t^{2}uv)^{\frac{r}{2}\rk \image(I_\Lambda -w)}}{|C(w)|}\left(\sum_{g \in C(w)} \, n_{w}(g)^r \det\left(I_{\Phi_w} + tuv\,\phi_w(g)\right)^r \right).
    \]
\end{theintro}

Analogously to Theorem \ref{thm:main-intro}, we can also consider orbifold cohomology with compact support, and we obtain an analogous formula for the stringy mixed Hodge polynomial (see Corollary \ref{cor:compact-support}). From these polynomials, we can readily obtain other algebraic invariants of $T^r/W$, such as its stringy Poincar\'e polynomial (corresponding to setting $u = v = 1$ in the stringy mixed Hodge polynomial). Also of particular interest is the specialization $t = -1$ for compactly supported orbifold cohomology,  which is known as the stringy $E$-polynomial $E^{\str}(T^r/W) \in \ZZ[(uv)^{1/2}]$, and the Euler characteristic (corresponding to $u = v = -t = 1$).

\subsection*{Mirror symmetry.} All the information appearing in Theorem \ref{thm:main-intro} is written purely in terms of Lie-theoretic data of $G$, and the action of $W$ on $\Lambda$.
In particular, this description allows us to provide a simple and direct proof of topological mirror symmetry for these varieties. First observed by Thaddeus in \cite{Th}, it claims that the stringy mixed Hodge polynomials agree for a connected reductive group $G$ and its Langlands dual group $^LG$. In \cite{Th}, such an equality is obtained by identifying the structures of the $C(w)$-modules $H^{\bullet}((T_G^r)^w / C(w))$ and $H^{\bullet}((T_{^LG}^r)^w / C(w))$ and the corresponding fermionic shifts ($T_G$ and $T_{^LG}$ are maximal tori of $G$ and $^LG$, respectively).

In our approach, due to the formula provided by Theorem \ref{thm:main-intro}, it is possible to directly prove (see Theorem \ref{thm:equality-terms}) that, for any $w \in W$ and $g \in C(w)$, both $n_w(g)$ and $\det\left(I_{\Phi_w} + tuv\,\phi_w(g)\right)$ are equal for $G$ and $^LG$ because, roughly speaking, they correspond to transpose maps in an appropriate representation. We point out that in \cite{Th}, this mirror symmetry is only proven for the $E$-polynomial, whereas our proof works directly with mixed Hodge numbers and, therefore, all stringy invariants, in the usual and compactly supported versions.

\subsection*{Computed examples.} The most important feature of Theorem \ref{thm:main-intro} is that it is completely explicit. In particular, it provides a way to compute stringy mixed Hodge polynomials for arbitrary groups $G$, including the exceptional ones. The classification of conjugacy classes in the Weyl group $W$ relies on Carter's graph-theoretic classification via Dynkin-like diagrams \cite{Carter1972}. In Sections \ref{ssec:SO7} and \ref{ssec:G2}, we exemplify this procedure in the case $G = \SO_7$ (with Dynkin diagram of type $B_3$) and $G = \G_2$. In particular, we describe how to perform these calculations effectively, and have implemented the aforementioned algorithm in the symbolic algebra system SageMath (see \cite{code_repository}).

Table \ref{tab:summary} records the results for an elliptic curve (the case $d=1$), for several classical and exceptional groups of rank up to $4$. For completeness, in Appendix \ref{app:calculations}, we provide a full description of the stringy mixed Hodge polynomials for $T^r/W$ for arbitrary $r$ and all semisimple exceptional $G$. It is worth pointing out that the algorithm obtained with this procedure is very efficient: all these calculations were performed in less than one minute, so we can expect higher rank cases to be addressed similarly.

\subsection*{Numerical properties and Batyrev's conjecture.}

The stringy $E$-polynomial was introduced in the framework of motivic integration and mirror symmetry by Batyrev \cite{Ba} to generalize classical Hodge--Deligne polynomials to algebraic varieties with mild singularities, such as log-terminal or Gorenstein canonical singularities. Parallel to Batyrev's birational-geometric invariant, Chen and Ruan introduced its orbifold cohomology framework $H^\bullet_{\mathrm{CR}}$ on orbifolds. Yasuda \cite{Ya} proved, using motivic integration over Deligne--Mumford stacks, that for a quotient variety $X/F$ with Gorenstein singularities, where $F$ is a finite group, the motivic stringy $E$-function of the quotient stack $[X/F]$ coincides with the orbifold Hodge function defined via Chen--Ruan orbifold cohomology.
This identification interprets the stringy invariants as actual dimensions of a graded, pure cohomology space.
In searching for a cohomology theory for general singular spaces, Batyrev formulated his positivity conjecture (see \cite[Conjecture 3.10]{Ba}) stating that when the stringy $E$-function of a projective Gorenstein canonical variety $X$ is actually a polynomial, its stringy Hodge numbers are non-negative.

The varieties $T^{2d}/W$ that we consider are not projective as in Batyrev's conjecture. On the other hand, the stringy mixed Hodge numbers of the quotient singularities appearing here are automatically non-negative because they coincide with the true dimensions of Chen--Ruan orbifold cohomology.
In Section \ref{sec:batyrev}, in the spirit of Batyrev's conjecture, we conjecture the non-negativity of the coefficients of the stringy $E$-polynomial of $T^{2d}/W$, for $d=1$ (corresponding to the elliptic curve) and all semisimple groups, providing strong positive evidence for it, based on a number of computed examples (see Conjecture \ref{conj:positivity}). We also prove it to be true (even for $d>1$) when $G=\SO_7$ and $G=\G_2$ (see Corollary \ref{cor:G2_SO7_conj}), and show that the hypotheses of the conjecture are sharp (being false for $d=2$ and $G=\SL_3$; and also for $G = \GL_2$, and general $d$). We also show in Proposition \ref{prop:palindromic} that the stringy $E$-polynomial of $T^{2d}/W$ is palindromic for every reductive $G$ and all $d\geq 1$.

\subsection*{Moduli spaces of Higgs bundles on abelian varieties.} Despite the main focus of this paper being on the character variety side, it is straightforward to translate these results to the moduli space of $G$-Higgs bundles (of trivial topological type) on an abelian variety $A$ of complex dimension $d$. In that case, it can be proven that the normalization of the moduli space $\mathcal{M}_A(G)$ of $G$-Higgs bundles on $A$ is isomorphic to the quotient $\mathcal{M}_A(T) / W$ of the moduli space of $T$-Higgs bundles over $A$ modulo the action of the Weyl group $W$. It turns out that $\mathcal{M}_A(T)$ can also be understood in terms of the character lattice of $G$ so, arguing as for character varieties, we obtain the following result.

\begin{theintro}[Section \ref{sec:mixedHodge-Higgs}]\label{thm:algorithm_Higgs-intro}
    The stringy mixed Hodge polynomial of the normalization $\hat{\mathcal{M}}_A(G) = \mathcal{M}_A(T) / W$ is given by
\[
        \mu^{\str} (\mathcal{M}_A(T) / W)(t,u,v) 
        =\sum_{[w] \in \textup{Conj}(W)}
        \frac{(t^{2}uv)^{d\rk \image(I_\Lambda -w)}}{|C(w)|}\left(\sum_{g \in C(w)} \, n_{w}(g)^{2d}\, p_{g,w}(tu)^d\, p_{g,w}(tv)^d \right),
\]
with $p_{g,w}(x):=\det (I_{\Phi_w} + x\,\phi_w(g))$.
\end{theintro}

Here the fermionic shift is always an integer, so the stringy mixed Hodge polynomial is in $\ZZ[t,u,v]$. Using Theorem \ref{thm:algorithm_Higgs-intro}, analogous conclusions can be readily obtained for the moduli space of Higgs bundles over an abelian variety, such as topological mirror symmetry (Corollary \ref{cor:mirror-symmetry-Higgs}) and explicit calculations for any reductive group $G$, including the exceptional ones.

\subsection*{Symplectic resolutions.} The orbifold cohomology studied in this paper has a deep interpretation in terms of symplectic resolutions. Let us consider a quotient $X/F$ of a variety $X$ under the action of a finite group $F$, and suppose that $X/F$ is a symplectic variety, i.e. $X/F$ admits a holomorphic symplectic form on its regular locus that extends to a, possibly degenerate, holomorphic form on any resolution of singularities. It is well-known (\cite[Theorem 3.12]{Ba}) that if the singular space $X/F$ admits a symplectic resolution $Z \to X/F$ (i.e, one in which the holomorphic form on $Z$ is a true symplectic form), then the stringy $E$-polynomial coincides with the (standard) $E$-polynomial of $Z$. From this perspective, the stringy invariants of $X/F$ can be seen as a sort of ``virtual invariants'' of a potential symplectic resolution (if that resolution exists).

This is particularly interesting in the case of character varieties of abelian varieties, since $T^{2d}$ admits a natural $W$-invariant holomorphic symplectic form, so that $T^{2d}/W$ are symplectic varieties. It makes sense, thus, to discuss whether $T^{2d}/W$ admits a symplectic resolution. In this direction, we prove (Corollary \ref{cor:existence-crepant}) that a symplectic resolution can only exist if $d = 1$ (i.e.\ for character varieties of an elliptic curve). Furthermore, following Ginzburg and Kaledin \cite{GK04}, we also prove that, even for $T^2/W$, such a symplectic resolution can only exist when all components of the Dynkin diagram of $G$ are of type $A$, $B$ or $C$ (Theorem \ref{thm:resolution_ABC}). In particular, no such resolution exists for exceptional groups, and thus the stringy mixed Hodge polynomial is a subtle invariant that captures ``virtual Hodge numbers''. This also highlights the exceptionality of the positivity of the stringy Hodge numbers and the palindromic $E$-polynomial in these cases, since they are not secretly justified by an underlying symplectic resolution.

\subsection*{Structure of the paper} In Section \ref{sec:Chen-Ruan} we review the notion of Chen--Ruan cohomology and the corresponding stringy invariants. In Section \ref{sec:Betti-Dolbeault} we review the fundamentals of $G$-character varieties of free abelian groups and of moduli spaces of $G$-Higgs bundles over an abelian variety, explaining how, in both cases, the (normalization of the) interesting connected components, are quotients under the Weyl group of $G$. In Section \ref{sec:symplectic-resolution} we discuss the existence of symplectic resolutions for these moduli spaces depending on $G$. Section \ref{sec:stringy-mhpoly} proves the main results of this paper, providing explicit closed formulas for the stringy mixed Hodge polynomials of abelian character varieties and moduli spaces of Higgs bundles over an abelian variety. In Section \ref{sec:topological-mirror-symmetry}, we exploit this description to provide a direct alternative proof of Thaddeus' topological mirror symmetry for these spaces, whose details are in Appendix \ref{app:proof-mirror-symmetry}. As an application, in Section \ref{sec:examples} we explicitly compute these stringy invariants for $G = \SO_7$ and $G = \G_2$, and list them for $d=r/2=1$ and several groups of rank less than $5$. In Section \ref{sec:batyrev} we propose and provide evidence for a conjecture inspired by Batyrev's positivity conjecture, and show that our stringy $E$-polynomials are palindromic. Appendix \ref{app:calculations} contains a full description of the stringy mixed Hodge polynomials of semisimple exceptional groups for arbitrary $r$.

\subsection*{Notations and conventions}
Throughout this article, we work over the field $\CC$ of complex numbers. Given a connected reductive group $G$, we shall fix a maximal torus $T \subseteq G$, which automatically defines an associated Weyl group $W$. The character lattice of $T$ is the free abelian group $\Lambda = \Hom(T, \CC^*)$ of algebraic group homomorphisms, whose rank is called the rank of $T$ (and of $G$), and contains the root lattice $Q \subseteq \Lambda$. The geometry of the root system $R \subseteq Q$ defines the Dynkin diagram of $G$.

Analogously, we can also consider the cocharacter lattice $\Lambda^\vee = \Hom(\CC^*, T)$, with associated coroot system $R^\vee \subseteq \Lambda^\vee$. The tuple $(\Lambda, R, \Lambda^\vee, R^\vee)$ is called the root datum of $G$ and characterizes $G$ up to isomorphism. Its Langlands dual group $^L G$ is then the unique connected reductive group with root datum $(\Lambda^\vee, R^\vee, \Lambda, R)$.

\subsection*{Acknowledgements}

CF wishes to thank Universidad Complutense de Madrid and ICMAT for hospitality during his visit while part of this work was completed, and has been partially supported by CEMS.UL - Center for Mathematics Studies at the University of Lisbon, and FCT projects     UID/04561/2025  (doi.org/10.54499/UID/04561/2025). AGP has been partially supported by project PID2024-156578NB-I00 funded by MICIU /AEI /10.13039/501100011033 / FEDER, EU, as well as the bilateral AEI-DFG project Celestial Mechanics, Hydrodynamics, and Turing Machines / Himmelsmechanik, Hydrodynamik und Turing-Maschinen (AQUACELL), with reference codes PCI2024-155042-2 and PCI2024-155062-2. AZ has been partially supported by Spanish Ministerio de Ciencia, Innovaci\'on y Universidades project PID2022-142024NB-I00.

\section{Chen-Ruan cohomology and polynomial invariants}\label{sec:Chen-Ruan}

In this section we recall the ideas behind Chen--Ruan cohomology and the corresponding stringy invariants. We also recall the polynomials defined from mixed Hodge structures on complex algebraic varieties.

\subsection{Hodge structures and virtual classes}

Throughout this paper, we work over $\mathbb{C}$. Every algebraic variety $X$ (smooth or not) admits a mixed Hodge structure (MHS) on its cohomology groups $H^\bullet(X,\CC)$, 
which allows one to define the mixed Hodge vector spaces $H^{k,p,q}(X,\mathbb{C})$, where $k$ is the degree, and $p,q$ are the weights (see \cite[Theoreme 3.2.5]{De1} and \cite[Proposition 8.3.9]{De2}). In turn, these define the \emph{mixed Hodge numbers}
\[
h^{k,p,q}(X):=\dim_{\mathbb{C}}H^{k,p,q}(X,\mathbb{C})\; ,
\]
and the three-variable \emph{mixed Hodge polynomial}
\[
\mu(X)(t,u,v):=\sum_{k,p,q\geq0}h^{k,p,q}(X)\,t^{k}u^{p}v^{q} \in \ZZ[t,u,v]\; .
\]
If $X$ satisfies that $h^{k,p,q}(X) = 0$ for $p \neq q$, the variety $X$ is said to be \emph{Hodge-Tate} or of \emph{balanced type}. 

Analogously, the compactly supported cohomology groups $H_c^\bullet(X,\CC)$ also admit a natural MHS defining the compactly supported mixed Hodge numbers
$h_c^{k,p,q}(X)$ (the dimensions of the corresponding vector spaces
$H_c^{k,p,q}(X,\mathbb{C})$),
as well as the \emph{compactly supported mixed Hodge polynomial}
\[
\mu_c(X)(t,u,v):=\sum_{k,p,q\geq0}h_c^{k,p,q}(X)\,t^{k}u^{p}v^{q} \in \ZZ[t,u,v]\; .
\]

\begin{rem}
\label{rem:mu-and-muc} When $X$ is irreducible of complex dimension $d$,
and its mixed Hodge structure satisfies Poincar\'e duality (eg.\ when $X$ is smooth
or $X=Y/F$, the quotient of a smooth irreducible variety $Y$ by a finite
group $F$) the compactly supported mixed Hodge polynomial can be obtained from the usual one as (see \cite[Remark 3.10]{FS})
\[
\mu_c(X)(t,u,v)=(t^{2}uv)^{d}\,\mu(X)(t^{-1},u^{-1},v^{-1})\; .
\]

\end{rem}

An important specialization of the mixed Hodge polynomial is the \emph{$E$-polynomial} (also called \emph{Serre polynomial}), which can be obtained from $\mu_{c}(X)$
by plugging
$t=-1$, that is 
\[
E(X)(u,v):=\mu_c(X)(-1,u,v) \in \ZZ[u,v]\; .
\]
From it, we can recover the topological Euler characteristic
of $X$ as 
\[
\chi(X)=E(X)(1,1)=\mu_c(X)(-1,1,1)=\sum_{k}(-1)^{k}\dim_{\CC}H_{c}^{k}(X,\CC)\; .
\]

\begin{rem}
(1) Note that the Euler characteristic $\chi(X)$ and the compactly supported 
one $\chi_c(X)$ agree on complex varieties. Indeed, they agree on smooth varieties, 
by Poincar\'e duality and the fact that the real dimension is even. Moreover, if the singular locus $X^{\textup{sing}}\subseteq X$ of $X$ is smooth (and again $X^{\textup{sing}}$ has even real dimension),
then the additivity of $\chi$ and of $\chi^c$ implies 
\[
\chi(X)=\chi(X^{\textup{sing}})+\chi(X\setminus X^{\textup{sing}})=
\chi^c(X^{\textup{sing}})+\chi^c(X\setminus X^{\textup{sing}})=\chi^c(X)\; .
\]
Finally, the same holds in general, by using the finite singular stratification.

(2) The $E$-polynomial can also be seen as a specialization of a more general
invariant, known as the virtual class. Let ${\rm KVar}$ be the Grothendieck
ring of algebraic varieties, generated by isomorphism classes $[X]$ of algebraic varieties
modulo cut-and-paste relations $[X]=[Y]+[X-Y]$, $Y\subseteq X$ being
a closed subvariety, and where the ring multiplication comes from
the cartesian product $[X\times Z]=[X]\cdot[Z]$. The elements in
this ring are called virtual classes of the varieties and the $E$-polynomial
factors through the Grothendieck ring as a ring homomorphism $E:{\rm KVar}\rightarrow\mathbb{Z}[u,v]$. 
\end{rem}

\subsection{Orbifold cohomology}

Chen--Ruan cohomology \cite{CR} provides an algebro-geometric foundation for the ``orbifold quantum cohomology'' used in string theory. Let $F$ be a finite group acting on a smooth complex variety $X$ and let $\mathcal{X} = [X/F]$ be the corresponding orbifold or \emph{global quotient stack}. The ordinary cohomology of the quotient, $H^*(X/F)$ captures only the ground states of the ``untwisted sector'', while string theory requires including ``twisted sectors'' which are closed string states that are only periodic up to the action of a group element $g \in F$. Chen--Ruan cohomology achieves this by decomposing the underlying vector space as a direct sum over the conjugacy classes of the group $F$.

\begin{defn}[{\cite[Definition 3.1.2]{CR}}]\label{def:CR-cohomology}
The Chen--Ruan cohomology, also known as the orbifold cohomology, of $[X/F]$ is
\begin{equation}\label{eq:CR}
H^\bullet_{\text{CR}}([X/F]; \mathbb{C}) := \bigoplus_{[g] \in \text{Conj}(F)} H^{\bullet}(X^g / C(g); \mathbb{C})[2 \, \age(g)]\; .
\end{equation}
Here, $\text{Conj}(F)$ is the set of conjugacy classes of $F$, $g \in [g]$ is any representative of the class, $X^g$ is the fixed locus of the action of $g\in F$ and $C(g)$ is the centralizer of $g$ in $F$, acting as the residual symmetry group on the fixed locus $X^g$. The shifting  $2\,\age(g)$ in the cohomology complex is the so-called \emph{fermionic shift} (cf.\ \cite[Definition 3.1.1]{CR}), and its definition is recalled in subsection \ref{subsec:age-shift}.
\end{defn}

In parallel to the usual definition, we can also consider the compactly supported Chen--Ruan cohomology, defined as
\begin{equation*}
H^\bullet_{\text{CR},c}([X/F]; \mathbb{C}) := \bigoplus_{[g] \in \text{Conj}(F)} H^{\bullet}_c(X^g / C(g); \mathbb{C})[2 \, \age(g)]\; .
\end{equation*}

\begin{rem}
    From a modern stack-theoretic perspective, the Chen--Ruan cohomology is equivalent to the ordinary singular cohomology of the \emph{inertia stack $\mathcal{I}(\mathcal{X})$} associated to the orbifold $\mathcal{X}$. For a global quotient stack $\mathcal{X} = [X/F]$, the inertia stack parameterizes pairs of points and group elements that fix them, decomposing topologically into the disjoint union of the twisted sectors $X^g / C(g)$:
    \begin{equation}
    \mathcal{I}([X/F]) \cong \coprod_{[g] \in \text{Conj}(F)} [X^g / C(g)]\; 
    \end{equation}
    Consequently, Chen--Ruan cohomology can be interpreted as the cohomology of the inertia stack, further equipped with a fundamentally modified grading defined by the fermionic shift.
\end{rem}

\subsection{Stringy mixed Hodge polynomial}\label{sec:stringy-MH-poly}

From its very description, the Chen--Ruan cohomology admits a natural mixed Hodge structure inherited from the one of $X$. Explicitly, formula (\ref{eq:CR}) admits an interpretation as direct sum of mixed Hodge structures
$$
H^\bullet_{\text{CR}}([X/F]; \mathbb{C}) = \bigoplus_{[g] \in \text{Conj}(F)} H^{\bullet}(X^g / C(g); \mathbb{C}) \otimes_{\CC} \CC(\age(g))\; ,
$$
where $\CC(m) = \CC(1)^{\otimes m}$ is the $m$-Tate twist of weight $-2m$, and $H^{\bullet}(X^g / C(g); \mathbb{C})$ is equipped with its natural MHS coming from the complex variety $X^g/C(g)$. Taking into account that the Hodge structure on $\CC(m)$ satisfies $\CC(m)^{-m,-m} = \CC$, and $\CC(m)^{p,q}=0$
otherwise, we can consider the associated mixed Hodge polynomial.

\begin{defn}
\label{def:str_hodge_pol}
    Let $X$ be a smooth complex variety acted by a finite group $F$. The \emph{stringy mixed Hodge polynomial} of $[X/F]$ is 
$$
    \mu^{\str}(X/F)(t, u, v) = \sum_{[g] \in \text{Conj}(F)} \mu(X^g / C(g))\, (t^{2}uv)^{\age(g)}.
$$
\end{defn}
In principle, the fermionic shift may not be an integer (see Subsection  \ref{subsec:age-shift}), so $\mu^{\str}(X/F)$ is not a polynomial in general. A special case when fermionic shift is indeed an integer is when $X/F$ is Gorenstein, and in this case $\mu^{\str}(X/F) \in \ZZ[t,u,v]$. As we will see, for the cases studied in this work, the fermionic shifts will be half-integers, and thus the stringy mixed Hodge polynomial takes values in $\ZZ[t,u^{1/2},v^{1/2}]$.

In particular, the specialization $u = v = 1$ of the stringy mixed Hodge polynomial is the so-called \emph{stringy Poincar\'e polynomial}
\begin{equation}
    \label{eq:str_Poincare_pol}
    P^{\str}(X/F)(t) = \sum_{[g] \in \text{Conj}(F)} P(X^g / C(g))\, t^{2\,\age(g)} \in \ZZ[t]\; .
\end{equation}
In a similar vein, using compactly supported Chen--Ruan cohomology, we can consider the \emph{compactly supported stringy mixed Hodge polynomial}
\begin{equation}
    \label{eq:str_compact_hodge_pol}
    \mu_c^{\str}(X/F)(t, u, v) = \sum_{[g] \in \text{Conj}(F)} \mu_c(X^g / C(g))\, (t^{2}uv)^{\age(g)}.
\end{equation}
Analogously, the specialization $t = -1$ of the compactly supported stringy mixed Hodge polynomial is called the \emph{stringy $E$-polynomial}
\begin{equation}
    \label{eq:str_e_pol}
    E^\str(X/F) =\mu_c^{\str}(X/F)(-1, u, v) = \sum_{[g] \in \text{Conj}(F)}(-1)^{2\,\age(g)} E(X^g / C(g))\, (uv)^{\age(g)}.
\end{equation}
Notice that when $\age(g)$ is an integer, the leading sign is $(-1)^{2\,\age(g)} = 1$.

\begin{rem}
The above formula shows that, for $X/F$ to be balanced, it is enough that all quotients $X^g/C(g)$ are balanced. In this case, it is customary to write the stringy mixed Hodge polynomial in the variable $q := uv$ as
$$
    \mu^{\str}(X/F)(t, q) = \sum_{[g] \in \text{Conj}(F)} \mu(X^g / C(g))\, (t^2q)^{\age(g)}.
$$
\end{rem}

\section{Betti and Dolbeault moduli spaces of abelian varieties}\label{sec:Betti-Dolbeault}

Here we recall the two moduli over abelian varieties linked by the non-abelian Hodge correspondence: the Betti and the Dolbeault moduli spaces. 

\subsection{Character varieties of abelian groups}

Let $G$ be a connected complex reductive group and consider a finitely generated
group $\Gamma$. The \emph{representation variety} is the set of group
homomorphisms 
\[
R_{\Gamma}(G)=\Hom(\Gamma,G)\; .
\]
After considering a presentation of $\Gamma$ with $n$ generators, it inherits the structure of an algebraic subset of $G^{n}$ (whose algebraic structure is independent of the presentation). The
group $G$ acts on $R_{\Gamma}(G)$ by conjugation as $(h\cdot\rho)(\gamma)=h\rho(\gamma)h^{-1}$
for $h\in G$ and $\gamma\in\Gamma$, and the (affine) Geometric Invariant
Theory (GIT) quotient 
\[
\mathfrak{X}_{\Gamma}(G)=R_\Gamma(G)\sslash G
\]
is known as the $G$-\emph{character variety} of $\Gamma$ (also called the $G$-\emph{Betti moduli space of $\Gamma$}, in case $\Gamma$ is the fundamental group of a K\"ahler manifold).

For the rest of this paper, we will focus on the case in which $\Gamma=\ZZ^{r}$
is the free abelian group of rank $r\in\NN$.
We will use the simplified notation $\mathfrak{X}_{r}(G):=\mathfrak{X}_{\ZZ^{r}}(G)$ for the associated character
variety, and call it \emph{the abelian character variety of rank $r$}. There is a natural regular injective morphism 
\[
\psi: T^{r}/W \longrightarrow \mathfrak{X}_r(G)\; ,
\]
which is a kind of Chevalley map, where $T\subseteq G$ is a maximal torus and $W=N_{G}(T)/T$ is the
Weyl group acting on $T$ by conjugation and on $T^{r}$ diagonally (see \cite{Si,GPZ}). In principle, depending on the group $G$, the character variety $\mathfrak{X}_r(G)$ may have several connected components, and the morphism $\psi$ is bijective onto the connected component of the identity representation, denoted by $\mathfrak{X}_r^\circ(G) \subseteq \mathfrak{X}_r(G)$.

This morphism $\psi: T^{r}/W \to \mathfrak{X}_r^\circ(G)$ is an isomorphism for $r=1,2$ and every (connected) reductive group (see \cite{Si} and \cite[Theorem 6.3.1]{LNY24}), and for all $r$ and some classical groups, such as for $G = \GL_n, \SL_n, \Sp_n$ and $\SO_n$. However, for general $G$ and $r>2$ it is only a bijective normalization map \cite[Theorem 2.1]{Si}. In particular, $\psi$ is proper and thus is a homeomorphism, so $T^{r}/W$ and $\mathfrak{X}_r^\circ(G)$ have the same cohomology. Furthermore, since $\psi$ is algebraic, it induces an isomorphism of mixed Hodge structures (see also \cite[Section 5.1]{FS}), as well as an equality of virtual classes \cite{GPZ}. In general, let us denote the normalization of $\mathfrak{X}_r^\circ(G) \subseteq \mathfrak{X}_r(G)$ by $\hat{\mathfrak{X}}_r^\circ(G) = T^r/W$.

\begin{rem}
Using the arguments of \cite[Corollary 4.16]{GPZ}, the description
of $\mathfrak{X}_r^\circ(G)$ can also be extended for non-connected
groups $G$, yielding an isomorphism of MHSs and of virtual classes 
\[
Z_{G}(T)^{r}/W \to \mathfrak{X}_r^\circ(G)\; ,
\]
where $Z_{G}(T)$ is the centralizer of the maximal torus $T$ and
now we take $W=N_{G}(T)/Z_{G}(T)$. Notice that, when $G$ is connected,
$Z_{G}(T)=T$, recovering the previous formula. 
\end{rem}

\subsection{Higgs bundles on abelian varieties}

Let $A$ be a complex abelian variety of complex dimension $d$ and
denote by $0$ the origin of $A$. In this situation, the (holomorphic) cotangent bundle $\Omega^1_A = T^*A$
is a trivial vector bundle on $A$.
In this setting, a principal $G$-Higgs bundle on $A$ is a pair $(E,\varphi)$, where $E$ is a principal $G$-bundle over $A$, and $\varphi$ is a Higgs field
\[
\varphi\in H^{0}(A,E(\mathfrak{g})\otimes \Omega^1_A)\; ,
\]
where $E(\mathfrak{g})$ is the adjoint bundle (i.e.\ the vector bundle
on $A$ with fiber the Lie algebra $\mathfrak{g}=\textup{Lie}(G)$ given by
the adjoint action of $G$). Furthermore, if we consider the homomorphism 
\[
\eta:(E(\mathfrak{g})\otimes \Omega^1_A)^{\otimes2}\longrightarrow E(\mathfrak{g})\otimes \Omega^{2}_A
\]
given by the composition of the Lie bracket on $E(\mathfrak{g})$
with the exterior product of $1$-forms, the Higgs field must also satisfy $\eta(\varphi\otimes\varphi)=0$.

By choosing an ample line bundle on $A$, it is possible to define a degree for
coherent sheaves on $A$, leading to the usual definition of semistable $G$-Higgs bundles
\cite{ramanathan1975moduli, simpson1994moduli1}. Then, there exists a moduli space $\mathcal{M}_{A}(G)$ of
\emph{semistable $G$-Higgs bundles on $A$} with trivial topological type \cite{simpson1994moduli2, simpson1995moduli2}, also known as the \emph{Dolbeault moduli space of $A$}. Since $\pi_{1}(A)=\mathbb{Z}^{2d}$, by the non-abelian
Hodge theory \cite{corlette1988flat, hitchin1987self, simpson1992higgs} there exist homeomorphisms between the connected components of the identity in the character
varieties $\mathfrak{X}_{2d}^\circ(G)$ and the moduli spaces of Higgs
bundles $\mathcal{M}_{A}(G)$.

As in the case of character varieties, in this abelian setting, it is possible to give a more explicit description of the moduli space $\mathcal{M}_A(G)$, as described in \cite{florentino2013commuting, franco2014higgs, franco2019higgs, BFN}. We briefly sketch this identification here for completeness. First, analogously to the case of character varieties, we have a reduction to the maximal torus via a regular bijective normalization map
$$
    \psi: \mathcal{M}_A(T)/W \to \mathcal{M}_A(G)\; ,
$$
where $T \subseteq G$ is the maximal torus, with the induced action of the Weyl group $W$. As in the case of character varieties, the morphism $\psi$ induces an equality of Hodge structures and virtual classes.

Now, recall that a topologically trivial principal $T$-bundle on $A$ is the same as a group homomorphism $\Lambda \to A^\vee$, where $A^\vee = \Pic^0(A)$ is the group of line bundles on $A$ with vanishing first Chern class and $\Lambda$ is the character lattice of $T$ (equivalently, of $G$). Indeed, given a principal $T$-bundle $P \to A$, for any character $\lambda: T \to \CC^* = \GL_1(\CC)$ in $\Lambda$, we can define the associated line bundle
$$
    P_\lambda := P \times_{\lambda} \CC = (P \times \CC)/T\; ,
$$
with $T$ acting by $t \cdot (p, z) = (t^{-1} \cdot p, \lambda(t)z)$ for $t \in T$, $p \in P$ and $z \in \CC$. Additionally, it is straightforward to check that $P_{\lambda_1 + \lambda_2} = P_{\lambda_1} \otimes P_{\lambda_2}$, so any such principal $T$-bundle $P$ defines a group homomorphism $\Lambda \to A^\vee$, $\lambda \mapsto P_\lambda$. Reciprocally, any such family of line bundles induced by a homomorphism $\Lambda \to A^\vee$ determines a topologically trivial principal bundle.

Regarding the Higgs field $\varphi$, it takes values in the bundle $E(\mathfrak{t}) \otimes \Omega^1_A$, with $\mathfrak{t}$ the Lie algebra of $T$. In particular, since $\mathfrak{t}$ is abelian, we have that the associated bundle is trivial $E(\mathfrak{t}) = A \times \mathfrak{t}$, and therefore
$$
    \varphi \in H^0(A, E(\mathfrak{t}) \otimes \Omega^1_A) = \mathfrak{t} \otimes H^0(A, \Omega^1_A) = \Hom_{\Ab}(\Lambda, H^0(A, \Omega^1_A))\; .
$$
In the last equality, we have used that $\mathfrak{t} \cong \Hom_\Ab(\Lambda, \CC)$ via the natural perfect pairing $\mathfrak{t} \otimes_{\ZZ} \Lambda \to \CC$ given by $(t, \lambda) \mapsto d_1\lambda(t)$, where $d_1\lambda: \mathfrak{t} = T_1T \to \CC = T_1 \mathbb{C}^*$ is the derivative of the character $\lambda: T \to \mathbb{C}^*$. Observe that, in this setting, it is automatic that $\eta(\varphi \otimes \varphi) = 0$ since $T$ is abelian.

Summarizing, the previous discussion implies that
$$
    \mathcal{M}_A(T) = \Hom_{\Ab}(\Lambda, A^\vee \times H^0(A, \Omega^1_A))\; ,
$$
where the first factor keeps track of the principal bundle, whereas the latter encodes its Higgs field.

Therefore, we have that the normalization of $\mathcal{M}_A(G)$ is
\begin{equation}\label{eq:Higgs}
    \hat{\mathcal{M}}_A(G) = \Hom_{\Ab}(\Lambda, A^\vee \times H^0(A, \Omega^1_A)) / W\; ,
\end{equation}
where $W$ acts on $\mathcal{M}_A(T)$ through its natural action on $\Lambda$, i.e. $(w \cdot \chi)(\lambda) = \chi(w^{-1} \cdot \lambda)$ for $w \in W$, $\chi: \Lambda \to A^\vee \times  H^0(A, \Omega^1_A)$ and $\lambda \in \Lambda$. In particular, since $\dim A^\vee = \dim H^0(A, \Omega^1_A) = \dim A = d$, we have $\dim \mathcal{M}_A(G) = 2d\rk G$, where $\rk G = \rk \Lambda$ is the rank of $G$.

\begin{rem}
Although the stringy invariants in this article are considered for character varieties and moduli spaces of Higgs bundles associated with abelian varieties, all constructions depend purely on the local quotient singularity structure and the action of the Weyl group $W$ on the character lattice. Therefore, all results and formulae apply verbatim when replacing the abelian variety $A$ with a complex torus $\mathbb{C}^d/\mathbb{Z}^{2d}$, not  necessarily satisfying Riemann's positivity conditions.
\end{rem}

\section{Symplectic resolutions}\label{sec:symplectic-resolution}

Let $X$ be a normal and irreducible complex algebraic variety. We
say that $X$ is a \emph{symplectic variety} if its open smooth locus
$X^{\textrm{sm}} \subseteq X$ admits a holomorphic symplectic form $\omega$ (i.e.\ closed and
nondegenerate) whose pullback to any resolution $\pi: Z \to X$ extends
to a holomorphic 2-form $\tilde{\omega}$ on $Z$ (not necessarily
non-degenerate). In this setting, an automorphism $f: X \to X$ is said to be \emph{symplectic} if $f^*\omega = \omega$.

A resolution of singularities $\pi: Z\to X$ of a symplectic variety
$X$ is then called a {\em symplectic resolution} if $\pi$ is projective and the above extended
form $\tilde{\omega}$ on $Z$ is non-degenerate (so that it is a
{\em de facto} symplectic form on $Z$). By \cite[Proposition 1.1]{Fu},
a resolution $\pi: Z \to X$ of a symplectic variety is a symplectic resolution if and
only if it is a projective crepant resolution, i.e.\ if $\pi^{\ast}K_{X}=K_{Z}$,
where $K_{X}$ and $K_{Z}$ are the canonical bundles on $X$ and
$Z$, respectively.

A standard example of a singular symplectic variety arises as the quotient $X/\Gamma$ of a smooth symplectic variety $X$ under the action of a finite group $\Gamma$ by symplectic automorphisms. The following result is scattered in the literature,
so we provide a proof here for convenience. 
\begin{thm}
\label{thm:crepant-resolution}
Suppose a finite group $\Gamma$ acts effectively
by symplectic automorphisms on a smooth symplectic variety $X$. For $g \in \Gamma$, let $X^g \subseteq X$ be the locus of fixed points of $g$, and consider the union
$$
    X_0 = \bigcup_{g \neq 1} X^g \subseteq X\; .
$$
If $X_0 \neq \emptyset$ and the quotient $X/\Gamma$ admits a symplectic resolution, then $X_0$ has codimension 2 in $X$. 
\end{thm}

\begin{proof}
Let $\pi:Z\to X/\Gamma$ be a resolution of singularities of $X/\Gamma$, and consider $E\subseteq Z$ the exceptional locus of the resolution.
By definition $Z\setminus E$ is the largest open set where we have
a local isomorphism
\[
\pi|_{Z\setminus E}:Z\setminus E\to(X/\Gamma)\setminus\pi(E)\; .
\]
Since $X$ is smooth, the quotient $X/\Gamma$ is normal and $\mathbb{Q}$-factorial.
Hence, by van der Waerden's theorem (see \cite[Thm 21.12.12]{EGA}),
$E$ has pure codimension $1$ in $Z$.

Now, if $\pi$ is also symplectic, then it is a semismall morphism
(see \cite[Lemma 2.11]{kaledin2006symplectic}), that is, 
\[
2=2\,\text{codim}\,E\geq\text{codim}\,\pi(E)\; .
\]
Furthermore, since $X/\Gamma$ is normal, its singular locus $\pi(E)$
has codimension $\geq2$. This implies that $\textrm{codim}\,\pi(E)=2$.

Now, consider the projection
$p:X\to X/\Gamma$. Notice that any singular point of $X/\Gamma$
appears at a fixed point, which implies that $\pi(E)\subseteq p(X_{0})$,
and thus $\textrm{codim}\,p(X_{0})\leq\textrm{codim}\,\pi(E)=2$.
As $p$ is finite, we also have $\textrm{codim}\,X_{0}\leq2$. Finally,
since $\Gamma$ acts by symplectic automorphisms, each irreducible
component of $X^{g}$ is a symplectic subvariety, and thus it has
even codimension. This implies that the codimension of $X_0$ is $2$.
\end{proof}

\begin{rem}
The proof of Theorem \ref{thm:crepant-resolution} goes along the same lines as in \cite{Fu} and in \cite{BFN},
where it was stated only for $G=\GL_{n}$.
Note also that this corrects a typo in the loc.\ cit.\ since, in principle, for $g\neq h\in \Gamma$, the (co)dimensions of $X^g$ and $X^h$ might be different, as the following lemma indicates (see also Section \ref{sec:examples}).
\end{rem}

\begin{lem}
Let $V$ be a finite dimensional complex symplectic vector space, and $\Gamma \subseteq \textup{Sp}(V)$ a finite subgroup. For $g,h\in \Gamma$, we have
\[
V^{g\cdot h} \cap (V^g+V^h) = V^g \cap V^h.
\]
\end{lem}
\begin{proof}
This is a slightly more general statement than \cite[Lemma A.16]{GK04}, but the proof is the same.
\end{proof}

\subsection{Symplectic resolutions in abelian character varieties} \label{ssec:symplectic_resolutions}
The discussion of the previous section has a direct consequence in the geometry of character varieties of abelian groups. Given a complex connected reductive group $G$ with maximal torus $T$, recall that $T^2$ admits a natural symplectic structure given by left translation of the natural symplectic form on the cotangent bundle $\mathfrak{t}^2 \cong \mathfrak{t} \oplus \mathfrak{t}^* = T^*\mathfrak{t}$. Furthermore, since the Weyl group $W$ of $G$ is generated by reflections acting on $\mathfrak{t}$, its action on $T^{2d}$ is symplectic for all $d \geq 1$.

In this manner, the normalization of the connected component of the identity $\hat{\mathfrak{X}}_{2d}(G) = T^{2d}/W$ of the character variety $\mathfrak{X}_{2d}(G)$ of representations of $\ZZ^{2d}$ into $G$ has naturally a (singular) symplectic variety structure for any $d \geq 1$. It thus makes sense to consider which of these varieties may admit a symplectic resolution. In this direction, we get the following result. 

\begin{cor}\label{cor:existence-crepant}
For any connected reductive group $G$ with positive semisimple rank, the symplectic variety $T^{2d}/W$ does not admit a symplectic resolution for $d \geq 2$.
\end{cor}

\begin{proof}
Notice that, since $G$ has positive semisimple rank, its Weyl group $W$ is non-trivial. 
Recall that $T = \Hom_{\Ab}(\Lambda, \CC^*)$, where $\Lambda$ is the character lattice of $G$, and the action of $W$ on $T$ is the one induced by $W$ acting on $\Lambda$, so in particular $\dim T = \rk \Lambda$. Furthermore, for any $w \in W$, its fixed locus is given by $T^w = \Hom_{\Ab}(\coker(I_\Lambda - w), \CC^*)$, where $I_\Lambda$ is the identity map on $\Lambda$ (see Section \ref{section:stringy-CV}). In particular, for $w$ a simple reflection, we have that $I_{\Lambda} - w$ has rank $1$, and thus $\dim T^w = \rk \Lambda - 1$.

Applying this to $T^{2d} = \Hom_{\Ab}(\Lambda^{\oplus 2d}, \CC^*)$, we get that for a reflection $w$ we have $\dim (T^{2d})^w = 2d \rk \Lambda - 2d$, and therefore $\codim (T^{2d})^w = 2d$. Furthermore, since any $w \in W$ with $w \neq 1$ is a non-trivial product of reflections, $\rk \image(I_\Lambda - w) \geq 1$, so $\dim (T^{2d})^w \leq 2d \rk \Lambda - 2d$ and thus $\codim (T^{2d})^w \geq 2d$. Therefore, we get that
$$
	\codim \bigcup_{w \neq 1} (T^{2d})^w = 2d\; .
$$
Hence, by Theorem \ref{thm:crepant-resolution}, $T^{2d}/W$ can only admit a symplectic resolution if $d=1$.
\end{proof}

\subsection{Non-existence of projective symplectic resolutions in abelian character varieties}

By Corollary \ref{cor:existence-crepant}, the only abelian character varieties that may admit a symplectic resolution are the varieties $\hat{\mathfrak{X}}_{2}(G) = T^2/W$, corresponding to representations of the fundamental group of an elliptic curve into $G$. However, such a result does not imply that these resolutions exist in general.

In the case that $G = \GL_n$, it was proven that the Hilbert scheme provides a projective symplectic resolution (see \cite[Theorem 13]{BFN}). More precisely, if $\Hilb^n(X)$ denotes the Hilbert scheme of $n$-points of $X$, then the Hilbert-Chow morphism
$$
	\Hilb^n((\CC^*)^{2}) \to \textup{Sym}^n((\CC^*)^{2})
$$
is projective and provides the desired resolution. Indeed, the maximal torus of $\GL_n$ is $T = (\CC^*)^n$ and the Weyl group is the symmetric group $S_n$ acting by permutation, so $\hat{\mathfrak{X}}_2(\GL_n) = ((\CC^*)^2)^n / S_n = \textup{Sym}^n((\CC^*)^2)$. From this description, we also get that such a projective symplectic resolution exists for $G=\SL_n$, since it corresponds to the pullback of $\Hilb^n((\CC^*)^{2}) \to \textup{Sym}^n((\CC^*)^{2})$ under the inclusion $T_{\SL_n}^2/S_n \hookrightarrow T_{\GL_n}^2/S_n = \textup{Sym}^n((\CC^*)^{2})$ of the maximal torus $T_{\SL_n}$ of $\SL_n$ inside the maximal torus of $\GL_n$.

Similar results can be obtained for other groups of type $B$ and $C$, such as $G = \Sp_{2n}$ whose Weyl group is $S_n \ltimes (\ZZ_2)^n$, as in \cite{BDL}. Explicitly, a symplectic resolution of $T^2/(S_n \ltimes (\ZZ_2)^n)$ is constructed by resolving the singularity $Z \to (\CC^*)^2/\ZZ_2$, and then considering the $n$-points Hilbert scheme of $Z$. However, in other cases, such a symplectic resolution may not exist, such as for $G = \PGL_3$ (see \cite[Proposition 4.5]{Th}). 

The aim of this section is to prove that these symplectic resolutions cannot exist for irreducible groups of type $D$ or exceptional groups. Concretely, we will prove the following.

\begin{thm}
\label{thm:resolution_ABC}
Let $G$ be a connected reductive group with irreducible root system, maximal torus $T$ and Weyl group $W$. If the symplectic variety $T^2/W$ admits a symplectic resolution, then $G$ has a root system of Dynkin type $A$, $B$ or $C$.
\end{thm}

\begin{proof}
Taking into account that the center component is acted trivially, without loss of generality we can suppose that $G$ is semisimple. By a result of Ginzburg and Kaledin \cite[Theorem 1.1]{GK04}, the quotient $\mathfrak{t}^2 /W$ admits a symplectic resolution only if $G$ has type $A$, $B$ or $C$, where $\mathfrak{t}$ is the Lie algebra of $T$ and $W$ acts on $\mathfrak{t}^2 = \mathfrak{t} \oplus \mathfrak{t}$ diagonally. We will thus prove that if $T^2/W$ admits a symplectic resolution, then $\mathfrak{t}^2/W$ also admits such a resolution, leading to the desired result.

In order to do that suppose that $Z \to T^2/W$ is a symplectic resolution. The exponential map of the abelian group $T^2$, $\exp: \mathfrak{t}^2 \to T^2$, is a $W$-equivariant analytic local isomorphism. In particular, it defines a $W$-equivariant isomorphism $\widehat{\exp}: \widehat{\mathfrak{t}^2_{0}} \to \widehat{T^2_{1}}$ between the formal completions $\widehat{\mathfrak{t}^2_{0}}$ of $\mathfrak{t}^2$ at $0 \in \mathfrak{t}^2$ and $\widehat{T^2_{1}}$ of $T^2$ at $1 \in T^2$. Since $0$ and $1$ are fixed points of the action of $W$, this isomorphism descends to an isomorphism
$$
\widehat{\exp}_W: \widehat{(\mathfrak{t}^2/W)}_{[0]} \to \widehat{(T^2/W)}_{[1]}
$$
between the corresponding formal completions of $\mathfrak{t}^2/W$ and $T^2/W$.

In particular, the restriction of $Z \to T^2/W$ to the formal completion gives a symplectic formal resolution
$$
    Z|_{\widehat{(T^2/W)}_{[1]}} \to \widehat{(T^2/W)}_{[1]}\; ,
$$
and its pullback $\widehat{\exp}_W^*(Z|_{\widehat{(T^2/W)}_{[1]}}) \to \widehat{(\mathfrak{t}^2/W)}_{[0]}$ is a symplectic resolution of $\widehat{(\mathfrak{t}^2/W)}_{[0]}$. By \cite[Theorem 1.4]{kaledin2003crepant}, this local symplectic resolution can be extended to a global symplectic resolution of $\mathfrak{t}^2/W$, as we wanted to prove.
\end{proof}

\begin{rem}
Under the additional assumption that $G$ is an almost simple group, Shu has recently obtained a complete classification of the connected components of the $G$-character variety of an elliptic curve that admit a symplectic resolution \cite[Theorem C]{shu2024singularities}. In particular, for the identity component $T^2/W$, his results show that the necessary restriction to Dynkin types $A$, $B$, and $C$ established in Theorem \ref{thm:resolution_ABC} can be sharpened considerably: existence of a symplectic resolution depends also on the isogeny type of $G$ (\cite[Propositions 7.9, 7.11, and 7.13]{shu2024singularities}). Shu's analysis of the genus one case uses Borel--Friedman--Morgan theory to describe the connected components and the corresponding Weyl group actions. By contrast, our argument relies on the Ginzburg--Kaledin obstruction and yields a uniform necessary condition for every reductive group, without any almost-simplicity assumption.
\end{rem}

\section{Stringy mixed Hodge polynomials of character varieties of abelian groups}\label{sec:stringy-mhpoly}

In this section, we prove a formula for the stringy mixed Hodge polynomial of a quotient of the form $\hat{\mathfrak{X}}_r(G) = T^r/W$, where $r\in {\mathbb N}$, $T$ is a maximal torus of a connected reductive group $G$, and $W$ is the Weyl group. 
As mentioned before, this corresponds to the normalization of the connected component of the identity of the character variety $\mathfrak{X}_{r}(G) = R_{\ZZ^r}(G)\sslash G$.

\subsection{Mixed Hodge structures on quotients and the fermionic shift}\label{sec:mhs-quotients}

We start by discussing how to compute the mixed Hodge polynomial of the quotient $X/F$ of a complex variety $X$ acted by a finite group $F$, in terms of the characters of the action. Formulae of this type, which are complex analogues of the well-known Burnside formula for point-counting, are scattered in the literature (see for instance \cite{batyrev1996strong}), and we include its proof here for completeness.

\begin{prop}[Burnside formula for mixed Hodge polynomials]\label{prop:Burnside}
Let $X$ be a complex algebraic variety and let $F$ be a finite group acting algebraically on $X$. Then, the mixed Hodge polynomial of $X/F$ is
$$
    \mu(X/F) = \frac{1}{|F|}\sum_{g \in F}\sum_{k,p,q} \chi_{k,p,q}(g) \,t^ku^pv^q\; ,
$$
where $\chi_{k,p,q}(g) =\tr(g: H^{k,p,q}(X) \to H^{k,p,q}(X))$ is the character of the action of $g \in F$ on $H^{k,p,q}(X)$, i.e.\ the trace of the $\CC$-linear map induced in cohomology by $g$.
\end{prop}

\begin{proof}
In every degree, the cohomology of $X/F$ is the $F$-fixed part $H^k(X)^F = H^k(X/F)$.
In particular, since $F$ preserves the bigrading of the mixed Hodge structure, its $(p,q)$-pieces are
$$
    H^{k,p,q}(X/F) = H^{k,p,q}(X)^F.
$$
Recall that if $H$ is any $F$-module, then
$$
    \dim H^F = \frac{1}{|F|} \sum_{g \in F} \chi_H(g)\; ,
$$
where $\chi_H$ is the character associated to $H$. This follows directly from decomposing $H$ into its irreducible components $H = H_0^{\oplus d_0} \oplus \ldots \oplus H_n^{\oplus d_n}$, with $H_0$ being the trivial representation. Now, considering the irreducible characters $\chi_0, \ldots, \chi_n$ of $H_0, \ldots, H_n$, respectively, and using the orthogonality of characters we get
$$
    \dim H^F = d_0 = \left\langle \sum_i d_i \chi_i, \chi_0 \right\rangle = \left\langle \chi_H, \chi_0 \right\rangle = \frac{1}{|F|} \sum_{g\in F} \chi_H(g)\overline{\chi}_0(g) = \frac{1}{|F|} \sum_{g\in F} \chi_H(g)\; .
$$

Therefore, applying this to $H = H^{k,p,q}(X)$, we get
$$
    \mu(X/F) = \sum_{k,p,q} \dim H^{k,p,q}(X)^F\,t^ku^pv^q = \frac{1}{|F|} \sum_{k,p,q} \sum_{g \in F} \chi_{k,p,q}(g)\,\,t^ku^pv^q\; ,
$$
as we wanted to prove.
\end{proof}

We recall that the calculation of the character $\chi$ for an $F$-module $M$ can be easily simplified in the following scenarios:
\begin{enumerate}
    \item\label{item:disjoint-union} Let $H = H_1 \oplus \ldots \oplus H_n$, and let $F$ acting on $H$, not necessarily diagonally. Then, for any $g \in F$ and $1 \leq i \leq n$, consider the ``diagonal components''
    $$
        g_i: H_i \to H/\left({\displaystyle \oplus_{j\neq i} H_j}\right) \cong H_i\; .
    $$
    Then, since the off-diagonal entries do not contribute to the trace, we have that $\chi_H(g) = \sum_i \tr(g_i: H_i \to H_i)$. In particular, let $X_1, \ldots, X_n$ be connected complex varieties and consider $X = X_1 \sqcup \ldots \sqcup X_n$, so that $H^\bullet(X) = H^\bullet(X_1) \oplus \ldots \oplus H^\bullet(X_n)$ and any $F$-action on $X$ induces an $F$-action on cohomology as above. Furthermore, the action of $g \in F$ sends connected components into connected components, so we have
    $$
        \chi_{k,p,q}(g) = \sum_{\{i \,\mid\, g(X_i) = X_i\}} \tr\left(g_i: H^{k,p,q}(X_i) \to H^{k,p,q}(X_i)\right),
    $$
    where the sum runs over the indices $i$ such that $g$ preserves the $i$-th component $X_i$.
    \item\label{item:determinant} Consider the exterior algebra $\bigwedge^\bullet H$ associated to an $F$-module $H$, with the natural action. Then, for any $g \in F$, the generating function of the characters of $\bigwedge^\bullet H$ can be identified as the determinant $\sum_{k\geq 0} \chi_{\bigwedge^k H}(g)x^k = \det(I_H+xg)$, where $I_H: H \to H$ is the identity map and $g: H \to H$ is the action of $g$. Furthermore, if $H$ is equipped with a mixed Hodge structure then, for the natural mixed Hodge structure induced on $\bigwedge^\bullet H$, we have
    $$
        \sum_{k, p, q} \chi_{(\bigwedge^k H)^{p,q}}(g)\,t^ku^p v^q = \prod_{p,q} \det(I_{H^{p,q}}+t u^p v^q g_{p,q})\; ,
    $$
    where $g_{p,q}: H^{p,q} \to H^{p,q}$ is the action of $g \in F$ in the $(p,q)$-piece of $H$ and the product runs over all the pairs $(p,q)$ such that $H^{p,q} \neq 0$. In particular, if $X$ is an algebraic variety whose cohomology (and its MHS) is an exterior algebra generated by $H^1(X)$, $H^\bullet(X) = \bigwedge^\bullet H^1(X)$ (for example, when $X$ is an abelian variety), then (see also \cite{FS}):
    $$
    \mu(X/F) = \frac{1}{|F|}\sum_{g \in F}\sum_{k,p,q} \chi_{k,p,q}(g) \,t^ku^pv^q = \frac{1}{|F|}\sum_{g \in F} \prod_{p,q}\det(I_{H^{1, p, q}(X)} + tu^pv^q\,g_{p,q})\; ,
    $$
    where again $g_{p,q}: H^{1, p, q}(X) \to H^{1, p, q}(X)$ is the induced action on the $(p,q)$-piece. 
    \item\label{item:product} If $H = H_0^{\otimes r}$ with the natural tensor product action, then $\chi_H(g) = \chi_{H_0}(g)^r$. In particular, if $X_0$ is a complex variety whose cohomology and mixed Hodge structure are given as the exterior algebra of $H^1(X_0)$, and we consider $X = X_0^r$ with the diagonal action, then
    $$
    \mu(X_0^r/F) = \frac{1}{|F|}\sum_{g \in F}\left(\sum_{k,p,q} \chi_{k,p,q}(g) \,t^ku^pv^q\right)^r = \frac{1}{|F|}\sum_{g \in F} \prod_{p,q}\det(I_{H^{1, p, q}(X_0)} + tu^pv^q\,g_{p,q})^r,
    $$
    with $g_{p,q}: H^{1,p,q}(X_0) \to H^{1,p,q}(X_0)$, and $\chi_{k,p,q}(g)=\chi_{H^{k,p,q}(X_0)}(g)$.
\end{enumerate}

\subsubsection{Fermionic shift}
\label{subsec:age-shift}
Again, let $F$ be a finite group acting on a complex variety $X$. Recall from Definition \ref{def:CR-cohomology} that the Chen-Ruan cohomology of the orbifold $X/F$ is given by:
\begin{equation}\label{eq:CR1}
H^\bullet_{\text{CR}}([X/F]) := \bigoplus_{[g] \in \text{Conj}(F)} H^{\bullet}(X^g / C(g))[2 \, \age(g)]\; .
\end{equation}
For each $g\in F$, the corresponding fermionic shift $\age(g)$ is defined as follows. Suppose that $g$ acts on the complex tangent bundle of $X$ with eigenvalues $\lambda_j = e^{2\pi i q_j}$, satisfying $0 \le q_j < 1$, then we define $$\age(g) = \sum_{j} q_j\; .$$ 
This number is constant on each connected component of $X^{g}$, making it a well-defined topological grading.
In general, $\age(g)$ is a half-integer and, when the singularities of $X/F$ are Gorenstein, it can be shown that $\age(g)$ is a non-negative integer, so we get an 
upward shift in the degree of the cohomology classes of that sector in Chen--Ruan cohomology by exactly $2\,\age(g)$. 

\begin{rem}\label{rem:first-term}
Let us look at the term corresponding to $g = 1$ in (\ref{eq:CR1}). In this case, $C(g) = F$ and the fermionic shift of the identity is zero. In this manner, this term actually corresponds to the usual cohomology $H^\bullet(X/F)$ of the quotient $X/F$ and, in this sense, the Chen--Ruan cohomology can be understood as a ``fermionic correction'' to this regular cohomology. 
\end{rem}

It is worth pointing out that, in the case that the action is symplectic, the computation of the fermionic shift is much simpler. Recall that a complex variety $X$ is \emph{complex symplectic} if it admits a closed non-degenerate holomorphic $2$-form $\omega$. The action of $F$ on $X$ is said to be symplectic if the action preserves $\omega$. 

\begin{lem}\label{lem:fermionic-shift}
Let a finite group $F$ act symplectically and diagonally on the $r$-fold cartesian product $X^{r}$ for some $r \geq 1$ (we do not require the action to be symplectic on $X$). Then, the fermionic shift of $g \in F$ on $X^{r}$ satisfies 
$$\age_{X^{r}}(g) = \frac r2\codim_\CC X^g.$$
\end{lem}

\begin{proof} If $F$ acts symplectically on a complex symplectic manifold $Y$, the linearization of the action of $g \in F$ on its fixed point locus is a symplectic matrix. Hence, if $\lambda \in \CC$ is an eigenvalue of the linearization of $g$, then so is $\lambda^{-1}$. In particular, if 
$\lambda = e^{2\pi i q} \neq 1$, then $0 < q < 1$ thus $\lambda^{-1} = e^{-2 \pi i q} = e^{2\pi i (1-q)}$, where now $0 < 1-q < 1$; so this pair of eigenvalues contributes to the fermionic shift by $q + (1-q) = 1$. However, this symmetry does not happen when $\lambda = 1$ since in this case the phase of both $\lambda$ and $\lambda^{-1}$ is $q = 0$, so this pair does not contribute to the shift. The number of eigenvalues $\lambda = 1$ is equal to the dimension of the fixed locus $Y^g$, so we have
\begin{equation}\label{eq:age}
    \age_Y(g) = \frac{1}{2}\codim_\CC Y^g.
\end{equation}

Finally, if $Y=X^r$ and the action is diagonal, we have that the eigenvalues of the action of $g$ on $Y$ are the same as those of the action on $X$, but repeated $r$ times. Therefore, we get $\age_Y(g) = r\, \age_X(g)$, finishing the proof.
\end{proof}

\subsection{Stringy mixed Hodge polynomials of $T^r/W$}\label{section:stringy-CV}

We are now ready to state and prove a formula for the stringy mixed Hodge polynomial of the quotient $T^r/W$, which is the normalization of the connected component of the identity of the character variety $\mathfrak{X}_{r}(G) = R_{\ZZ_r}(G)\sslash G$.

For this purpose, let us fix some notation. Let $G$ be a connected reductive group over $\CC$ and fix a maximal torus $T \subseteq G$, with associated character lattice $\Lambda$ and Weyl group $W$. Scheme-theoretically, recall that
$T = \Spec\CC[\Lambda]$, where $\CC[\Lambda]$ is the $\CC$-algebra generated by $\Lambda$. In particular, its complex closed points are
$$
    T(\CC) = \Hom_{\Sch}(\Spec \CC, \Spec \CC[\Lambda]) = \Hom_{\CC\text{-Alg}}(\CC[\Lambda], \CC) = \Hom_{\Ab}(\Lambda, \CC^*)\; .
$$

With this description, the action of $W$ on $T(\CC)$ is the induced one from the action of $W$ on the character lattice $\Lambda$. Explicitly, given an element $\chi \in T(\CC)$, corresponding to an abelian group homomorphism $\chi: \Lambda \to \CC^*$, and $w \in W$, we have $(w \cdot \chi)(\lambda) = \chi(w^{-1}\cdot \lambda)$ for all $\lambda \in \Lambda$. This implies that $w \cdot \chi = \chi$ if and only if $\chi(\lambda - w\cdot \lambda) = 1$ for all $\lambda \in \Lambda$, or equivalently, if $\chi(\image(I_{\Lambda} - w)) = 1$, where $I_{\Lambda}: \Lambda \to \Lambda$ is the identity and $w: \Lambda \to \Lambda$ is the homomorphism induced by the element $w$. In other words, $w \cdot \chi = \chi$ if and only if $\chi$ factors through the cokernel of the homomorphism $I_\Lambda -w: \Lambda \to \Lambda$, as
$$
\xymatrix{
    \Lambda \ar[r]^{\chi}\ar[d] & \CC^* \\
    \coker(I_\Lambda - w) \ar@{--{>}}[ru] &
}
$$
Denoting $A_w := \image(I_\Lambda - w)$ and $K_w := \coker(I_\Lambda - w) = \Lambda/A_w$, which is a finitely generated group, we get that the fixed locus of $T(\CC)$ under the action of $w$ is
$$
    T(\CC)^w = \Hom_{\Ab}(K_w, \CC^*)\; .
$$
In particular, the codimension of $T(\CC)^w$ in $T(\CC)$ is $\codim T(\CC)^w = \rk A_w$, and recall that $\frac{1}{2}\rk A_w = \age(w)$, by (\ref{eq:age}). The finitely generated abelian group $K_w$ decomposes as 
\begin{equation}\label{eq:free-torsion}
    K_w \cong D_w \oplus \Phi_w\; ,
\end{equation}
with $D_w \subseteq K_w$ the torsion subgroup, and $\Phi_w \cong \ZZ^{s_w}$ a free abelian group of rank $s_w:=\rk K_w$. In particular, $T(\CC)^w = \Hom_{\Ab}(D_w \oplus \Phi_w, \CC^*) = \Hom_{\Ab}(D_w, \CC^*) \times (\CC^*)^{s_w}$ is the disjoint union of $|D_w|$ tori of rank $s_w$.
Henceforth, we will make use of the relation
\begin{equation}
    \label{eq:ranks_lattices}
    \rk \Lambda = \rk A_w+\rk \Phi_w=2\age(w)+s_w\; .
\end{equation}

\begin{rem}
    The term $\Hom_{\Ab}(D_w, \CC^*)$ can actually be described in very simple terms. In fact, if $D_w = \ZZ/(k_1\ZZ) \times \cdots \times \ZZ/(k_l\ZZ)$, with $k_1, \ldots, k_l$ the invariant factors of $D_w$, then $\Hom_{\Ab}(D_w, \CC^*) = \mu_{k_1} \times \cdots \times \mu_{k_l}$, where $\mu_n \subseteq \CC^*$ is the set of $n$-th roots of unity. However, we will not need this description in the following.
\end{rem}

Now, observe that for any other $g \in W$, we have $g \cdot T(\CC)^w = T(\CC)^{gwg^{-1}}$. In particular, if $g$ belongs to the centralizer $C(w)$ of $w$, then $g \cdot T(\CC)^w = T(\CC)^w$. Therefore, for any $g \in C(w)$, we have that the action on the character lattice $g: \Lambda \to \Lambda$ descends to an automorphism on the cokernel
$$
    g_w: K_w \to K_w\; .
$$
From this morphism, we can obtain the `torsion' and `free' parts of $g_w$, denoted by
\begin{equation}\label{g_t-g_f}
    \tau_w(g): D_w \to D_w\quad , \quad \phi_w(g): \Phi_w \cong  K_w/D_w \to K_w/D_w \cong \Phi_w\; .
\end{equation}
Notice that the torsion part $\tau_w(g)$ is a homomorphism of finite groups, whereas the free part \linebreak $\phi_w(g): \Phi_w \to \Phi_w$ is a homomorphism of lattices. Finally, denote by $n_{w}(g) := |D_w/\image(I_{D_w} - \tau_w(g))|$ the order of the cokernel of $I_{D_w} - \tau_w(g)$ in $D_w$.

\begin{thm}\label{thm:algorithm_CV}
    The stringy mixed Hodge polynomial of $T^r/W$ under the diagonal action of $W$ is:
    \[
        \mu^{\str}(T^r/W)(t,u,v)=\sum_{[w] \in \textup{Conj}(W)}
        \frac{(t^{2}uv)^{r\age(w)}}{|C(w)|}\left(\sum_{g \in C(w)} \, n_{w}(g)^r \det\left(I_{\Phi_w} + tuv\,\phi_w(g)\right)^r \right)
        ,
    \]
    where $\age(w) = \frac{1}{2}\rk \image(I_\Lambda -w)$.
\end{thm}

\begin{proof}
    By Definition \ref{def:str_hodge_pol} of the stringy mixed Hodge polynomial, we have
$$
    \mu^{\str}(T^r/W) = \sum_{[w] \in \text{Conj}(W)} \mu(T^r(\CC)^w / C(w))\, (t^{2}uv)^{\age_{T^r}(w)},
$$
where $\age_{T^r}(w) = r \age(w)$ is the fermionic shift of $w$ on $T^r$ by Lemma \ref{lem:fermionic-shift}, being $\age(w)$ the shift on $T$. For this shift, recall that the action of $W$ on $T^2$ is symplectic (cf.\ subsection \ref{ssec:symplectic_resolutions}). Therefore, by (\ref{eq:age}) and the previous discussion, for the action of $w \in W$ on $T$ we have $\age(w) = \frac{1}{2}\codim T(\CC)^w = \frac{1}{2}\rk A_w$. Thus, for the action of $W$ on $T^r$ we have $\age_{T^r}(w) = r \age(w) = \frac{r}{2}\rk A_w$.

Therefore, using the Burnside formula from Proposition \ref{prop:Burnside}, the fact that $T^r(\CC)^w = (T(\CC)^w)^r$ and observation (\ref{item:product}) from Section \ref{sec:mhs-quotients}, it is enough to prove that, for any $w \in W$ and $g \in C(w)$, we have
$$
    \sum_{k,p,q} \chi_{H^{k,p,q}(T(\CC)^w)}(g) \,t^ku^pv^q = n_{w}(g) \det\left(I_{\Phi_w} + tuv\,\phi_w(g)\right)\; .
$$

For this purpose, recall that $T(\CC)^w$ is isomorphic to the disjoint union of $|\Hom_{\Ab}(D_w, \CC^*)|$ copies of the torus $(\CC^*)^{s_w}$, with $s_w= \rk K_w$. Therefore, by observation (\ref{item:disjoint-union}) of Section \ref{sec:mhs-quotients}, we have that
$$
\chi_{H^{k,p,q}(T(\CC)^w)}(g) = \sum_{\rho \in \Hom_{\Ab}(D_w, \CC^*)^g} \tr\left(g_{\rho}: H^{k,p,q}(\{\rho\} \times (\CC^*)^{s_w}) \to H^{k,p,q}(\{\rho\} \times (\CC^*)^{s_w})\right),
$$
where the sum runs over those $\rho \in \Hom_{\Ab}(D_w, \CC^*)$ such that $g \cdot \rho = \rho$, the torus $\{\rho\} \times (\CC^*)^{s_w}_\rho$ is the copy of $(\CC^*)^{s_w}$ corresponding to $\rho$, and
$g_{\rho}: H^{k,p,q}(\{\rho\} \times (\CC^*)^{s_w})\to H^{k,p,q}(\{\rho\} \times (\CC^*)^{s_w})$ is the associated automorphism on the cohomology of the torus induced by $\phi_w(g): \Phi_w \to \Phi_w$ using that $(\CC^*)^{s_w} = \Hom_{\Ab}(\Phi_w, \CC^*)$.

The number of homomorphisms $\rho \in \Hom_{\Ab}(D_w, \CC^*)$ with $g \cdot \rho = \rho$ can be easily calculated as above, since it corresponds to those $\rho: D_w \to \CC^*$ with $\rho(\image(I_{D_w} - \tau_w(g))) = 1$. Hence, $\Hom_{\Ab}(D_w, \CC^*)^g = \Hom_{\Ab}(D_w/ \image(I_{D_w} - \tau_w(g)), \CC^*)$ and thus its cardinality is $|D_w/ \image(I_{D_w} - \tau_w(g))| = n_w(g)$. Now, observe that the maps $g_\rho: \{\rho\} \times (\CC^*)^{s_w} \to \{\rho\} \times (\CC^*)^{s_w}$ for different $\rho \in \Hom_{\Ab}(D_w, \CC^*)^g$ are all homotopic, since they differ by a translation in the torus. Therefore, we get that
\begin{align*}
\sum_{k,p,q} \chi_{H^{k,p,q}(T(\CC)^w)}(g) \,t^ku^pv^q & = \sum_{k,p,q}  n_w(g) \tr\left(g: H^{k,p,q}((\CC^*)^{s_w}) \to H^{k,p,q}((\CC^*)^{s_w})\right) \,t^ku^pv^q \\
& = n_w(g) \sum_{k,p,q} \chi_{H^{k,p,q}((\CC^*)^{s_w})}(g)\,t^ku^pv^q\; .
\end{align*}

Now, observe that the cohomology of $(\CC^*)^{s_w}$ is an exterior algebra generated by the degree one elements, i.e.\ $H^\bullet((\CC^*)^{s_w}) = \Lambda^\bullet H^1((\CC^*)^{s_w}) = \Lambda^\bullet(\CC(-1)^{\oplus s_w})$, with $\CC(-1)$ the Tate twist of weight $2$. The mixed Hodge structure on the basic piece is $(\CC(-1)^{\oplus s_w})^{1,1} = \CC^{s_w}$ and $(\CC(-1)^{\oplus s_w})^{p,q} = 0$ otherwise, so by observation (\ref{item:determinant}) we have
$$
\sum_{k,p,q} \chi_{H^{k,p,q}((\CC^*)^{s_w})}(g)\,t^ku^pv^q = \det(I_{\CC^{s_w}} + tuv\,g)\; ,
$$
where $g$ is acting on the cohomology $H^{1}((\CC^*)^{s_w})=H^{1,1,1}((\CC^*)^{s_w}) = \CC^{s_w}$. The proof is thus completed by noticing that $H^{1}((\CC^*)^{s_w}) = \Phi_w \otimes_{\ZZ} \CC$ and that the action of $\phi_w(g)$ preserves the lattice, so \linebreak $\det(I_{\CC^{s_w}} + tuv\,g) = \det(I_{\Phi_w} + tuv\,\phi_w(g))$, with $\phi_w$ acting on $\Phi_w$.
\end{proof}

\begin{rem}\label{rmk:rational-det}
    The determinant $\det\left(I_{\Phi_w} + tuv\,\phi_w(g)\right)$ on the free part can be alternatively computed using its rational presentation. More explicitly, recall that $\Phi_w \otimes_\ZZ \QQ = K_w \otimes_\ZZ \QQ$, and the determinant is independent of extension of scalars. Therefore, we have
    $$\det\left(I_{\Phi_w} + tuv\,\phi_w(g)\right) = \det\left(I_{K_{w} \otimes_\ZZ \QQ} + tuv\,(g_{w})_\QQ\right),
    $$
    where $(g_w)_\QQ: K_w \otimes_\ZZ \QQ \to K_w \otimes_\ZZ \QQ $ is the induced $\QQ$-linear map.
\end{rem}

\begin{rem}\label{rmk:w-1}
As explained in Remark \ref{rem:first-term}, in the formula of Theorem \ref{thm:algorithm_CV}, the term corresponding to $w = 1$ corresponds to the usual non-stringy mixed Hodge polynomial of $T^r/W$. In that case, $K_{w} = \Lambda$ and thus its torsion part $D_w = 0$, the centralizer is $C(w)=W$, and the fermionic shift is $\rk (I-w) = 0$. In this manner, from Theorem \ref{thm:algorithm_CV}, we directly get
    \[
        \mu(T^r/W)(t,u,v)=
        \frac{1}{|W|}\sum_{g \in W} \, \det\left(I_{{\Lambda}} + tuv\,\phi_1(g)\right)^r.
    \]
This recovers the formula of \cite[Theorem 5.2]{FS} by noticing that $\phi_1(g): \Lambda \to \Lambda$ is just denoted by $g$ in that work.
\end{rem}

\begin{cor}\label{cor:compact-support}
    The stringy mixed Hodge polynomial with compact support of the diagonal action of the Weyl group on $T^r$ is given by 
    \[
        \mu^{\str}_c(T^r/W)=\sum_{[w] \in \textup{Conj}(W)}
        \frac{t^{r\,\rk \Lambda}(uv)^{r\, \age(w)}}{|C(w)|}\left(\sum_{g \in C(w)} \, n_{w}(g)^r \det\left(tuv\,I_{\Phi_w} + \phi_w(g)\right)^r \right)
        .
    \]
\end{cor}

\begin{proof}
    We shall apply Poincar\'e duality as described in Remark \ref{rem:mu-and-muc} to each term $T^r(\CC)^w/C(w)$. Recall that, by the proof of Theorem \ref{thm:algorithm_CV} the Hodge polynomial for the action of each fixed torus is
    $$
        \mu(T^r(\CC)^w/C(w))(t,u,v) = \frac{1}{|C(w)|}\sum_{g \in C(w)} \, n_{w}(g)^r \det\left(I_{\Phi_w} + tuv\,\phi_w(g)\right)^r.
    $$
    Therefore, using that $\dim T^r(\CC)^w/C(w) = \dim T^r(\CC)^w = r\rk K_w=rs_w$, we have
    \begin{align*}
        \mu_c(T^r(\CC)^w/C(w))(t,u,v) &= (t^2uv)^{rs_w}\mu(T^r(\CC)^w/C(w))(t^{-1},u^{-1},v^{-1}) \\
        &= \frac{1}{|C(w)|}\sum_{g \in C(w)} \, n_{w}(g)^r (t^2uv)^{rs_w}\det\left(I_{\Phi_w} + t^{-1}u^{-1}v^{-1}\,\phi_w(g)\right)^r\\
        &= \frac{1}{|C(w)|}\sum_{g \in C(w)} \, n_{w}(g)^r t^{rs_w}\det\left(tuv\,I_{\Phi_w} + \phi_w(g)\right)^r.
        \end{align*}
    Now, the result follows from shifting each contribution by $(t^2uv)^{r\,\age(w)}$, as prescribed by Lemma \ref{lem:fermionic-shift}, and the relation (\ref{eq:ranks_lattices}). 
\end{proof}

\begin{rem}\label{rem:form-CV-conjugacy-classed}
    If $g,g' \in C(w)$ are conjugate elements in $C(w)$, then $n_{w}(g) = n_{w}(g')$ and \linebreak $\det(I_{\Phi_w}+ tuv\,\phi_w(g)) = \det(I_{\Phi_w}+ tuv\,\phi_w(g'))$. Hence, in the inner sum in Theorem \ref{thm:algorithm_CV}, we can group together the elements in the same conjugacy class to get the more compact expressions
    \begin{align}\label{eq:magic-formula}
    \begin{split}
        \mu^{\str}(T^r/W)&= \sum_{[w] \in \textup{Conj}(W)}
        \frac{(t^2uv)^{r\age(w)}}{|C(w)|} \left(\sum_{[g] \in \textup{Conj}(C(w))} c_{w}(g)\, n_{w}(g)^r \det\left(I_{\Phi_w}+ tuv\,\phi_w(g)\right)^r \right),\\
        \mu^{\str}_c(T^r/W)&= \sum_{[w] \in \textup{Conj}(W)}
        \frac{t^{r \rk \Lambda}(uv)^{r\age(w)}}{|C(w)|} \left(\sum_{[g] \in \textup{Conj}(C(w))} c_{w}(g)\, n_{w}(g)^r \det\left(tuv\,I_{\Phi_w}+ \phi_w(g)\right)^r \right),
    \end{split}
    \end{align}
    where $c_w(g)$ is the number of elements of the conjugacy class of $g$ in $C(w)$.
\end{rem}

\subsection{Stringy mixed Hodge polynomials of moduli spaces of $G$-Higgs bundles}\label{sec:mixedHodge-Higgs}

The calculation performed above can be repeated for the moduli space of $G$-Higgs bundles on an abelian variety $A$ of dimension $d$, in order to prove Theorem \ref{thm:algorithm_Higgs-intro}. We first note that, as in (\ref{eq:Higgs}), the normalization of $\mathcal{M}_A(G)$ is given by
$$
    \hat{\mathcal{M}}_A(G) = \mathcal{M}_A(T) / W\; ,
$$
where $\mathcal{M}_A(T) = \Hom_{\Ab}(\Lambda, A^\vee \times H^0(A, \Omega^1_A))$.  
As above, we can decompose the stringy mixed Hodge polynomial of $\mathcal{M}_A(T)/W$ into its sectors
$$
    \mu^{\str}(\mathcal{M}_A(T)/W) = \sum_{[w] \in \text{Conj}(W)} \mu(\mathcal{M}_A(T)^w / C(w))\, (t^{2}uv)^{\age(w)}.
$$

Thus, following the same notation of Section \ref{section:stringy-CV}, Theorem \ref{thm:algorithm_Higgs-intro} is equivalent to the following.

\begin{thm}\label{thm:sector-Higgs}
    Let $A$ be an abelian variety of dimension $d$. For every $w\in W$, \linebreak $\age(w)=d \rk \image(I_\Lambda - w)$, and the mixed Hodge polynomial of its corresponding twisted sector is:
$$
    \mu(\mathcal{M}_A(T)^w / C(w)) = \frac{1}{|C(w)|} \sum_{g \in C(w)}
   n_w(g)^{2d} \det(I_{\Phi_w} + tu\,\phi_w(g))^d\det(I_{\Phi_w} + tv\,\phi_w(g))^d.
$$
\end{thm}
\begin{proof}
The action of $W$ on $\mathcal{M}_A(T) = \Hom_{\Ab}(\Lambda, A^\vee \times H^0(A, \Omega^1_A))$ is induced by its natural action on $\Lambda$. This implies that the proof of Theorem \ref{thm:algorithm_CV} can be repeated verbatim in this setting to show that
$$
    \mu(\mathcal{M}_A(T)^w / C(w)) = \frac{1}{|C(w)|} \sum_{g \in C(w)}
    \sum_{k,p,q} \chi_{H^{k,p,q}(\mathcal{M}_A(T)^w)}(g) \,t^ku^pv^q\; .
$$
Arguing as for character varieties in subsection \ref{section:stringy-CV} we have that
\begin{align*}
    \mathcal{M}_A(T)^w = \Hom_{\Ab}(K_w, A^\vee \times H^0(A, \Omega^1_A)) = \Hom_{\Ab}(D_w, A^\vee) \times \Hom_{\Ab}(\Phi_w, A^\vee \times H^0(A, \Omega^1_A))\; ,
\end{align*}
where we have used that $\Hom_{\Ab}(D_w, H^0(A, \Omega^1_A)) = 0$ since $D_w$ is a finite group and $H^0(A, \Omega^1_A)$ is a complex vector space (and thus is torsion free). 

The fixed components are now parametrized by $\Hom_{\Ab}(D_w/\image(I_{D_w} - \tau_w(g)), A^\vee)$. If we write $D_w/\image(I_{D_w} - \tau_w(g)) \cong \prod_i \ZZ_{n_i}$, then we have $\Hom_{\Ab}(D_w/\image(I_{D_w} - \tau_w(g)), A^\vee) = \prod_i A^\vee[n_i]$, where $A^\vee[n_i]$ denotes the elements of $n_i$-torsion in $A^\vee$. Using that $|A^\vee[n_i]| = n_i^{2d}$, we get that the number of fixed components is
$$
    |\Hom_{\Ab}(D_w/\image(I_{D_w} - \tau_w(g)), A^\vee)| = |D_w/\image(I_{D_w} - \tau_w(g))|^{2d} = \prod_i n_i^{2d} = n_w(g)^{2d},
$$
where recall that, by definition, $n_w(g) = |D_w/\image(I_{D_w} - \tau_w(g))|$. Furthermore, all these components are isomorphic to $\Hom_{\Ab}(\Phi_w, A^\vee \times H^0(A, \Omega^1_A))$, and the induced maps are isomorphic. Therefore, for any $g \in C(w)$, we have 
$$
 \sum_{k,p,q} \chi_{H^{k,p,q}(\mathcal{M}_A(T)^w)}(g) \,t^ku^pv^q = n_w(g)^{2d} \sum_{k,p,q} \chi_{H^{k,p,q}(\Hom_{\Ab}(\Phi_w, A^\vee \times H^0(A, \Omega^1_A)))}(g)\,t^ku^pv^q\; ,
$$
where $g: H^{k,p,q}(\Hom_{\Ab}(\Phi_w, A^\vee \times H^0(A, \Omega^1_A))) \to H^{k,p,q}(\Hom_{\Ab}(\Phi_w, A^\vee \times H^0(A, \Omega^1_A)))$ is the induced map in cohomology by the free part $\phi_w(g): \Phi_w \to \Phi_w$ of $g$. Since $H^0(A, \Omega^1_A)$ is contractible, we have that $\Hom_{\Ab}(\Phi_w, A^\vee \times H^0(A, \Omega^1_A))$ has the same homotopy type as the space $Z_w := \Hom_{\Ab}(\Phi_w, A^\vee)$. Notice that $Z_w \cong (A^\vee)^{s_w}$, with $s_w = \rk \Phi_w$. %

The cohomology of $A^\vee$ is an exterior algebra generated by the degree one elements, so $H^\bullet(A^\vee) = \bigwedge^\bullet H^1(A^\vee)$. Now, observe that $A^\vee$ is a smooth projective variety, so $H^1(A^\vee)$ is a pure Hodge structure of weight $1$ given by $H^1(A^\vee) = H^{1,0}(A^\vee) \oplus H^{0,1}(A^\vee)$ with $H^{1,0}(A^\vee) = H^{0,1}(A^\vee) = \CC^d$. In this way, the cohomology of $Z_w$ is also an exterior algebra of the form
    $$
H^\bullet(Z_w) = \bigwedge^\bullet H^1(Z_w), \textrm{ with } H^1(Z_w) = \Phi_w^\CC \otimes_\CC H^1(A^\vee)\; , 
$$
where $\Phi_w^\CC = \Phi_w \otimes_\ZZ \CC$ is the complexification of the lattice $\Phi_w$. Again, the Hodge structure on $H^1(Z_w)$ is pure of weight $1$, with
$$
    H^{1,0}(Z_w) = \Phi_w^\CC \otimes_\CC H^{1,0}(A^\vee) = \Phi_w^\CC \otimes_\CC \CC^d\quad , \quad H^{0,1}(Z_w) = \Phi_w^\CC \otimes_\CC H^{0,1}(A^\vee) = \Phi_w^\CC \otimes_\CC \CC^d\; .
$$

From this cohomological calculation, we thus get
$$
    \sum_{k, p, q} \chi_{H^{k,p,q}(Z_w)}(g)\,t^ku^p v^q = \det(I_{H^{1,0}(Z_w)} + tu\,g^{1,0})\det(I_{H^{0,1}(Z_w)} + tv\,g^{0,1})\; ,
$$
where we are considering the maps in each piece of the Hodge structure $g^{1,0}: H^{1,0}(Z_w) \to H^{1,0}(Z_w)$ and $g^{0,1}: H^{0,1}(Z_w) \to H^{0,1}(Z_w)$.

Actually, using that $H^{1,0}(Z_w) = H^{0,1}(Z_w) = \Phi_w^\CC \otimes_\CC \CC^d$ and that $g$ only acts on the lattice part, we observe that $g^{1,0} = g^{0,1} = \phi_w(g)_\CC \otimes I_{\CC^d}$, where $\phi_w(g)_\CC: \Phi_w^\CC \to \Phi_w^\CC $ is the complexification of the free part $\phi_w(g): \Phi_w \to \Phi_w$.
This implies that $\det(I_{H^{1,0}(Z_w)} + tu\,g^{1,0}) = \det(I_{\Phi_w^\CC} + tu\,\phi_w(g)_\CC)^d = \det(I_{\Phi_w} + tu\,\phi_w(g))^d$, and analogously $\det(I_{H^{0,1}(Z_w)} + tv\,g^{0,1}) = \det(I_{\Phi_w} + tv\,\phi_w(g))^d$. Hence, we have
$$
    \mu(\mathcal{M}_A(T)^w / C(w)) = \frac{1}{|C(w)|} \sum_{g \in C(w)}
   n_w(g)^{2d} \det(I_{\Phi_w} + tu\,\phi_w(g))^d\det(I_{\Phi_w} + tv\,\phi_w(g))^d\; .
$$

Finally, with respect to the fermionic shift of $w \in W$ when acting on $\mathcal{M}_A(T)$, recall that
\begin{align*}
    \age(w) &= \frac{1}{2}\codim_\CC \mathcal{M}_A(T)^w = \frac{1}{2}\codim_\CC \Hom_{\Ab}(K_w, A^\vee \times H^0(A, \Omega^1_A)) \\
    &= \frac{2d\rk \Lambda - 2d \rk K_w}{2} = d\rk A_w\; ,
\end{align*}
where recall that $A_w = \image(I_\Lambda - w)$. This completes the proof of Theorem \ref{thm:sector-Higgs} and thus also of Theorem \ref{thm:algorithm_Higgs-intro}.
\end{proof}

Analogous formulae can be deduced for the compactly supported stringy mixed Hodge polynomial of $\mathcal{M}_A(T) / W$, as in Corollary \ref{cor:compact-support} and Remark \ref{rem:form-CV-conjugacy-classed}, obtaining the following.

\begin{thm}\label{thm:algorithm_Higgs_compact}
    Let $A$ be an abelian variety of dimension $d$. The compactly supported stringy mixed Hodge polynomial of $\hat{\mathcal{M}}_A(G) = \mathcal{M}_A(T) / W$ is given by
\[
        \mu_c^{\str} (\mathcal{M}_A(T) / W)(t,u,v) 
        =\sum_{[w] \in \textup{Conj}(W)}
        \frac{(t^{2}uv)^{d\rk\Lambda }}{|C(w)|}\left(\sum_{g \in C(w)} \, n_{w}(g)^{2d}\, q_{g,w}(tu)^d\, q_{g,w}(tv)^d \right),
\]
where $q_{g,w}(x):=\det (x\, I_{\Phi_w} + \phi_w(g))$.
\end{thm}

\section[mirror]{Topological mirror symmetry for moduli spaces on abelian varieties}\label{sec:topological-mirror-symmetry}

In this section, using Theorem \ref{thm:algorithm_CV} we prove that the stringy invariants of the character varieties of an abelian variety coincide for every pair of Langlands dual groups. This should be seen as another instance of topological mirror symmetry for these moduli spaces (e.g. see \cite{FlorentinoNozadZamora2021}).

It is important to point out that this incarnation of mirror symmetry was previously discovered by Thaddeus in \cite{Th} for stringy $E$-polynomials. The proof provided in this work runs parallel to the one in \cite{Th} since both are based on studying the $C(w)$-module structure of $H^\bullet(T(\CC)^w/C(w))$. However, the explicit formula of Theorem \ref{thm:algorithm_CV} allows us to provide a more direct proof based on analyzing dual morphisms of dual abelian groups. These calculations on dual groups are standard and are provided in Appendix \ref{app:proof-mirror-symmetry} for completeness.

Let $G$ be a connected reductive group and $^LG$ its Langlands dual. This means that the character lattice of $^LG$ is the dual cocharacter lattice $\Lambda^\vee = \Hom(\Lambda, \ZZ)$ of the character lattice $\Lambda$ of $G$. Furthermore, for any $w \in W$, the action of $w$ on $\Lambda^\vee$, corresponding to $^LG$, is the dual of the action of $w^{-1}$ on $\Lambda$ corresponding to $G$, i.e.\ $w \cdot \alpha = (w^{-1})^\vee \alpha$ for $\alpha \in \Lambda^\vee$. Here, $(w^{-1})^\vee: \Lambda^\vee \to \Lambda^\vee$ denotes the map $((w^{-1})^\vee \alpha)(\lambda) =  \alpha(w^{-1} \cdot \lambda)$ for $\lambda \in \Lambda$.

Therefore, the $\CC$-points of the fixed points of the maximal tori, denoted by $T_{G}$ and $T_{^LG}$ respectively, are
$$
    T_G(\CC)^w = \Hom_{\Ab}(K_w, \CC^*)\quad , \quad T_{^LG}(\CC)^{w^{-1}} = \Hom_{\Ab}(\hat{K}_{w}, \CC^*)\; ,
$$
where $K_w = \coker(I-w) = \Lambda/\image(I-w)$ and $\hat{K}_w := \coker((I-w)^\vee) = \Lambda^\vee/\image((I-w)^\vee)$, with respective free parts $\Phi_w$ and $\hat{\Phi}_w$.

\begin{thm}\label{thm:equality-terms}
    For any connected reductive group $G$ and any $w \in W$, we have
    $$
        \mu(T^r_G(\CC)^w/C(w)) = \mu(T^r_{^LG}(\CC)^{w^{-1}}/C(w^{-1}))\; .
    $$
\end{thm}

\begin{proof}
Let $\Phi_w$ and $\hat{\Phi}_w$ denote the free parts of $K_w = \coker(I-w)$ and $\hat{K}_w = \coker((I-w)^\vee)$, and $\phi_w(g): \Phi_w \to \Phi_w$ and $\hat{\phi}_w(g^\vee):\hat{\Phi}_w \to \hat{\Phi}_w$ are the free parts of $g: K_w \to K_w$ and $g^\vee: \hat{K}_w \to \hat{K}_w$, respectively. As proven in Theorem \ref{thm:algorithm_CV}, we have
\begin{align*}
    \mu(T^r_G(\CC)^w/C(w)) &= \frac{1}{|C(w)|}\sum_{g \in C(w)} n_{w}(g)^r \det\left(I_{\Phi_w} + tuv\,\phi_w(g)\right)^r\; ,\\
    \mu(T^r_{^LG}(\CC)^{w^{-1}}/C(w^{-1})) &= \frac{1}{|C(w^{-1})|}\sum_{g \in C(w^{-1})} \hat{n}_{w}((g^{-1})^\vee)^r \det\left(I_{\hat{\Phi}_w} + tuv\,\hat{\phi}_w((g^{-1})^\vee)\right)^r\; .
\end{align*}
Here, if $D_w$ and $\hat{D}_{w}$ are the torsion parts of $K_w$ and $\hat{K}_w$, then $n_w(g) = |\coker(I-\tau_w(g))|$ and $\hat{n}_{w}(g^\vee) = |\coker(I-\hat{\tau}_w(g^\vee))|$ are the number of points of the respective cokernels for the morphisms on the torsion parts $I-\tau_w(g): D_w \to D_w$ and $I-\hat{\tau}_w(g^\vee): \hat{D}_w \to \hat{D}_w$. Notice that, $C(w) = C(w^{-1})$ and the sum in $\mu(T^r_{^LG}(\CC)^{w^{-1}}/C(w^{-1}))$ is closed under inversion, so we can alternatively write 
\begin{align*}
    \mu(T^r_{^LG}(\CC)^{w^{-1}}/C(w^{-1})) &= \frac{1}{|C(w)|}\sum_{g \in C(w)} \hat{n}_{w}(g^\vee)^r \det\left(I_{\hat{\Phi}_w} + tuv\,\hat{\phi}_w(g^\vee)\right)^r.
\end{align*}

The torsion part is controlled by means of the calculations performed in Appendix \ref{app:proof-mirror-symmetry}. Concretely, Corollary \ref{cor:number-points} applied to $h = I-g$ implies that $n_w(g) = \hat{n}_{w}(g^\vee)$ for all $g \in C(w) = C(w^{-1})$.

Now, for the determinants, notice that by Remark \ref{rmk:rational-det}, we have 
\begin{align*}
\det\left(I_{\Phi_w} + tuv\,\phi_w(g)\right) &= \det\left(I_{K_{w} \otimes_\ZZ \QQ} + tuv\,(g_{w})_\QQ\right), \\
\det\left(I_{\hat{\Phi}_{w}} + tuv\,\hat{\phi}_w(g^\vee)\right) &= \det\left(I_{\hat{K}_{w} \otimes_\ZZ \QQ} + tuv\,(g_{w}^\vee)_\QQ\right),
\end{align*}
with $(g_w)_\QQ: K_w \otimes_\ZZ \QQ \to K_w \otimes_\ZZ \QQ $ and $(g_w^\vee)_\QQ: \hat{K}_w \otimes_\ZZ \QQ \to \hat{K}_w \otimes_\ZZ \QQ$ the associated $\QQ$-linear extensions. However, in this case, as shown in Appendix \ref{app:proof-mirror-symmetry}, we have that $\hat{K}_w \otimes_\ZZ \QQ = (K_w \otimes_\ZZ \QQ)^*$ and $(g_w^\vee)_\QQ$ is the transpose map to $(g_w)_\QQ$ so, in particular, we have $\det(I_{K_{w} \otimes_\ZZ \QQ} + tuv\,(g_{w})_\QQ) = \det(I_{\hat{K}_{w} \otimes_\ZZ \QQ} + tuv\,(g_{w}^\vee)_\QQ)$, completing the proof.
\end{proof}

\begin{cor}\label{cor:mirror-symmetry-CV}
For any connected reductive group $G$ with Langlands dual $^LG$ and any $r \geq 0$ we have 
$$
\mu^{\str}(T_G^r/W) = \mu^{\str}(T_{^LG}^r/W)\; ,
$$
where $T_G$ and $T_{^LG}$ are the maximal tori of $G$ and $^LG$, respectively.
\end{cor}

\begin{proof}
    By Theorem \ref{thm:equality-terms}, the associated formulas for the stringy mixed Hodge polynomial in Theorem \ref{thm:algorithm_CV} for $G$ and $^LG$ are in agreement where the term for $w$ in $G$ corresponds to the term for $w^{-1}$ in $^LG$. Hence, the global equality follows by observing that the ages of $w$ for $G$ and of $w^{-1}$ for $^LG$ coincide, and this follows from the fact that
    $$
        \rk \image (I-w) = \dim_\QQ \image (I-w)_\QQ =  \dim_\QQ \image (I-w)^*_\QQ  = \rk \image (I-w)^\vee\; ,
    $$
    where $(I-w)_\QQ: \Lambda \otimes_\ZZ \QQ \to \Lambda \otimes_\ZZ \QQ$ and $(I-w)_\QQ^*: \Lambda^\vee \otimes_\ZZ \QQ \to \Lambda^\vee \otimes_\ZZ \QQ$ are the associated $\QQ$-rational maps.
\end{proof}

\begin{rem}
    As pointed out in Remark \ref{rmk:w-1}, the case $w = 1$ of Theorem \ref{thm:equality-terms} implies an equality of the standard mixed Hodge polynomials for Langlands dual groups 
    $$
        \mu(T^r_G/W) = \mu(T^r_{^LG}/W)\; .
    $$
    In particular, recall that $T^r_G/W \to {\mathfrak{X}}^\circ_r(G)$ (resp.\ $T^r_{^LG}/W \to {\mathfrak{X}}^\circ_r(^LG)$) is the normalization of the identity component ${\mathfrak{X}}^\circ_r(G) \subseteq \mathfrak{X}_r(G)$ (resp.\ of ${\mathfrak{X}}^\circ_r(^LG) \subseteq \mathfrak{X}_r(^LG)$), and this normalization map preserves the mixed Hodge structure \cite[Section 5.1]{FS}. Therefore, we get an equality of mixed Hodge polynomials for the identity component of the character varieties $\mu({\mathfrak{X}}^\circ_r(G)) = \mu({\mathfrak{X}}^\circ_r(^LG))$ for all $r \geq 1$. However, it is not known whether such an equality holds for \emph{stringy} mixed Hodge polynomials of ${\mathfrak{X}}^\circ_r(G)$.
\end{rem}

It is worth mentioning that the same argument works verbatim for moduli spaces of Higgs bundles over an abelian variety. Therefore, we get the following result.

\begin{cor}\label{cor:mirror-symmetry-Higgs}
Let $G$ be a connected reductive group with Langlands dual $^LG$, with respective maximal tori $T_G$ and $T_{^LG}$ and Weyl group $W$. For any abelian variety $A$ and every $w \in W$, we have $
        \mu(\mathcal{M}_{A}(T_G)^w/C(w)) = \mu(\mathcal{M}_{A}(T_{^LG})^{w^{-1}}/C(w^{-1}))$.
In particular, we have
$$
\mu^{\str}(\mathcal{M}_{A}(T_G)/W)= \mu^{\str}(\mathcal{M}_{A}(T_{^LG})/W) \quad \textrm{ and } \quad \mu(\mathcal{M}_{A}(G))= \mu(\mathcal{M}_{A}(^LG))\; .
$$
\end{cor}

\subsection*{An alternative approach}
    The matrix description of the Weyl group described in Section \ref{sec:action_weyl} can be used to give a simpler proof of Theorem \ref{thm:equality-terms}. Suppose that $G$ is a connected reductive group of rank $n$ and semisimple rank $m \leq n$. Let $\alpha_1, \ldots, \alpha_m$ be a basis of the root lattice $Q$ of simple roots, with associated coroots $\alpha_1^\vee, \ldots, \alpha_m^\vee$ of the coroot lattice $Q^\vee$. Then, the Cartan matrix $A = (a_{ij})$ has entries
    $$
        a_{ij} = \langle \alpha_j, \alpha_i^\vee\rangle
    $$
    with respect to the natural pairing $\langle - , -\rangle: Q \otimes_\ZZ Q^\vee \to \ZZ$.

    However, for the Langlands dual group $^LG$, the role of roots and coroots is permuted, so a basis of roots for $^LG$ is precisely $\alpha_1^\vee, \ldots, \alpha_m^\vee \in Q^{\vee}$, whereas a basis of coroots is $\alpha_1, \ldots, \alpha_m\in Q$. This implies that the Cartan matrix $^L\widetilde{A} = (^La_{ij})$ associated to $^LG$ is given by
    $$
        ^La_{ij} = \langle \alpha_j^\vee, \alpha_i\rangle = a_{ji}\; .
    $$
    In this manner, we get that $^LA = A^t$ is the transpose of $A$. Equivalently, the Dynkin diagrams of the root systems of $G$ and $^LG$ are dual, so the corresponding Cartan matrices, which can be read from the diagram, are transposed.
    
    As a direct consequence of this, we get that the matrices $^LC_{s_i}$ representing the action of the reflection of the simple root $\alpha_i^\vee$ of $^LG$ on the basis $\alpha_1^\vee, \ldots, \alpha_m^\vee$, is also the transpose of the corresponding matrix $C_{s_i}$ of the action of $\alpha_i$ for $G$ (see (\ref{eq:matrix_Ck})).
    
    This can be extended to the character lattice $\Lambda$ as follows. Choose a basis of characters $\chi_1, \ldots, \chi_n$ with $\langle \chi_1, \ldots, \chi_m\rangle_\QQ = Q \otimes_\ZZ \QQ$ and $\chi_{m+1}, \ldots, \chi_n$ spanning the central part of the vector space $\Lambda_{z,\mathbb{Q}}$, as described in Section \ref{sec:action_weyl}. Then, the dual basis $\chi_1^\vee, \ldots, \chi_n^\vee$ of the cocharacter lattice $\Lambda^\vee$ is another basis with the same properties. In particular, the matrix $^LD_{s_i}$ for the action of the Weyl group on $\chi_1^\vee, \ldots, \chi_n^\vee$, corresponding to $^LG$, becomes also transpose to the matrix $D_{s_i}$ for the action on $\chi_1, \ldots, \chi_n$ corresponding to $G$ (see (\ref{eq:matrix-Dk})).

    Using these matrices, it is possible to compute the decomposition of the cokernel of $I-w: \Lambda \to \Lambda$ into its free and torsion part. Particularly, if $w$ is the product of the reflections associated to roots $\alpha_{i_1}, \ldots, \alpha_{i_\ell}$, then $B := B_{i_\ell}\cdots B_{i_1}$ is the corresponding matrix representing $w$, where $B_{i}$ is the matrix of the reflection corresponding to $\alpha_i$. In particular, $I-w$ is represented by the matrix $I-B$. Considering the Smith normal form of $I-w$, we can read the torsion part as the invariant factors of the normal form, and the rank of the free part as the number of vanishing entries in the diagonal. In particular, notice that for the action of $(I-w)^*$ on $\Lambda^\vee$, we have the associated matrix $(I-B)^t$ and thus the Smith normal form agrees with that of $I-w$ so, in particular, the torsion and free parts of $K_w$ and $\hat{K}_w$ are isomorphic.
    
    Arguing as above, given any $g \in C(w)$, it is possible to check that the matrices representing the action of $g$ on $K_w$ and $\hat{K}_w$ are also transpose. This implies, in particular, the equality of the determinants and the numbers $n_{w}(g)$ and $n_{w^*}(g^*)$ appearing in Theorem \ref{thm:algorithm_CV}, and thus \textit{a fortiori} the desired topological mirror symmetry.

\section{Examples and explicit calculations}
\label{sec:examples}

In this section, we shall perform some explicit calculations of the mixed Hodge polynomial as described in Section \ref{sec:stringy-mhpoly} for several groups. Concretely, we will calculate the compactly supported mixed Hodge polynomial for character varieties as provided by Corollary \ref{cor:compact-support}, as it will also allow us to straightforwardly compute the associated $E$-polynomial. Similar computations can be done with the data shown in this section for other variants, such as the regular mixed Hodge polynomial or the same calculation for moduli spaces of Higgs bundles.

\subsection{Action of the Weyl group on the character lattice}\label{sec:action_weyl}

 Let $G$ be a connected reductive group of rank $n$ with semisimple rank $m \le n$. Let us describe matricially the action of the Weyl group of $G$ in its character lattice $\Lambda$.

Recall that $\Lambda:=X^*(T)$ is the character lattice of a maximal torus $T \subseteq G$ and has rank $n$. Let us denote by $Q$ the lattice of roots of $G$ and by $P$ the lattice of weights of $G$, which satisfy $Q\subseteq P$. Furthermore, if $G$ is semisimple, then $Q \subseteq \Lambda \subseteq P$, in which case the fundamental group of $G$ is $\pi_1(G)=P/\Lambda$. In particular, $G$ is simply connected if and only if $\Lambda = P$. Analogously, we say that a semisimple group $G$ is of adjoint type if $\Lambda = Q$. Denote by $\Lambda^{\vee}=X_{\ast}(T)$ the cocharacter lattice and $\langle \cdot \; , \; \cdot \rangle: X^{\ast}(T)\times X_{\ast}(T)\rightarrow \mathbb{Z}$ the perfect pairing between characters and cocharacters. Similarly, if $G$ is semisimple, we find the relation $P^{\vee}\subseteq \Lambda^{\vee}\subseteq Q^{\vee}$ between the dual lattices, where $P^{\vee}$ is the coroot lattice and $Q^{\vee}$ is the coweight lattice. 

Consider the root lattice $Q$ and let $\alpha_1,\ldots, \alpha_m$ be a basis of 
simple roots for $Q$. Similarly, let $\alpha_1^\vee, \dots, \alpha_m^\vee \in Q^\vee$ 
be the basis of simple coroots for $Q^{\vee}$. These coroots define a basis for 
the weight lattice $P$ given by the fundamental weights $\omega_1, \ldots, \omega_m \in P$, 
dual to the coroot basis under the natural pairing:
\[
\langle \omega_i, \alpha_j^{\vee}\rangle=\delta_{ij}\; , \quad 1\leq i,j\leq m\; .
\] 
In particular, the simple roots can be uniquely expressed as integer combinations 
of the fundamental weights:
\[
\alpha_i = \sum_{j=1}^m \langle \alpha_i, \alpha_j^\vee \rangle\,\omega_j = \sum_{j=1}^m a_{ji}\,\omega_j\; ,
\]
whose coefficients are given by the transpose entries of the Cartan matrix 
$A = (a_{ij})_{1 \leq i,j \leq m} := \big(\langle \alpha_j, \alpha_i^\vee \rangle\big)_{i,j}$ 
of the semisimple part of $G$, following the convention of Kac \cite{Kac1990}. 
Recall that this Cartan matrix can be directly read from the Dynkin diagram of the root system of $G$ as follows: 
$a_{ii} = 2$ for all $i$;
$a_{ij} = 0$ if $\alpha_i$ is not connected with $\alpha_j$;
$a_{ij} = -1$ if $\alpha_i$ and $\alpha_j$ are connected by a simple edge;
$a_{ji} = -2$ and $a_{ij} = -1$ if there is a double edge pointing from $\alpha_i$ to $\alpha_j$ (i.e., $\alpha_j$ is the shorter root);
$a_{ji} = -3$ and $a_{ij} = -1$ if there is a triple edge pointing from $\alpha_i$ to $\alpha_j$ (i.e., $\alpha_j$ is the shorter root).

The Weyl group of $G$ is generated by simple reflections $s_k$ associated to the 
simple roots $\alpha_k$, for $1\leq k\leq m$. The action of $s_k$ on the simple root $\alpha_i$ is given by:
\[
s_k(\alpha_i) = \alpha_i - \langle \alpha_i, \alpha_k^\vee \rangle \alpha_k = \alpha_i - a_{ki}\alpha_k\; , \quad 1\leq i\leq m\; .
\]
In particular, using the column-vector representation for linear transformations, the 
matrix of $s_k$ with respect to the basis $\alpha_1, \ldots, \alpha_m$ is:
\begin{equation}
\label{eq:matrix_Ck}
    C_{k} := I_m - E_{kk}A\; ,
\end{equation}
where $I_m$ is the identity matrix of order $m$, $E_{kk}$ is the matrix with a $1$ 
in the $(k,k)$-entry and zero elsewhere, and $A = (a_{ij})_{1 \le i,j \le m}$ is the Cartan matrix.

Furthermore, in the non-semisimple case, let us extend the root basis $\alpha_1, 
\ldots, \alpha_m$ into an extended basis $\alpha_1, \ldots, \alpha_n$ of 
$\Lambda \otimes_{\mathbb{Z}} \mathbb{Q} = \Lambda_{ss,\mathbb{Q}} \oplus \Lambda_{z,\mathbb{Q}}$, 
where $\alpha_{m+1}, \ldots, \alpha_n$ span the central part $\Lambda_{z,\mathbb{Q}}$ 
as a vector space. Then, since the action of the Weyl group $W$ on the center is 
trivial, we have that:
\[
s_k(\alpha_i) = \alpha_i\; , \quad \text{for } m+1 \leq i \leq n\, , \; 1\leq k\leq m\, ,
\]
and thus the extended matrix of $s_k$ with respect to the basis $\alpha_1, \ldots, \alpha_n$ is:
\begin{equation}
\label{eq:matrix_Ck_extended}
    \widetilde{C}_{k} := I_n - E_{kk}\widetilde{A}\; ,
\end{equation}
where $\widetilde{A}$ is the matrix of order $n \times n$ obtained by padding the 
Cartan matrix $A$ on the right and below with zeros.

Now, for the character lattice $\Lambda$, let us choose a basis $\chi_1, \ldots, \chi_n \in \Lambda$. Let us write:
\[
\alpha_i = \sum_{j=1}^n l_{ji} \chi_j\, , \quad l_{ji} \in \mathbb{Z}\, ,
\]
so that $L = (l_{ij})_{1 \leq i,j \leq n}$ is the matrix of change of basis tracking the coordinates of the extended basis $\alpha_1, \ldots, \alpha_n$ in terms of the character basis $\chi_1, \ldots, \chi_n$. In this manner, the action of $W$ on $\chi_1, \ldots, \chi_n$ is given by:
\begin{equation}\label{eq:matrix-Dk}
    B_{k} = L\widetilde{C}_{k}L^{-1} = I_n - LE_{kk}\widetilde{A}L^{-1}\; .
\end{equation}

\begin{rem}
    If $G$ is simply connected (hence semisimple), the character lattice $\Lambda$ coincides with the weight lattice $P$. Choosing the fundamental weights as the character basis (i.e.\ $\chi_j = \omega_j$) yields $\alpha_i = \sum a_{ji} \omega_j$. Comparing this with $\alpha_i = \sum l_{ji} \chi_j$ immediately gives the identification $L = A$. Consequently, the action matrices $B_{s_k}$ simplify to 
\begin{equation}
  \label{eq:matrix-Dk-sc}  
B_{k} := L C_{k} L^{-1} = A (I_m - E_{kk} A) A^{-1} = I_m - A E_{kk}\; .
\end{equation}
\end{rem}

\subsubsection{Example: Matrix action of the Weyl group on the character lattice for $\GL_3$}

Let $G = \text{GL}_3$ of rank $n = 3$ and semisimple rank $m = 2$. We choose the standard diagonal torus $T \subseteq \text{GL}_3$. In the character lattice $\Lambda = X^*(T) \cong \mathbb{Z}^3$ we choose the natural basis given by the coordinate projections:
\[
\chi_1=(1,0,0)\; ,\quad \chi_2=(0,1,0)\; ,\quad \chi_3=(0,0,1)\; .
\]
The root lattice is $Q = \{(x_1, x_2, x_3) \in \ZZ^3\mid x_1+x_2+x_3=0\}$, so a basis consists of the two simple roots $\alpha_1 = \chi_1 - \chi_2, \alpha_2 = \chi_2 - \chi_3$. 
This basis $\alpha_1, \alpha_2$ of $Q$ can be completed to a rational basis of $\Lambda_{\QQ}$ by taking the character spanning the center
$
    \alpha_3 = \chi_1 + \chi_2 + \chi_3.
$

With respect to the Cartan matrix, observe that the Dynkin diagram of the root system of $\GL_3$ is $A_2$, which contains two vertices connected by a simple edge. In this manner, the Cartan matrix is
\[
A = \left(\begin{array}{cc} 2 & -1 \\ -1 & 2 \end{array}\right), 
\]
and note that it is symmetric, i.e.\ $A^t=A$, as it happens with a simply-laced Dynkin diagram. Hence, the extended Cartan matrix is
\[
\widetilde{A} = \left( \begin{array}{cc|c} 2 & -1 & 0 \\ -1 & 2 & 0 \\ \hline 0 & 0 & 0 \end{array} \right).
\]
The weight lattice basis consists of the two fundamental weights $\omega_1, \omega_2$, given by
\[
\omega_1 = \frac{1}{3}(2\chi_1 -\chi_2-\chi_3)=\left(\frac{2}{3}, -\frac{1}{3}, -\frac{1}{3}\right), \quad  \omega_2 = \frac{1}{3}(\chi_1 +\chi_2-2\chi_3)=\left(\frac{1}{3}, \frac{1}{3}, -\frac{2}{3}\right).
\]
In particular, we have
    $\alpha_1 = 2 \omega_1 -\omega_2, \alpha_2 = -\omega_1 + 2 \omega_2$, 
in agreement with the computed Cartan matrix.

Now we write the matrices corresponding to the two simple reflections for the 
extended basis $\alpha_1, \alpha_2, \alpha_3$. Let $s_1$ be the reflection 
corresponding to the first simple root $\alpha_1$ and let $E_{11}=\text{diag}(1,0,0)$. 
We compute by (\ref{eq:matrix_Ck_extended}):
\[
\widetilde{C}_{1} = I_3 - E_{11}\widetilde{A} = \begin{pmatrix} 1 & 0 & 0 \\ 0 & 1 & 0 \\ 0 & 0 & 1 \end{pmatrix} - \begin{pmatrix} 2 & -1 & 0 \\ 0 & 0 & 0 \\ 0 & 0 & 0 \end{pmatrix} = \begin{pmatrix} -1 & 1 & 0 \\ 0 & 1 & 0 \\ 0 & 0 & 1 \end{pmatrix},
\]
which matches the geometric action
$ s_1(\alpha_1)=-\alpha_1$, $s_1(\alpha_2)=\alpha_1+\alpha_2$, $s_1(\alpha_3)=\alpha_3$. 
For the reflection $s_2$ corresponding to the second simple root $\alpha_2$, we have $E_{22} = \text{diag}(0,1,0)$. We thus get the matrix:
\[
    \widetilde{C}_{2} = I_3 - E_{22}\widetilde{A} = \begin{pmatrix} 1 & 0 & 0 \\ 0 & 1 & 0 \\ 0 & 0 & 1 \end{pmatrix} - \begin{pmatrix} 0 & 0 & 0 \\ -1 & 2 & 0 \\ 0 & 0 & 0 \end{pmatrix} = \begin{pmatrix} 1 & 0 & 0 \\ 1 & -1 & 0 \\ 0 & 0 & 1 \end{pmatrix},
\]
which, again, matches the geometric action of this reflection.

To compute these matrices with respect to the basis $\chi_1, \chi_2, \chi_3$, observe 
that the change of basis matrix from the extended root basis $\alpha_1, \alpha_2, \alpha_3$ 
to the character basis $\chi_1, \chi_2, \chi_3$ is:
\[
L = \begin{pmatrix}
    1 & 0 & 1 \\
    -1 & 1 & 1 \\
    0 & -1 & 1
\end{pmatrix}.
\]
Therefore, the action of the Weyl group on the basis $\chi_1, \chi_2, \chi_3$ is given by (\ref{eq:matrix-Dk}):
\[
B_{1} = L\widetilde{C}_{1}L^{-1} = \begin{pmatrix} 0 & 1 & 0 \\ 1 & 0 & 0 \\ 0 & 0 & 1 \end{pmatrix}\quad , \quad B_{2} = L\widetilde{C}_{2}L^{-1} = \begin{pmatrix} 1 & 0 & 0 \\ 0 & 0 & 1 \\ 0 & 1 & 0 \end{pmatrix}.
\]

\subsection{Example: $G=\SO_7$}
\label{ssec:SO7}

Let us describe the computations in Theorem \ref{thm:algorithm_CV} for the case $G=\mathrm{SO}_7$. This will serve to exemplify how the formula of Theorem \ref{thm:algorithm_CV} is highly computable and only depends on combinatorial data of the root datum.

We follow the general description in Section \ref{sec:action_weyl}, where $\mathrm{SO}_7$ is a group of adjoint (centerless) type of rank $3$, with root system of type $B_3$. Let $T_{\SO_7} \subseteq \mathrm{SO}_7$ be the standard diagonal torus and let $\Lambda = X^*(T_{\SO_7}) = Q$ be the character lattice, which equals the root lattice in the adjoint case. Here we have $n=m=3$, and the Cartan matrix for $B_3$ under Kac's convention \cite{Kac1990} is:
\[
A = \begin{pmatrix} 2 & -1 & 0 \\ -1 & 2 & -1 \\ 0 & -2 & 2 \end{pmatrix}.
\]
Notice that $A$ is non-symmetric since $B_3$ is not simply laced.

Consider the standard orthonormal Euclidean basis $\{e_1, e_2, e_3\}$ of $\mathbb{R}^3$. In terms of this coordinate basis, the simple roots of $B_3$ are given by:
$
\alpha_1 = e_1 - e_2$, $\alpha_2 = e_2 - e_3$, $\alpha_3 = e_3$.
The Weyl group $W \cong S_3 \ltimes (\mathbb{Z}_2)^3$ is the hyperoctahedral group of order $|W| = 2^3 \cdot 3! = 48$, acting on $\mathbb{R}^3$ via signed permutations (permutations $\sigma \in S_3$ combined with coordinate sign changes $\boldsymbol{\epsilon} \in (\mathbb{Z}_2)^3$). 

An arbitrary element $w \in W$ is uniquely written as a pair $w = (\sigma, \boldsymbol{\epsilon})$, where $\sigma \in S_3$ and $\boldsymbol{\epsilon} = (\epsilon_1, \epsilon_2, \epsilon_3) \in \{\pm 1\}^3$. Under the convention where the action on the standard basis vectors is given by $w(e_i) = \epsilon_i e_{\sigma(i)}$ (the permutation $\sigma$ acts first, followed by the sign choice $\epsilon_i$), the induced canonical action of $w = (\sigma, \boldsymbol{\epsilon})$ on a coordinate vector $\mathbf{x} = (x_1, x_2, x_3) \in \mathbb{R}^3$ is given by:
\begin{equation*}
w \cdot (x_1, x_2, x_3) = \left( \epsilon_{\sigma^{-1}(1)} x_{\sigma^{-1}(1)}, \; \epsilon_{\sigma^{-1}(2)} x_{\sigma^{-1}(2)}, \; \epsilon_{\sigma^{-1}(3)} x_{\sigma^{-1}(3)} \right)
\end{equation*}

Since $\mathrm{SO}_7$ is of adjoint type, the character lattice $\Lambda$ is spanned over $\mathbb{Z}$ by the simple roots, so we choose our character basis to be $\chi_i = \alpha_i$ for $i=1,2,3$. Consequently, the change-of-basis matrix between simple roots and characters is $L = I_3$. While $W$ acts on $\{e_1, e_2, e_3\}$ by signed permutation matrices, its action on the character lattice $\Lambda$ relative to the character basis $\{\chi_1, \chi_2, \chi_3\} = \{\alpha_1, \alpha_2, \alpha_3\}$ is given by the matrices $B_{k} = C_k= I_3 - E_{kk} A$ (see (\ref{eq:matrix_Ck})). For the simple reflections $s_1, s_2, s_3$ corresponding to $\alpha_1, \alpha_2, \alpha_3$, these yield:
\[
B_{1} = \begin{pmatrix} -1 & 1 & 0 \\ 0 & 1 & 0 \\ 0 & 0 & 1 \end{pmatrix}, \quad 
B_{2} = \begin{pmatrix} 1 & 0 & 0 \\ 1 & -1 & 1 \\ 0 & 0 & 1 \end{pmatrix}, \quad 
B_{3} = \begin{pmatrix} 1 & 0 & 0 \\ 0 & 1 & 0 \\ 0 & 2 & -1 \end{pmatrix}.
\]
For an arbitrary word $w = s_{i_1} \cdots s_{i_k} \in W$, the representation matrix on $\Lambda$ is the product $B_w = B_{i_k} \cdots B_{i_1}$.

Conjugacy classes in $W$ are uniquely parameterized by signed cycle types, or equivalently by pairs of partitions $(\lambda^+, \lambda^-)$ tracking the lengths of positive cycles (containing an even number of sign flips) and negative cycles (containing an odd number of sign flips), such that $|\lambda^+| + |\lambda^-| = 3$. There are $10$ conjugacy classes in total. In Table \ref{tab:SO7_action}  we collect minimal Coxeter length representatives $w$ for each of the $10$ conjugacy classes, their minimal words\footnote{The minimal words for $w\in W$ have the Coxeter length of $w$, equal to the number of positive roots which transform to negative roots under the action of $w$ in the root system. Observe that the element $-I_3$ takes every positive root to its negative, therefore it transforms the $9$ positive roots of $B_3$ into negative ones and its minimal word has length $9$.} in the simple reflections (where $w = s_{i_1} \dots s_{i_k}$ indicates $s_{i_1}$ acts first), their signed cycle types $(\lambda^+, \lambda^-)$, and their matrix actions $B_w$.
\begin{table}[H]
\tiny
    \centering
    \setlength{\tabcolsep}{3pt}
    \begin{tabular}{|c|c|c|c|c|c|c|}
    \hline
        $[w]\in\Conj(W)$ & $|[w]|$ & $w = (\sigma, \boldsymbol{\epsilon})$ & \text{Signed cycle type } $(\lambda^+,\lambda^-)$& \text{Action on } $(x_1, x_2, x_3)$ & \text{Word} & $B_{w}$ \\\hline
        $(\text{i})$ & $1$ & $(e, (+++))$ & $((1)^+(2)^+(3)^+, \emptyset)$ & $(x_1, x_2, x_3)$ & $1$ & $\begin{pmatrix} 1 & 0 & 0 \\ 0 & 1 & 0 \\ 0 & 0 & 1 \end{pmatrix}$ \\\hline
        $(\text{ii})$ & $3$ & $(e, (++-))$ & $((1)^+(2)^+, (3)^-)$ & $(x_1, x_2, -x_3)$ & $s_3$ & $\begin{pmatrix} 1 & 0 & 0 \\ 0 & 1 & 0 \\ 0 & 2 & -1 \end{pmatrix}$ \\\hline
        $(\text{iii})$ & $3$ & $(e, (+--))$ & $((1)^+, (2)^-(3)^-)$ & $(x_1, -x_2, -x_3)$ & $s_2 s_3 s_2 s_3$ & $\begin{pmatrix} 1 & 0 & 0 \\ 2 & -1 & 0 \\ 2 & 0 & -1 \end{pmatrix}$ \\\hline
        $(\text{iv})$ & $1$ & $(e, (---))$ & $(\emptyset, (1)^-(2)^-(3)^-)$ & $(-x_1, -x_2, -x_3)$ & $s_1 s_2 s_1 s_3 s_2 s_1 s_3 s_2 s_3$ & $\begin{pmatrix} -1 & 0 & 0 \\ 0 & -1 & 0 \\ 0 & 0 & -1 \end{pmatrix}$ \\\hline
        $(\text{v})$ & $6$ & $((12), (+++))$ & $((12)^+(3)^+, \emptyset)$ & $(x_2, x_1, x_3)$ & $s_1$ & $\begin{pmatrix} -1 & 1 & 0 \\ 0 & 1 & 0 \\ 0 & 0 & 1 \end{pmatrix}$ \\\hline
        $(\text{vi})$ & $6$ & $((12), (++-))$ & $((12)^+, (3)^-)$ & $(x_2, x_1, -x_3)$ & $s_1 s_3$ & $\begin{pmatrix} -1 & 1 & 0 \\ 0 & 1 & 0 \\ 0 & 2 & -1 \end{pmatrix}$ \\\hline
        $(\text{vii})$ & $6$ & $((23), (+-+))$ & $((1)^+, (23)^-)$ & $(x_1, x_3, -x_2)$ & $s_2 s_3$ & $\begin{pmatrix} 1 & 0 & 0 \\ 1 & -1 & 1 \\ 2 & -2 & 1 \end{pmatrix}$ \\\hline
        $(\text{viii})$ & $6$ & $((12), (-+-))$ & $(\emptyset, (12)^-(3)^-)$ & $(x_2, -x_1, -x_3)$ & $s_1 s_2 s_3 s_2 s_3$ & $\begin{pmatrix} -1 & 1 & 0 \\ -2 & 1 & 0 \\ -2 & 2 & -1 \end{pmatrix}$ \\\hline
        $(\text{ix})$ & $8$ & $((123), (+++))$ & $((123)^+, \emptyset)$ & $(x_3, x_1, x_2)$ & $s_2 s_1$ & $\begin{pmatrix} 0 & -1 & 1 \\ 1 & -1 & 1 \\ 0 & 0 & 1 \end{pmatrix}$ \\\hline
        $(\text{x})$ & $8$ & $((123), (++-))$ & $(\emptyset, (123)^-)$ & $(-x_3, x_1, x_2)$ & $s_3 s_2 s_1$ & $\begin{pmatrix} 0 & 1 & -1 \\ 1 & 1 & -1 \\ 0 & 2 & -1 \end{pmatrix}$ \\\hline
    \end{tabular}
    \caption{Actions on $\Lambda$ of representatives of conjugacy classes in $W(B_3)$, for $\mathrm{SO}_7$.}
    \label{tab:SO7_action}
\end{table}
\vspace{-0.5cm}
With this information, we can directly compute the stringy mixed Hodge polynomials of $T_{\SO_7}^r/W$ using (\ref{eq:magic-formula}). Table \ref{tab:SO7_main} summarizes the computations involved in the calculation of the compactly supported stringy mixed Hodge polynomial $\mu^{\str}_c(T_{\SO_7}^r/W)$, for the $10$ conjugacy classes in $W \cong S_3 \ltimes (\mathbb{Z}_2)^3$.
\vspace{-0.5cm}
\begin{table}[H]
\tiny
\centering
\setlength{\tabcolsep}{4pt}
\renewcommand{\arraystretch}{1.25}
\begin{tabular}{|c|c|c|c|c|c|c|c|c|c|}
\hline
$[w]$ & $w = (\sigma, \boldsymbol{\epsilon})$ & $|C(w)|$ & $I-B_w$ & $A_w = \text{im}(I-B_w)$ & $\text{rk}\,A_w$ & $K_w = \Lambda / A_w$ & $s_w$ & $\Phi_w$ & $D_w$ \\ \hline

$(\text{i})$ & $(e, (+++))$ & $48$ & 
$\begin{pmatrix} 0 & 0 & 0 \\ 0 & 0 & 0 \\ 0 & 0 & 0 \end{pmatrix}$ & 
$0$ & $0$ & 
\begin{tabular}{@{}c@{}} $\mathbb{Z}\chi_1 \oplus \mathbb{Z}\chi_2$ \\ $\oplus\,\mathbb{Z}\chi_3 \cong \mathbb{Z}^3$ \end{tabular} & 
$3$ & 
\begin{tabular}{@{}c@{}} $\mathbb{Z}\chi_1 \oplus \mathbb{Z}\chi_2$ \\ $\oplus\,\mathbb{Z}\chi_3 \cong \mathbb{Z}^3$ \end{tabular} & 
$0$ \\ \hline

$(\text{ii})$ &$(e, (++-))$ & $16$ & 
$\begin{pmatrix} 0 & 0 & 0 \\ 0 & 0 & 0 \\ 0 & -2 & 2 \end{pmatrix}$ & 
\begin{tabular}{@{}c@{}} $\mathbb{Z}(2\chi_3) \cong \mathbb{Z}$ \end{tabular} & 
$1$ & 
\begin{tabular}{@{}c@{}} $\mathbb{Z}\chi_1 \oplus \mathbb{Z}\chi_2 \oplus \frac{\mathbb{Z}\chi_3}{\mathbb{Z}(2\chi_3)}$ \\ $\cong \mathbb{Z}^2 \oplus \mathbb{Z}_2$ \end{tabular} & 
$2$ & 
\begin{tabular}{@{}c@{}} $\mathbb{Z}\chi_1 \oplus \mathbb{Z}\chi_2$ \\ $\cong \mathbb{Z}^2$ \end{tabular} & 
$\frac{\mathbb{Z}\chi_3}{\mathbb{Z}(2\chi_3)} \cong \mathbb{Z}_2$ \\ \hline

$(\text{iii})$ &$(e, (+--))$ & $16$ & 
$\begin{pmatrix} 0 & 0 & 0 \\ -2 & 2 & 0 \\ -2 & 0 & 2 \end{pmatrix}$ & 
\begin{tabular}{@{}c@{}} $\mathbb{Z}(2\chi_2)\oplus\,\mathbb{Z}(2\chi_3) $ \\ $\cong \mathbb{Z}^2$ \end{tabular} & 
$2$ & 
\begin{tabular}{@{}c@{}} $\mathbb{Z}\chi_1 \oplus \frac{\mathbb{Z}\chi_2}{\mathbb{Z}(2\chi_2)}$ \\ $\oplus\,\frac{\mathbb{Z}\chi_3}{\mathbb{Z}(2\chi_3)} \cong \mathbb{Z} \oplus (\mathbb{Z}_2)^2$ \end{tabular} & 
$1$ & 
$\mathbb{Z}\chi_1 \cong \mathbb{Z}$ & 
\begin{tabular}{@{}c@{}} $\frac{\mathbb{Z}\chi_2}{\mathbb{Z}(2\chi_2)} \oplus \frac{\mathbb{Z}\chi_3}{\mathbb{Z}(2\chi_3)}$ \\ $\cong (\mathbb{Z}_2)^2$ \end{tabular} \\ \hline

$(\text{iv})$ &$(e, (---))$ & $48$ & 
$\begin{pmatrix} 2 & 0 & 0 \\ 0 & 2 & 0 \\ 0 & 0 & 2 \end{pmatrix}$ & 
\begin{tabular}{@{}c@{}} $\mathbb{Z}(2\chi_1) \oplus \mathbb{Z}(2\chi_2)$ \\ $\oplus\,\mathbb{Z}(2\chi_3) \cong \mathbb{Z}^3$ \end{tabular} & 
$3$ & 
\begin{tabular}{@{}c@{}} $\frac{\mathbb{Z}\chi_1}{\mathbb{Z}(2\chi_1)} \oplus \frac{\mathbb{Z}\chi_2}{\mathbb{Z}(2\chi_2)}$ \\ $\oplus\,\frac{\mathbb{Z}\chi_3}{\mathbb{Z}(2\chi_3)} \cong (\mathbb{Z}_2)^3$ \end{tabular} & 
$0$ & 
$0$ & 
\begin{tabular}{@{}c@{}} $\frac{\mathbb{Z}\chi_1}{\mathbb{Z}(2\chi_1)} \oplus \frac{\mathbb{Z}\chi_2}{\mathbb{Z}(2\chi_2)}$ \\ $\oplus\,\frac{\mathbb{Z}\chi_3}{\mathbb{Z}(2\chi_3)} \cong (\mathbb{Z}_2)^3$ \end{tabular} \\ \hline

$(\text{v})$ &$((12), (+++))$ & $8$ & 
$\begin{pmatrix} 2 & -1 & 0 \\ 0 & 0 & 0 \\ 0 & 0 & 0 \end{pmatrix}$ & $\mathbb{Z}\chi_1\cong \mathbb{Z}$ & 
$1$ & 
\begin{tabular}{@{}c@{}} $\mathbb{Z}\chi_2 \oplus \mathbb{Z}\chi_3$ \\ $\cong \mathbb{Z}^2$ \end{tabular} & 
$2$ & 
\begin{tabular}{@{}c@{}} $\mathbb{Z}\chi_2 \oplus \mathbb{Z}\chi_3$ \\ $\cong \mathbb{Z}^2$ \end{tabular} & 
$0$ \\ \hline

$(\text{vi})$ &$((12), (++-))$ & $8$ & 
$\begin{pmatrix} 2 & -1 & 0 \\ 0 & 0 & 0 \\ 0 & -2 & 2 \end{pmatrix}$ & 
$\mathbb{Z}\chi_1\oplus \mathbb{Z}(2\chi_3)\cong \mathbb{Z}^2$ & 
$2$ & 
\begin{tabular}{@{}c@{}} $\mathbb{Z}\chi_2 \oplus \frac{\mathbb{Z}\chi_3}{\mathbb{Z}(2\chi_3)}$ \\ $\cong \mathbb{Z} \oplus \mathbb{Z}_2$ \end{tabular} & 
$1$ & 
$\mathbb{Z}\chi_2 \cong \mathbb{Z}$ & 
$\frac{\mathbb{Z}\chi_3}{\mathbb{Z}(2\chi_3)} \cong \mathbb{Z}_2$ \\ \hline

$(\text{vii})$ &$((23), (+-+))$ & $8$ & 
$\begin{pmatrix} 0 & 0 & 0 \\ -1 & 2 & -1 \\ -2 & 2 & 0 \end{pmatrix}$ & 
$\mathbb{Z}\chi_2\oplus \mathbb{Z}(2\chi_3)\cong \mathbb{Z}^2$ & 
$2$ & 
\begin{tabular}{@{}c@{}} $\mathbb{Z}\chi_1 \oplus \frac{\mathbb{Z}\chi_3}{\mathbb{Z}(2\chi_3)}$ \\ $\cong \mathbb{Z} \oplus \mathbb{Z}_2$ \end{tabular} & 
$1$ & 
$\mathbb{Z}\chi_1 \cong \mathbb{Z}$ & 
$\frac{\mathbb{Z}\chi_3}{\mathbb{Z}(2\chi_3)} \cong \mathbb{Z}_2$ \\ \hline

$(\text{viii})$ &$((12), (-+-))$ & $8$ & 
$\begin{pmatrix} 2 & -1 & 0 \\ 2 & 0 & 0 \\ 2 & -2 & 2 \end{pmatrix}$ & 
\begin{tabular}{@{}c@{}} $\mathbb{Z}\chi_1\oplus\mathbb{Z}(2\chi_2)\oplus$ \\ $\mathbb{Z}(2\chi_3)\cong \mathbb{Z}^3$ \end{tabular} & 
$3$ & 
\begin{tabular}{@{}c@{}} $\frac{\mathbb{Z}\chi_2}{\mathbb{Z}(2\chi_2)} \oplus \frac{\mathbb{Z}\chi_3}{\mathbb{Z}(2\chi_3)}$ \\ $\cong (\mathbb{Z}_2)^2$ \end{tabular} & 
$0$ & 
$0$ & 
\begin{tabular}{@{}c@{}} $\frac{\mathbb{Z}\chi_2}{\mathbb{Z}(2\chi_2)} \oplus \frac{\mathbb{Z}\chi_3}{\mathbb{Z}(2\chi_3)}$ \\ $\cong (\mathbb{Z}_2)^2$ \end{tabular} \\ \hline

$(\text{ix})$ &$((123), (+++))$ & $6$ & 
$\begin{pmatrix} 1 & 1 & -1 \\ -1 & 2 & -1 \\ 0 & 0 & 0 \end{pmatrix}$ & 
$\mathbb{Z}\chi_1\oplus \mathbb{Z}\chi_2 \cong \mathbb{Z}^2$ & 
$2$ & 
$\mathbb{Z}\chi_3 \cong \mathbb{Z}$ & 
$1$ & 
$\mathbb{Z}\chi_3 \cong \mathbb{Z}$ & 
$0$ \\ \hline

$(\text{x})$ &$((123), (++-))$ & $6$ & 
$\begin{pmatrix} 1 & -1 & 1 \\ -1 & 0 & 1 \\ 0 & -2 & 2 \end{pmatrix}$ & 
\begin{tabular}{@{}c@{}} $\mathbb{Z}\chi_1\oplus\mathbb{Z}\chi_2\oplus$\\
$\mathbb{Z}(2\chi_3)\cong \mathbb{Z}^3$ \end{tabular} & 
$3$ & 
$\frac{\mathbb{Z}\chi_3}{\mathbb{Z}(2\chi_3)}\cong \mathbb{Z}_2$ & 
$0$ & 
$0$ & 
$\frac{\mathbb{Z}\chi_3}{\mathbb{Z}(2\chi_3)}\cong \mathbb{Z}_2$ \\ \hline
\end{tabular}
\caption{Invariants of each conjugacy class in Theorem \ref{thm:algorithm_CV} for $\mathrm{SO}_7$.}
\label{tab:SO7_main}
\end{table}

Now, let us compute the invariants we need for each conjugacy class separately.

\begin{itemize}
    \item[(i)] For the conjugacy class of $w=(e, (+++))$, the centralizer is $C(w) = W$, with $|C(w)| = 48$. We have $\Phi_w \cong \mathbb{Z}^3$ and $D_w = 0$. Computations in Table \ref{tab:SO7_i} yield
\begin{align*}
(i) &= \frac{t^{3r}}{48} \Big[ 
(tuv+1)^{3r} 
+ 9(tuv-1)^r (tuv+1)^{2r} 
+ 9(tuv-1)^{2r} (tuv+1)^r 
+ (tuv-1)^{3r} \\
&\quad + 6(tuv+1)^r (t^2 u^2 v^2 + 1)^r 
+ 6(tuv-1)^r (t^2 u^2 v^2 + 1)^r + 8(tuv+1)^r (t^2 u^2 v^2 - tuv + 1)^r 
+ 8(t^3 u^3 v^3 - 1)^r 
\Big]
\end{align*}
\begin{table}[H]
\tiny
    \centering
    \setlength{\tabcolsep}{3pt}
    \begin{tabular}{|c|c|c|c|c|c|c|}
    \hline
        $[g]\in\textup{Conj}(C(w))$ & $c_w(g)$ & $\tau_{w}(g):D_w\rightarrow D_w$ & \begin{tabular}{@{}c@{}} $n_w(g)=$\\ $\left|\frac{D_w}{\im(I_{D_w}-\tau_w(g))}\right|$\end{tabular} & $\phi_w(g):\Phi_w\rightarrow \Phi_w$ & $tuvI_{\Phi_w}+\phi_w(g)$ & $\det(tuvI_{\Phi_w}+\phi_w(g))$ \\\hline
        
        $[ (e, (+++)) ]$ & $1$ & $(1)$ & $1$ & 
        $\begin{pmatrix} 1 & 0 & 0 \\ 0 & 1 & 0 \\ 0 & 0 & 1 \end{pmatrix}$ & 
        $\begin{pmatrix} tuv+1 & 0 & 0 \\ 0 & tuv+1 & 0 \\ 0 & 0 & tuv+1 \end{pmatrix}$ & 
        $(tuv+1)^3$ \\\hline
        
        $[ (e, (++-)) ]$ & $3$ & $(1)$ & $1$ & 
        $\begin{pmatrix} 1 & 0 & 0 \\ 0 & 1 & 0 \\ 0 & 2 & -1 \end{pmatrix}$ & 
        $\begin{pmatrix} tuv+1 & 0 & 0 \\ 0 & tuv+1 & 0 \\ 0 & 2 & tuv-1 \end{pmatrix}$ & 
        $(tuv-1)(tuv+1)^2$ \\\hline
        
        $[ (e, (+--)) ]$ & $3$ & $(1)$ & $1$ & 
        $\begin{pmatrix} 1 & 0 & 0 \\ 2 & -1 & 0 \\ 2 & 0 & -1 \end{pmatrix}$ & 
        $\begin{pmatrix} tuv+1 & 0 & 0 \\ 2 & tuv-1 & 0 \\ 2 & 0 & tuv-1 \end{pmatrix}$ & 
        $(tuv-1)^2(tuv+1)$ \\\hline
        
        $[ (e, (---)) ]$ & $1$ & $(1)$ & $1$ & 
        $\begin{pmatrix} -1 & 0 & 0 \\ 0 & -1 & 0 \\ 0 & 0 & -1 \end{pmatrix}$ & 
        $\begin{pmatrix} tuv-1 & 0 & 0 \\ 0 & tuv-1 & 0 \\ 0 & 0 & tuv-1 \end{pmatrix}$ & 
        $(tuv-1)^3$ \\\hline
        
        $[ ((12), (+++)) ]$ & $6$ & $(1)$ & $1$ & 
        $\begin{pmatrix} -1 & 1 & 0 \\ 0 & 1 & 0 \\ 0 & 0 & 1 \end{pmatrix}$ & 
        $\begin{pmatrix} tuv-1 & 1 & 0 \\ 0 & tuv+1 & 0 \\ 0 & 0 & tuv+1 \end{pmatrix}$ & 
        $(tuv-1)(tuv+1)^2$ \\\hline
        
        $[ ((12), (++-)) ]$ & $6$ & $(1)$ & $1$ & 
        $\begin{pmatrix} -1 & 1 & 0 \\ 0 & 1 & 0 \\ 0 & 2 & -1 \end{pmatrix}$ & 
        $\begin{pmatrix} tuv-1 & 1 & 0 \\ 0 & tuv+1 & 0 \\ 0 & 2 & tuv-1 \end{pmatrix}$ & 
        $(tuv-1)^2(tuv+1)$ \\\hline
        
        $[ ((23), (+-+)) ]$ & $6$ & $(1)$ & $1$ & 
        $\begin{pmatrix} 1 & 0 & 0 \\ 1 & -1 & 1 \\ 2 & -2 & 1 \end{pmatrix}$ & 
        $\begin{pmatrix} tuv+1 & 0 & 0 \\ 1 & tuv-1 & 1 \\ 2 & -2 & tuv+1 \end{pmatrix}$ & 
        $(tuv+1)(t^2 u^2 v^2 + 1)$ \\\hline
        
        $[ ((12), (-+-)) ]$ & $6$ & $(1)$ & $1$ & 
        $\begin{pmatrix} -1 & 1 & 0 \\ -2 & 1 & 0 \\ -2 & 2 & -1 \end{pmatrix}$ & 
        $\begin{pmatrix} tuv-1 & 1 & 0 \\ -2 & tuv+1 & 0 \\ -2 & 2 & tuv-1 \end{pmatrix}$ & 
        $(tuv-1)(t^2 u^2 v^2 + 1)$ \\\hline
        
        $[ ((123), (+++)) ]$ & $8$ & $(1)$ & $1$ & 
        $\begin{pmatrix} 0 & -1 & 1 \\ 1 & -1 & 1 \\ 0 & 0 & 1 \end{pmatrix}$ & 
        $\begin{pmatrix} tuv & -1 & 1 \\ 1 & tuv-1 & 1 \\ 0 & 0 & tuv+1 \end{pmatrix}$ & 
        $(tuv+1)(t^2 u^2 v^2 - tuv + 1)$ \\\hline
        
        $[ ((123), (++-)) ]$ & $8$ & $(1)$ & $1$ & 
        $\begin{pmatrix} 0 & 1 & -1 \\ 1 & 1 & -1 \\ 0 & 2 & -1 \end{pmatrix}$ & 
        $\begin{pmatrix} tuv & 1 & -1 \\ 1 & tuv+1 & -1 \\ 0 & 2 & tuv-1 \end{pmatrix}$ & 
        $(t^3 u^3 v^3 - 1)$ \\\hline
    \end{tabular}
    \caption{Computations for $\mathrm{SO}_7$ (i), conjugacy class of $w=(e, (+++))$ in the basis $\{\chi_1,\chi_2,\chi_3\}$ of $\Phi_w\cong \mathbb{Z}^3$}
    \label{tab:SO7_i}
\end{table}
\item[(ii)] For the conjugacy class of $w=(e, (++-))$, the centralizer is $C(w) \cong W(B_2) \times W(B_1)\cong (S_2\ltimes \mathbb{Z}_2^2)\times \mathbb{Z}_2$, with $|C(w)| = 16$. We have $\Phi_w \cong \mathbb{Z}^2$ and $D_w \cong \mathbb{Z}_2$. Observe that in the direct product $C(w) \cong W(B_2) \times W(B_1)\cong (S_2\ltimes \mathbb{Z}_2^2)\times \mathbb{Z}_2$, conjugacy classes of $C(w)$ are given by cartesian products of conjugacy classes of each factor. Computations in Table \ref{tab:SO7_ii} yield
\begin{align*}
(ii)
&= \frac{t^{3r}(uv)^{\frac{r}{2}}2^r}{8} \Big[ 
(tuv+1)^{2r} 
+ 4(tuv-1)^r (tuv+1)^r 
+ (tuv-1)^{2r} 
+ 2(t^2 u^2 v^2 + 1)^r 
\Big]
\end{align*}
\begin{table}[H]
\tiny
    \centering
    \setlength{\tabcolsep}{3pt}
    \begin{tabular}{|c|c|c|c|c|c|c|}
    \hline
        $[g]\in\textup{Conj}(C(w))$ & $c_w(g)$ & $\tau_{w}(g):D_w\rightarrow D_w$ & \begin{tabular}{@{}c@{}} $n_w(g)=$\\ $\left|\frac{D_w}{\im(I_{D_w}-\tau_w(g))}\right|$\end{tabular} & $\phi_w(g):\Phi_w\rightarrow \Phi_w$ & $tuvI_{\Phi_w}+\phi_w(g)$ & $\det(tuvI_{\Phi_w}+\phi_w(g))$ \\\hline
        
        $[ (e, (+++)) ]$ & $1$ & $(1)\pmod 2$ & $2$ & 
        $\begin{pmatrix} 1 & 0 \\ 0 & 1 \end{pmatrix}$ & 
        $\begin{pmatrix} tuv+1 & 0 \\ 0 & tuv+1 \end{pmatrix}$ & 
        $(tuv+1)^2$ \\\hline
        
        $[ (e, (++-)) ]$ & $1$ & $(1)\pmod 2$ & $2$ & 
        $\begin{pmatrix} 1 & 0 \\ 0 & 1 \end{pmatrix}$ & 
        $\begin{pmatrix} tuv+1 & 0 \\ 0 & tuv+1 \end{pmatrix}$ & 
        $(tuv+1)^2$ \\\hline
        
        $[ (e, (+-+)) ]$ & $2$ & $(1)\pmod 2$ & $2$ & 
        $\begin{pmatrix} 1 & 0 \\ 2 & -1 \end{pmatrix}$ & 
        $\begin{pmatrix} tuv+1 & 0 \\ 2 & tuv-1 \end{pmatrix}$ & 
        $(tuv-1)(tuv+1)$ \\\hline
        
        $[ (e, (+--)) ]$ & $2$ & $(1)\pmod 2$ & $2$ & 
        $\begin{pmatrix} 1 & 0 \\ 2 & -1 \end{pmatrix}$ & 
        $\begin{pmatrix} tuv+1 & 0 \\ 2 & tuv-1 \end{pmatrix}$ & 
        $(tuv-1)(tuv+1)$ \\\hline
        
        $[ (e, (--+)) ]$ & $1$ & $(1)\pmod 2$ & $2$ & 
        $\begin{pmatrix} -1 & 0 \\ 0 & -1 \end{pmatrix}$ & 
        $\begin{pmatrix} tuv-1 & 0 \\ 0 & tuv-1 \end{pmatrix}$ & 
        $(tuv-1)^2$ \\\hline
        
        $[ (e, (---)) ]$ & $1$ & $(1)\pmod 2$ & $2$ & 
        $\begin{pmatrix} -1 & 0 \\ 0 & -1 \end{pmatrix}$ & 
        $\begin{pmatrix} tuv-1 & 0 \\ 0 & tuv-1 \end{pmatrix}$ & 
        $(tuv-1)^2$ \\\hline
        
        $[ ((12), (+++)) ]$ & $2$ & $(1)\pmod 2$ & $2$ & 
        $\begin{pmatrix} -1 & 1 \\ 0 & 1 \end{pmatrix}$ & 
        $\begin{pmatrix} tuv-1 & 1 \\ 0 & tuv+1 \end{pmatrix}$ & 
        $(tuv-1)(tuv+1)$ \\\hline
        
        $[ ((12), (++-)) ]$ & $2$ & $(1)\pmod 2$ & $2$ & 
        $\begin{pmatrix} -1 & 1 \\ 0 & 1 \end{pmatrix}$ & 
        $\begin{pmatrix} tuv-1 & 1 \\ 0 & tuv+1 \end{pmatrix}$ & 
        $(tuv-1)(tuv+1)$ \\\hline
        
        $[ ((12), (+-+)) ]$ & $2$ & $(1)\pmod 2$ & $2$ & 
        $\begin{pmatrix} 1 & -1 \\ 2 & -1 \end{pmatrix}$ & 
        $\begin{pmatrix} tuv+1 & -1 \\ 2 & tuv-1 \end{pmatrix}$ & 
        $t^2 u^2 v^2 + 1$ \\\hline
        
        $[ ((12), (+--)) ]$ & $2$ & $(1)\pmod 2$ & $2$ & 
        $\begin{pmatrix} 1 & -1 \\ 2 & -1 \end{pmatrix}$ & 
        $\begin{pmatrix} tuv+1 & -1 \\ 2 & tuv-1 \end{pmatrix}$ & 
        $t^2 u^2 v^2 + 1$ \\\hline
    \end{tabular}
    \caption{Computations for $\mathrm{SO}_7$ (ii), conjugacy class of $w=(e, (++-))$ in the basis $\{\chi_1, \chi_2\}$ of $\Phi_w\cong \mathbb{Z}^2$ and $\{\chi_3\}$ for $D_w\cong \mathbb{Z}^2$}
    \label{tab:SO7_ii}
\end{table}

\item[(iii)] For the conjugacy class of $w=(e, (+--))$, the centralizer is also $C(w) \cong W(B_1) \times W(B_2)\cong \mathbb{Z}_2\times (S_2\ltimes \mathbb{Z}_2^2)$, with $|C(w)| = 16$. We have $\Phi_w \cong \mathbb{Z}$, $D_w \cong (\mathbb{Z}_2)^2$ and, again, conjugacy classes of $C(w)$ are given by cartesian products of conjugacy classes of each factor. Computations in Table \ref{tab:SO7_iii} yield
\vspace{-0.2cm}
\begin{align*}
(iii) &= \frac{t^{3r}(uv)^{r}(4^r+2^r)}{4} \Big[ 
(tuv+1)^r +(tuv-1)^r 
\Big]
\end{align*}
\vspace{-0.5cm}
\begin{table}[H]
\tiny
    \centering
    \setlength{\tabcolsep}{3pt}
    \begin{tabular}{|c|c|c|c|c|c|c|}
    \hline
        $[g]\in\textup{Conj}(C(w))$ & $c_w(g)$ & $\tau_{w}(g):D_w\rightarrow D_w$ & \begin{tabular}{@{}c@{}} $n_w(g)=$\\ $\left|\frac{D_w}{\im(I_{D_w}-\tau_w(g))}\right|$\end{tabular} & $\phi_w(g):\Phi_w\rightarrow \Phi_w$ & $tuvI_{\Phi_w}+\phi_w(g)$ & $\det(tuvI_{\Phi_w}+\phi_w(g))$ \\\hline
        
        $[ (e, (+++)) ]$ & $1$ & $\begin{pmatrix} 1 & 0 \\ 0 & 1 \end{pmatrix}\pmod 2$ & $4$ & $(1)$ & $(tuv+1)$ & $tuv+1$ \\\hline
        
        $[ (e, (-++)) ]$ & $1$ & $\begin{pmatrix} 1 & 0 \\ 0 & 1 \end{pmatrix}\pmod 2$ & $4$ & $(-1)$ & $(tuv-1)$ & $tuv-1$ \\\hline
        
        $[ (e, (++-)) ]$ & $2$ & $\begin{pmatrix} 1 & 0 \\ 0 & 1 \end{pmatrix}\pmod 2$ & $4$ & $(1)$ & $(tuv+1)$ & $tuv+1$ \\\hline
        
        $[ (e, (-+-)) ]$ & $2$ & $\begin{pmatrix} 1 & 0 \\ 0 & 1 \end{pmatrix}\pmod 2$ & $4$ & $(-1)$ & $(tuv-1)$ & $tuv-1$ \\\hline
        
        $[ (e, (+--)) ]$ & $1$ & $\begin{pmatrix} 1 & 0 \\ 0 & 1 \end{pmatrix}\pmod 2$ & $4$ & $(1)$ & $(tuv+1)$ & $tuv+1$ \\\hline
        
        $[ (e, (---)) ]$ & $1$ & $\begin{pmatrix} 1 & 0 \\ 0 & 1 \end{pmatrix}\pmod 2$ & $4$ & $(-1)$ & $(tuv-1)$ & $tuv-1$ \\\hline
        
        $[ ((23), (+++)) ]$ & $2$ & $\begin{pmatrix} 1 & 1 \\ 0 & 1 \end{pmatrix}\pmod 2$ & $2$ & $(1)$ & $(tuv+1)$ & $tuv+1$ \\\hline
        
        $[ ((23), (-++)) ]$ & $2$ & $\begin{pmatrix} 1 & 1 \\ 0 & 1 \end{pmatrix}\pmod 2$ & $2$ & $(-1)$ & $(tuv-1)$ & $tuv-1$ \\\hline
        
        $[ ((23), (++-)) ]$ & $2$ & $\begin{pmatrix} 1 & 1 \\ 0 & 1 \end{pmatrix}\pmod 2$ & $2$ & $(1)$ & $(tuv+1)$ & $tuv+1$ \\\hline
        
        $[ ((23), (-+-)) ]$ & $2$ & $\begin{pmatrix} 1 & 1 \\ 0 & 1 \end{pmatrix}\pmod 2$ & $2$ & $(-1)$ & $(tuv-1)$ & $tuv-1$ \\\hline
    \end{tabular}
    \caption{Computations for $\mathrm{SO}_7$ (iii), conjugacy class of $w=(e, (+--))$ in the basis $\{\chi_1\}$ of $\Phi_w\cong \mathbb{Z}$ and $\{\chi_2, \chi_3\}$ for $D_w\cong (\mathbb{Z}_2)^2$}
    \label{tab:SO7_iii}
\end{table}
\item[(iv)] For the conjugacy class of the central element $w=(e, (---))$, the centralizer is the whole $C(w) = W$, with $|C(w)| = 48$. We have $\Phi_w = 0$ and $D_w \cong (\mathbb{Z}_2)^3$. Computations in Table \ref{tab:SO7_iv} yield
\vspace{-0.2cm}
\begin{align*}
(iv) &= \frac{t^{3r}(uv)^{\frac{3}{2}r}}{6} \Big[ 8^r+3\cdot 4^r+2\cdot 2^r
\Big] 
\end{align*}
\vspace{-0.5cm}
\begin{table}[H]
\tiny
    \centering
    \setlength{\tabcolsep}{3pt}
    \begin{tabular}{|c|c|c|c|c|c|c|}
    \hline
        $[g]\in\textup{Conj}(C(w))$ & $c_w(g)$ & $\tau_{w}(g):D_w\rightarrow D_w$ & \begin{tabular}{@{}c@{}} $n_w(g)=$\\ $\left|\frac{D_w}{\im(I_{D_w}-\tau_w(g))}\right|$\end{tabular} & $\phi_w(g):\Phi_w\rightarrow \Phi_w$ & $tuvI_{\Phi_w}+\phi_w(g)$ & $\det(tuvI_{\Phi_w}+\phi_w(g))$ \\\hline
        
        $[ (e, (+++)) ]$ & $1$ & $\begin{pmatrix} 1 & 0 & 0 \\ 0 & 1 & 0 \\ 0 & 0 & 1 \end{pmatrix}\pmod 2$ & $8$ & $-$ & $-$ & $1$ \\\hline
        
        $[ (e, (++-)) ]$ & $3$ & $\begin{pmatrix} 1 & 0 & 0 \\ 0 & 1 & 0 \\ 0 & 0 & 1 \end{pmatrix}\pmod 2$ & $8$ & $-$ & $-$ & $1$ \\\hline
        
        $[ (e, (+--)) ]$ & $3$ & $\begin{pmatrix} 1 & 0 & 0 \\ 0 & 1 & 0 \\ 0 & 0 & 1 \end{pmatrix}\pmod 2$ & $8$ & $-$ & $-$ & $1$ \\\hline
        
        $[ (e, (---)) ]$ & $1$ & $\begin{pmatrix} 1 & 0 & 0 \\ 0 & 1 & 0 \\ 0 & 0 & 1 \end{pmatrix}\pmod 2$ & $8$ & $-$ & $-$ & $1$ \\\hline
        
        $[ ((12), (+++)) ]$ & $6$ & $\begin{pmatrix} 1 & 1 & 0 \\ 0 & 1 & 0 \\ 0 & 0 & 1 \end{pmatrix}\pmod 2$ & $4$ & $-$ & $-$ & $1$ \\\hline
        
        $[ ((12), (++-)) ]$ & $6$ & $\begin{pmatrix} 1 & 1 & 0 \\ 0 & 1 & 0 \\ 0 & 0 & 1 \end{pmatrix}\pmod 2$ & $4$ & $-$ & $-$ & $1$ \\\hline
        
        $[ ((23), (+-+)) ]$ & $6$ & $\begin{pmatrix} 1 & 0 & 0 \\ 1 & 1 & 1 \\ 0 & 0 & 1 \end{pmatrix}\pmod 2$ & $4$ & $-$ & $-$ & $1$ \\\hline
        
        $[ ((12), (-+-)) ]$ & $6$ & $\begin{pmatrix} 1 & 1 & 0 \\ 0 & 1 & 0 \\ 0 & 0 & 1 \end{pmatrix}\pmod 2$ & $4$ & $-$ & $-$ & $1$ \\\hline
        
        $[ ((123), (+++)) ]$ & $8$ & $\begin{pmatrix} 0 & 1 & 1 \\ 1 & 1 & 1 \\ 0 & 0 & 1 \end{pmatrix}\pmod 2$ & $2$ & $-$ & $-$ & $1$ \\\hline
        
        $[ ((123), (++-)) ]$ & $8$ & $\begin{pmatrix} 0 & 1 & 1 \\ 1 & 1 & 1 \\ 0 & 0 & 1 \end{pmatrix}\pmod 2$ & $2$ & $-$ & $-$ & $1$ \\\hline
    \end{tabular}
    \caption{Computations for $\mathrm{SO}_7$ (iv), conjugacy class of $w=(e, (---))$ in the basis $\{\chi_1, \chi_2, \chi_3\}$ for $D_w\cong (\mathbb{Z}_2)^3$}
    \label{tab:SO7_iv}
\end{table}
\vspace{-0.5cm}
\item[(v)] For the conjugacy class of $w=((12), (+++))$, the centralizer is $C(w) \cong \mathbb{Z}_2 \times \mathbb{Z}_2 \times \mathbb{Z}_2$, with $|C(w)| = 8$, generated by the transposition $((12), (+++))$, the simultaneous sign flip for positions 1, 2, $(e, (--+))$, and the flip for position $3$, $(e, (++-))$. We have $\Phi_w \cong \mathbb{Z}^2$ and $D_w = 0$. Being $C(w)$ an abelian group its conjugacy classes are singletons. Computations in Table \ref{tab:SO7_v} yield
\vspace{-0.3cm}
\begin{align*}
(v) &= \frac{t^{3r}(uv)^{\frac{r}{2}}}{4} \Big[ 
(tuv+1)^{2r} 
+ 2(tuv-1)^r(tuv+1)^r 
+ (tuv-1)^{2r} 
\Big]
\end{align*}
\vspace{-0.5cm}
\begin{table}[H]
\tiny
    \centering
    \setlength{\tabcolsep}{3pt}
    \begin{tabular}{|c|c|c|c|c|c|c|}
    \hline
        $[g]\in\textup{Conj}(C(w))$ & $c_w(g)$ & $\tau_{w}(g):D_w\rightarrow D_w$ & \begin{tabular}{@{}c@{}} $n_w(g)=$\\ $\left|\frac{D_w}{\im(I_{D_w}-\tau_w(g))}\right|$\end{tabular} & $\phi_w(g):\Phi_w\rightarrow \Phi_w$ & $tuvI_{\Phi_w}+\phi_w(g)$ & $\det(tuvI_{\Phi_w}+\phi_w(g))$ \\\hline
        
        $[ (e, (+++)) ]$ & $1$ & $-$ & $1$ & 
        $\begin{pmatrix} 1 & 0 \\ 0 & 1 \end{pmatrix}$ & 
        $\begin{pmatrix} tuv+1 & 0 \\ 0 & tuv+1 \end{pmatrix}$ & 
        $(tuv+1)^2$ \\\hline
        
        $[ (e, (++-)) ]$ & $1$ & $-$ & $1$ & 
        $\begin{pmatrix} 1 & 0 \\ 2 & -1 \end{pmatrix}$ & 
        $\begin{pmatrix} tuv+1 & 0 \\ 2 & tuv-1 \end{pmatrix}$ & 
        $(tuv+1)(tuv-1)$ \\\hline
        
        $[ (e, (--+)) ]$ & $1$ & $-$ & $1$ & 
        $\begin{pmatrix} -1 & 0 \\ -2 & 1 \end{pmatrix}$ & 
        $\begin{pmatrix} tuv-1 & 0 \\ -2 & tuv+1 \end{pmatrix}$ & 
        $(tuv-1)(tuv+1)$ \\\hline
        
        $[ (e, (---)) ]$ & $1$ & $-$ & $1$ & 
        $\begin{pmatrix} -1 & 0 \\ 0 & -1 \end{pmatrix}$ & 
        $\begin{pmatrix} tuv-1 & 0 \\ 0 & tuv-1 \end{pmatrix}$ & 
        $(tuv-1)^2$ \\\hline
        
        $[ ((12), (+++)) ]$ & $1$ & $-$ & $1$ & 
        $\begin{pmatrix} 1 & 0 \\ 0 & 1 \end{pmatrix}$ & 
        $\begin{pmatrix} tuv+1 & 0 \\ 0 & tuv+1 \end{pmatrix}$ & 
        $(tuv+1)^2$ \\\hline
        
        $[ ((12), (++-)) ]$ & $1$ & $-$ & $1$ & 
        $\begin{pmatrix} 1 & 0 \\ 2 & -1 \end{pmatrix}$ & 
        $\begin{pmatrix} tuv+1 & 0 \\ 2 & tuv-1 \end{pmatrix}$ & 
        $(tuv+1)(tuv-1)$ \\\hline
        
        $[ ((12), (--+)) ]$ & $1$ & $-$ & $1$ & 
        $\begin{pmatrix} -1 & 0 \\ -2 & 1 \end{pmatrix}$ & 
        $\begin{pmatrix} tuv-1 & 0 \\ -2 & tuv+1 \end{pmatrix}$ & 
        $(tuv-1)(tuv+1)$ \\\hline
        
        $[ ((12), (---)) ]$ & $1$ & $-$ & $1$ & 
        $\begin{pmatrix} -1 & 0 \\ 0 & -1 \end{pmatrix}$ & 
        $\begin{pmatrix} tuv-1 & 0 \\ 0 & tuv-1 \end{pmatrix}$ & 
        $(tuv-1)^2$ \\\hline
    \end{tabular}
    \caption{Computations for $\mathrm{SO}_7$ (v), conjugacy class of $w=((12), (+++))$ in the basis $\{\chi_2, \chi_3\}$ for $\Phi_w\cong \mathbb{Z}^2$}
    \label{tab:SO7_v}
\end{table}
\vspace{-0.3cm}
\item[(vi)] For the conjugacy class of $w=((12), (++-))$, reasoning as in (v) the centralizer is $C(w) \cong \mathbb{Z}_2 \times \mathbb{Z}_2 \times \mathbb{Z}_2$, with $|C(w)| = 8$. We have $\Phi_w \cong \mathbb{Z}$ in the basis $\{\chi_2\}$ and $D_w \cong \mathbb{Z}_2$. Computations in Table \ref{tab:SO7_vi} yield
\vspace{-0.3cm}
\begin{align*}
(vi)
&= t^{3r}(uv)^r2^{r-1} \Big[ (tuv+1)^r + (tuv-1)^r \Big]
\end{align*}
\vspace{-0.5cm}
\begin{table}[H]
\tiny
    \centering
    \setlength{\tabcolsep}{3pt}
    \begin{tabular}{|c|c|c|c|c|c|c|}
    \hline
        $[g]\in\textup{Conj}(C(w))$ & $c_w(g)$ & $\tau_{w}(g):D_w\rightarrow D_w$ & \begin{tabular}{@{}c@{}} $n_w(g)=$\\ $\left|\frac{D_w}{\im(I_{D_w}-\tau_w(g))}\right|$\end{tabular} & $\phi_w(g):\Phi_w\rightarrow \Phi_w$ & $tuvI_{\Phi_w}+\phi_w(g)$ & $\det(tuvI_{\Phi_w}+\phi_w(g))$ \\\hline
        
        $[ (e, (+++)) ]$ & $1$ & $(1)\pmod 2$ & $2$ & 
        $(1)$ & 
        $(tuv+1)$ & 
        $(tuv+1)$ \\\hline
        
        $[ (e, (++-)) ]$ & $1$ & $(1)\pmod 2$ & $2$ & 
        $(1)$ & 
        $(tuv+1)$ & 
        $(tuv+1)$ \\\hline
        
        $[ (e, (--+)) ]$ & $1$ & $(1)\pmod 2$ & $2$ & 
        $(-1)$ & 
        $(tuv-1)$ & 
        $(tuv-1)$ \\\hline
        
        $[ (e, (---)) ]$ & $1$ & $(1)\pmod 2$ & $2$ & 
        $(-1)$ & 
        $(tuv-1)$ & 
        $(tuv-1)$ \\\hline
        
        $[ ((12), (+++)) ]$ & $1$ & $(1)\pmod 2$ & $2$ & 
        $(1)$ & 
        $(tuv+1)$ & 
        $(tuv+1)$ \\\hline
        
        $[ ((12), (++-)) ]$ & $1$ & $(1)\pmod 2$ & $2$ & 
        $(1)$ & 
        $(tuv+1)$ & 
        $(tuv+1)$ \\\hline
        
        $[ ((12), (--+)) ]$ & $1$ & $(1)\pmod 2$ & $2$ & 
        $(-1)$ & 
        $(tuv-1)$ & 
        $(tuv-1)$ \\\hline
        
        $[ ((12), (---)) ]$ & $1$ & $(1)\pmod 2$ & $2$ & 
        $(-1)$ & 
        $(tuv-1)$ & 
        $(tuv-1)$ \\\hline
    \end{tabular}
    \caption{Computations for $\mathrm{SO}_7$ (vi), conjugacy class of $w=((12), (++-))$ in the basis $\{\chi_2\}$ for $\Phi_w\cong \mathbb{Z}$ and $\{\chi_3\}$ for $D_w\cong \mathbb{Z}_2$}
    \label{tab:SO7_vi}
\end{table}
\vspace{-0.3cm}
\item[(vii)] For the conjugacy class of $w=((23), (+-+))$, the centralizer is $C(w) \cong \mathbb{Z}_2 \times \mathbb{Z}_4$, with $|C(w)| = 8$, generated by the first sign flip and the element $w$ of order $4$. We have $\Phi_w \cong \mathbb{Z}$ and $D_w = \mathbb{Z}_2$. Computations in Table \ref{tab:SO7_vii} yield
\begin{align*}
(vii) &= t^{3r}(uv)^r2^{r-1} \Big[ (tuv+1)^r + (tuv-1)^r \Big]
\end{align*}
\vspace{-0.5cm}
\begin{table}[H]
\tiny
    \centering
    \setlength{\tabcolsep}{3pt}
    \begin{tabular}{|c|c|c|c|c|c|c|}
    \hline
        $[g]\in\textup{Conj}(C(w))$ & $c_w(g)$ & $\tau_{w}(g):D_w\rightarrow D_w$ & \begin{tabular}{@{}c@{}} $n_w(g)=$\\ $\left|\frac{D_w}{\im(I_{D_w}-\tau_w(g))}\right|$\end{tabular} & $\phi_w(g):\Phi_w\rightarrow \Phi_w$ & $tuvI_{\Phi_w}+\phi_w(g)$ & $\det(tuvI_{\Phi_w}+\phi_w(g))$ \\\hline
        
        $[ (e, (+++)) ]$ & $1$ & $(1)\pmod 2$ & $2$ & 
        $(1)$ & 
        $(tuv+1)$ & 
        $(tuv+1)$ \\\hline
        
        $[ (e, (+--)) ]$ & $1$ & $(1)\pmod 2$ & $2$ & 
        $(1)$ & 
        $(tuv+1)$ & 
        $(tuv+1)$ \\\hline
        
        $[ (e, (-++)) ]$ & $1$ & $(1)\pmod 2$ & $2$ & 
        $(-1)$ & 
        $(tuv-1)$ & 
        $(tuv-1)$ \\\hline
        
        $[ (e, (---)) ]$ & $1$ & $(1)\pmod 2$ & $2$ & 
        $(-1)$ & 
        $(tuv-1)$ & 
        $(tuv-1)$ \\\hline
        
        $[ ((23), (++-)) ]$ & $1$ & $(1)\pmod 2$ & $2$ & 
        $(1)$ & 
        $(tuv+1)$ & 
        $(tuv+1)$ \\\hline
        
        $[ ((23), (+-+)) ]$ & $1$ & $(1)\pmod 2$ & $2$ & 
        $(1)$ & 
        $(tuv+1)$ & 
        $(tuv+1)$ \\\hline
        
        $[ ((23), (-+-)) ]$ & $1$ & $(1)\pmod 2$ & $2$ & 
        $(-1)$ & 
        $(tuv-1)$ & 
        $(tuv-1)$ \\\hline
        
        $[ ((23), (--+)) ]$ & $1$ & $(1)\pmod 2$ & $2$ & 
        $(-1)$ & 
        $(tuv-1)$ & 
        $(tuv-1)$ \\\hline
    \end{tabular}
    \caption{Computations for $\mathrm{SO}_7$ (vii), conjugacy class of $w=((23), (+-+))$ in the basis $\{\chi_1\}$ for $\Phi_w\cong \mathbb{Z}$ and $\{\chi_3\}$ for $D_w\cong \mathbb{Z}_2$}
    \label{tab:SO7_vii}
\end{table}
\vspace{-0.3cm}
\item[(viii)] For the conjugacy class of $w=((12), (-+-))$, similarly to (vii), the centralizer is $C(w) \cong \mathbb{Z}_4 \times \mathbb{Z}_2$, with $|C(w)| = 8$, generated by $w$ and the third sign flip. We have $\Phi_w = 0$ and $D_w \cong (\mathbb{Z}_2)^2$. Computations in Table \ref{tab:SO7_viii} yield
\begin{align*}
(viii) &= t^{3r}(uv)^{\frac{3}{2}r}4^r
\end{align*}
\begin{table}[H]
\tiny
    \centering
    \setlength{\tabcolsep}{3pt}
    \begin{tabular}{|c|c|c|c|c|c|c|}
    \hline
        $[g]\in\textup{Conj}(C(w))$ & $c_w(g)$ & $\tau_{w}(g):D_w\rightarrow D_w$ & \begin{tabular}{@{}c@{}} $n_w(g)=$\\ $\left|\frac{D_w}{\im(I_{D_w}-\tau_w(g))}\right|$\end{tabular} & $\phi_w(g):\Phi_w\rightarrow \Phi_w$ & $tuvI_{\Phi_w}+\phi_w(g)$ & $\det(tuvI_{\Phi_w}+\phi_w(g))$ \\\hline
        
        $[ (e, (+++)) ]$ & $1$ & $\begin{pmatrix} 1 & 0 \\ 0 & 1 \end{pmatrix}\pmod 2$ & $4$ & 
        $-$ & $-$ & $1$ \\\hline
        
        $[ (e, (--+)) ]$ & $1$ & $\begin{pmatrix} 1 & 0 \\ 0 & 1 \end{pmatrix}\pmod 2$ & $4$ & 
        $-$ & $-$ & $1$ \\\hline
        
        $[ (e, (++-)) ]$ & $1$ & $\begin{pmatrix} 1 & 0 \\ 0 & 1 \end{pmatrix}\pmod 2$ & $4$ & 
        $-$ & $-$ & $1$ \\\hline
        
        $[ (e, (---)) ]$ & $1$ & $\begin{pmatrix} 1 & 0 \\ 0 & 1 \end{pmatrix}\pmod 2$ & $4$ & 
        $-$ & $-$ & $1$ \\\hline
        
        $[ ((12), (+--)) ]$ & $1$ & $\begin{pmatrix} 1 & 0 \\ 0 & 1 \end{pmatrix}\pmod 2$ & $4$ & 
        $-$ & $-$ & $1$ \\\hline
        
        $[ ((12), (-+-)) ]$ & $1$ & $\begin{pmatrix} 1 & 0 \\ 0 & 1 \end{pmatrix}\pmod 2$ & $4$ & 
        $-$ & $-$ & $1$ \\\hline
        
        $[ ((12), (+-+)) ]$ & $1$ & $\begin{pmatrix} 1 & 0 \\ 0 & 1 \end{pmatrix}\pmod 2$ & $4$ & 
        $-$ & $-$ & $1$ \\\hline
        
        $[ ((12), (-++)) ]$ & $1$ & $\begin{pmatrix} 1 & 0 \\ 0 & 1 \end{pmatrix}\pmod 2$ & $4$ & 
        $-$ & $-$ & $1$ \\\hline
    \end{tabular}
    \caption{Computations for $\mathrm{SO}_7$ (viii), conjugacy class of $w=((12), (-+-))$ in the basis $\{\chi_2, \chi_3\}$ for $D_w\cong (\mathbb{Z}_2)^2$}
    \label{tab:SO7_viii}
\end{table}
\vspace{-0.3cm}
\item[(ix)] For the conjugacy class of $w=((123), (+++))$, the centralizer is $C(w) \cong \mathbb{Z}_3 \times \mathbb{Z}_2$, with $|C(w)| = 6$, generated by the $3$-cycle $w$ and the triple simultaneous sign change. We have $\Phi_w \cong \mathbb{Z}$ generated by $\chi_3$ and $D_w = 0$. Computations in Table \ref{tab:SO7_ix} yield
\begin{align*}
(ix) &= \frac{t^{3r}(uv)^r}{2} \Big[ (tuv+1)^r + (tuv-1)^r \Big]
\end{align*}
\vspace{-0.4cm}
\begin{table}[H]
\tiny
    \centering
    \setlength{\tabcolsep}{3pt}
    \begin{tabular}{|c|c|c|c|c|c|c|}
    \hline
        $[g]\in\textup{Conj}(C(w))$ & $c_w(g)$ & $\tau_{w}(g):D_w\rightarrow D_w$ & \begin{tabular}{@{}c@{}} $n_w(g)=$\\ $\left|\frac{D_w}{\im(I_{D_w}-\tau_w(g))}\right|$\end{tabular} & $\phi_w(g):\Phi_w\rightarrow \Phi_w$ & $tuvI_{\Phi_w}+\phi_w(g)$ & $\det(tuvI_{\Phi_w}+\phi_w(g))$ \\\hline
        
        $[ (e, (+++)) ]$ & $1$ & $-$ & $1$ & 
        $(1)$ & 
        $(tuv+1)$ & 
        $(tuv+1)$ \\\hline
        
        $[ ((123), (+++)) ]$ & $1$ & $-$ & $1$ & 
        $(1)$ & 
        $(tuv+1)$ & 
        $(tuv+1)$ \\\hline
        
        $[ ((132), (+++)) ]$ & $1$ & $-$ & $1$ & 
        $(1)$ & 
        $(tuv+1)$ & 
        $(tuv+1)$ \\\hline
        
        $[ (e, (---)) ]$ & $1$ & $-$ & $1$ & 
        $(-1)$ & 
        $(tuv-1)$ & 
        $(tuv-1)$ \\\hline
        
        $[ ((123), (---)) ]$ & $1$ & $-$ & $1$ & 
        $(-1)$ & 
        $(tuv-1)$ & 
        $(tuv-1)$ \\\hline
        
        $[ ((132), (---)) ]$ & $1$ & $-$ & $1$ & 
        $(-1)$ & 
        $(tuv-1)$ & 
        $(tuv-1)$ \\\hline
    \end{tabular}
    \caption{Computations for $\mathrm{SO}_7$ (ix), conjugacy class of $w=((123), (+++))$ in the basis $\{\chi_3\}$ for $\Phi_w\cong \mathbb{Z}$}
    \label{tab:SO7_ix}
\end{table}
\vspace{-0.3cm}
\item[(x)] For the conjugacy class of $w=((123), (++-))$, the centralizer is $C(w) = \langle w \rangle \cong \mathbb{Z}_6$, with $|C(w)| = 6$. We have $\Phi_w = 0$ and $D_w \cong \mathbb{Z}_2$ generated by $\chi_3$. Computations in Table \ref{tab:SO7_x} yield
\vspace{-0.1cm}
\begin{align*}
(x) &= t^{3r}(uv)^{\frac{3}{2}r}2^r
\end{align*}
\vspace{-0.5cm}
\begin{table}[H]
\tiny
    \centering
    \setlength{\tabcolsep}{3pt}
    \begin{tabular}{|c|c|c|c|c|c|c|}
    \hline
        $[g]\in\textup{Conj}(C(w))$ & $c_w(g)$ & $\tau_{w}(g):D_w\rightarrow D_w$ & \begin{tabular}{@{}c@{}} $n_w(g)=$\\ $\left|\frac{D_w}{\im(I_{D_w}-\tau_w(g))}\right|$\end{tabular} & $\phi_w(g):\Phi_w\rightarrow \Phi_w$ & $tuvI_{\Phi_w}+\phi_w(g)$ & $\det(tuvI_{\Phi_w}+\phi_w(g))$ \\\hline
        
        $[ (e, (+++)) ]$ & $1$ & $(1)\pmod 2$ & $2$ & 
        $-$ & $-$ & $1$ \\\hline
        
        $[ ((123), (++-)) ]$ & $1$ & $(1)\pmod 2$ & $2$ & 
        $-$ & $-$ & $1$ \\\hline
        
        $[ ((132), (+--)) ]$ & $1$ & $(1)\pmod 2$ & $2$ & 
        $-$ & $-$ & $1$ \\\hline
        
        $[ (e, (---)) ]$ & $1$ & $(1)\pmod 2$ & $2$ & 
        $-$ & $-$ & $1$ \\\hline
        
        $[ ((123), (--+)) ]$ & $1$ & $(1)\pmod 2$ & $2$ & 
        $-$ & $-$ & $1$ \\\hline
        
        $[ ((132), (-++)) ]$ & $1$ & $(1)\pmod 2$ & $2$ & 
        $-$ & $-$ & $1$ \\\hline
    \end{tabular}
    \caption{Computations for $\mathrm{SO}_7$ (x), conjugacy class of $w=((123), (++-))$ in the basis $\{\chi_3\}$ for $D_w\cong \mathbb{Z}_2$ }
    \label{tab:SO7_x}
\end{table}
\end{itemize}
\vspace{-0.3cm}
Finally, the stringy compactly supported mixed Hodge polynomial of $T_{\SO_7}^r/W$ is obtained by summing over all 10 conjugacy classes $(\textrm{i}) + (\textrm{ii}) + \dots + (\textrm{x})$:
\begin{align*}
\mu_c^{str}(T_{\SO_7}^r/W)&= \frac{t^{3r}}{48} \Big[ (tuv+1)^{3r} + 9(tuv-1)^r (tuv+1)^{2r} + 9(tuv-1)^{2r} (tuv+1)^r + (tuv-1)^{3r} \\
&\qquad\quad + 6(tuv+1)^r (t^2 u^2 v^2 + 1)^r + 6(tuv-1)^r (t^2 u^2 v^2 + 1)^r \\
&\qquad\quad + 8(tuv+1)^r (t^2 u^2 v^2 - tuv + 1)^r + 8(t^3 u^3 v^3 - 1)^r \Big] \\
&\quad + \frac{t^{3r}(uv)^{\frac{r}{2}} 2^r}{8} \Big[ (tuv+1)^{2r} + 4(tuv-1)^r(tuv+1)^r + (tuv-1)^{2r} + 2(t^2 u^2 v^2 + 1)^r \Big] \\
&\quad + \frac{t^{3r}(uv)^{\frac{r}{2}}}{4} \Big[ (tuv+1)^{2r} + 2(tuv-1)^r(tuv+1)^r + (tuv-1)^{2r} \Big] \\
&\quad + \frac{t^{3r}(uv)^r \big(4^r + 5 \cdot 2^r + 2\big)}{4} \Big[ (tuv+1)^r + (tuv-1)^r \Big] + \frac{t^{3r}(uv)^{\frac{3}{2}r} \big(8^r + 9 \cdot 4^r + 8 \cdot 2^r\big)}{6}.
\end{align*}
For $r=2$ and setting $uv=q$, the compactly supported mixed Hodge polynomial for the $\SO_7$-character variety of the elliptic curve is:
\begin{align*}
\mu_c^{str}(T_{\SO_7}^2/W) &= q^{6} t^{12} + 5 \, q^{5} t^{10} + q^{4} t^{10} + 19 \, q^{4} t^{8} + 6 \, q^{3} t^{8} + q^{2} t^{8} + 40 \, q^{3} t^{6} + 19 \, q^{2} t^{6} + 5 \, q t^{6} + t^{6}\; .
\end{align*}
Evaluating at $t = -1$ yields the stringy $E$-polynomial of the $\SO_7$-character variety of the elliptic curve:
\begin{align*}
E^\str(T_{\SO_7}^2/W) &= q^6 + 5q^5 + 20q^4 + 46q^3 + 20q^2 + 5q + 1\; .
\end{align*}

\subsection{Example: $G = \G_2$}\label{ssec:G2}
Let us now describe the computations in Theorem \ref{thm:algorithm_CV} for 
the case $G=\G_2$ of rank $2$. Notice that for this root system, the weight and 
the root lattices coincide; hence, the group is both simply connected and centerless. 
Let $T_{\G_2} \subseteq \G_2$ be a maximal torus, and let $\Lambda = X^*(T_{\G_2}) = Q = P \cong \mathbb{Z}^2$ 
be the character lattice, equal to the root and weight lattices, and thus generated 
by the two simple roots. Following the notation in Section \ref{sec:action_weyl}, 
we have that $n = m = 2$, and we can choose the simple roots $\alpha_1$ and 
$\alpha_2$ as a basis for $\Lambda = Q$, with $\alpha_1$ being the long root 
and $\alpha_2$ the short one. Since the simple roots are connected in the corresponding Dynkin diagram by a triple 
link pointing from the long root $\alpha_1$ to the short root $\alpha_2$, the Cartan 
matrix of the Lie algebra of $\G_2$ is:
\[
A = \begin{pmatrix} 2 & -1 \\ -3 & 2 \end{pmatrix}.
\]
Note that this Cartan matrix is not symmetric, as this is a non-simply laced Dynkin diagrams.

The Weyl group of $\G_2$ is $W = D_6$, the dihedral group of order $|D_6|=12$. It is 
generated by the simple reflections $s_1, s_2$ with respect to the hyperplanes 
orthogonal to the simple roots $\alpha_1, \alpha_2$, respectively. Since $\Lambda = Q$, 
the matrix representing the action of each simple reflection $s_k$ on $\Lambda$ is 
given by $B_k = C_k = I_2 - E_{kk}A$ (see (\ref{eq:matrix_Ck})), yielding: 
\[
B_1 = \begin{pmatrix} -1 & 1 \\ 0 & 1 \end{pmatrix}\quad , \quad 
B_2 = \begin{pmatrix} 1 & 0 \\ 3 & -1 \end{pmatrix}.
\]
These simple reflections generate the Weyl group $W=D_6$, where each element acts on $\Lambda=\mathbb{Z}\chi_1\oplus \mathbb{Z}\chi_2$ as the matrix in Table \ref{tab:G2_action}.

\begin{table}[H]
\scriptsize
\centering
\vspace{0.2cm}
\setlength{\tabcolsep}{4pt}
\begin{tabular}{|c|c|c||c|c|c||c|c|c||c|c|c|}
\hline
$w\in D_6$ & word & $B_w$ & $w\in D_6$ & word & $B_w$ & $w\in D_6$ & word & $B_w$ & $w\in D_6$ & word & $B_w$ \\ 
\hline
$1$ & $1$ & $\begin{pmatrix} 1 & 0 \\ 0 & 1 \end{pmatrix}$ & $\rho^3$ & $(s_1 s_2)^3$ & $\begin{pmatrix} -1 & 0 \\ 0 & -1 \end{pmatrix}$ & $\tau$ & $s_1$ & $\begin{pmatrix} -1 & 1 \\ 0 & 1 \end{pmatrix}$ & $\tau\rho^3$ & $s_2 (s_1 s_2)^2$ & $\begin{pmatrix} 1 & -1 \\ 0 & -1 \end{pmatrix}$ \\\hline
$\rho$ & $s_1 s_2$ & $\begin{pmatrix} -1 & 1 \\ -3 & 2 \end{pmatrix}$ & $\rho^4$ & $(s_2 s_1)^2$ & $\begin{pmatrix} 1 & -1 \\ 3 & -2 \end{pmatrix}$ & $\tau\rho$ & $s_2$ & $\begin{pmatrix} 1 & 0 \\ 3 & -1 \end{pmatrix}$ & $\tau\rho^4$ & $s_1 (s_2 s_1)^2$ & $\begin{pmatrix} -1 & 0 \\ -3 & 1 \end{pmatrix}$ \\\hline
$\rho^2$ & $(s_1 s_2)^2$ & $\begin{pmatrix} -2 & 1 \\ -3 & 1 \end{pmatrix}$ & $\rho^5$ & $s_2 s_1$ & $\begin{pmatrix} 2 & -1 \\ 3 & -1 \end{pmatrix}$ & $\tau\rho^2$ & $s_2 s_1 s_2$ & $\begin{pmatrix} 2 & -1 \\ 3 & -2 \end{pmatrix}$ & $\tau\rho^5$ & $s_1s_2s_1$ & $\begin{pmatrix} -2 & 1 \\ -3 & 2 \end{pmatrix}$ \\\hline
\end{tabular}
\caption{Action of elements of $W=D_6$ on $\Lambda$ for $\G_2$.}
\label{tab:G2_action}
\end{table}

Table \ref{tab:G2_main} summarizes the computations involved in (\ref{eq:magic-formula}) for the compactly supported stringy mixed Hodge polynomial $\mu^{\str}_c(T_{\G_2}^r/W)$, for the six conjugacy classes in $W=D_6$.

\begin{table}[H]
\footnotesize
    \centering
    \setlength{\tabcolsep}{3pt}
    \begin{tabular}{|c|c|c|c|c|c|c|c|c|c|}
    \hline
        $[w]\in\textup{Conj}(W)$ & $w$ & $|C(w)|$ & $I-B_w$ & $A_w=\im (I-B_w)$& $\rk A_w$ & $K_w=\Lambda / A_w$ & $s_w$ & $\Phi_w$ & $D_w$\\\hline
        (i) & $1$ & $12$ & $\begin{pmatrix} 0 & 0 \\ 0 & 0 \end{pmatrix}$ & $0$ & $0$ & $\mathbb{Z}\chi_1\oplus \mathbb{Z}\chi_2$ & $2$ & $\mathbb{Z}\chi_1\oplus \mathbb{Z}\chi_2$ & $0$\\\hline
        (ii) & $\rho^3$ & $12$ & $\begin{pmatrix} 2 & 0 \\ 0 & 2 \end{pmatrix}$ & $\mathbb{Z}2\chi_1\oplus \mathbb{Z}2\chi_2\cong \mathbb{Z}^2$ & $2$ & \begin{tabular}{@{}c@{}} $\frac{\mathbb{Z}\chi_1}{\mathbb{Z}(2\chi_1)}\oplus\frac{\mathbb{Z}\chi_2}{\mathbb{Z}(2\chi_2)}$\\ $\cong \mathbb{Z}_2\oplus\mathbb{Z}_2$\end{tabular} & $0$ & $0$ & \begin{tabular}{@{}c@{}} $\frac{\mathbb{Z}\chi_1}{\mathbb{Z}(2\chi_1)}\oplus\frac{\mathbb{Z}\chi_2}{\mathbb{Z}(2\chi_2)}$\\ $\cong \mathbb{Z}_2\oplus\mathbb{Z}_2$\end{tabular}\\\hline
        (iii) & $\rho$ & $6$ & $\begin{pmatrix} 2 & -1 \\ 3 & -1 \end{pmatrix}$ & $\mathbb{Z}\chi_1\oplus \mathbb{Z}\chi_2\cong\mathbb{Z}^2$ & $2$ & $0$ & $0$ & $0$  & $0$ \\\hline
        (iv) & $\rho^2$ & $6$ & $\begin{pmatrix} 3 & -1 \\ 3 & 0 \end{pmatrix}$ & $\mathbb{Z}\chi_1\oplus\mathbb{Z}3\chi_2\cong \mathbb{Z}^2$& $2$ & $\frac{\mathbb{Z}\chi_2}{\mathbb{Z}3\chi_2}\cong \mathbb{Z}_3$ & $0$ & $0$ & $\frac{\mathbb{Z}\chi_2}{\mathbb{Z}3\chi_2}\cong \mathbb{Z}_3$\\\hline
        (v) & $\tau$ & $4$ & $\begin{pmatrix} 2 & -1 \\ 0 & 0 \end{pmatrix}$ & $\mathbb{Z}\chi_1\cong \mathbb{Z}$& $1$ & $\mathbb{Z}\chi_2 \cong \mathbb{Z}$ & $1$ & $\mathbb{Z}\chi_2\cong \mathbb{Z}$ & $0$ \\\hline
        (vi) & $\tau\rho$ & $4$ & $\begin{pmatrix} 0 & 0 \\ -3 & 2 \end{pmatrix}$ & $\mathbb{Z}\chi_2\cong \mathbb{Z}$ & $1$ & $\mathbb{Z}\chi_1 \cong \mathbb{Z}$ & $1$ & $\mathbb{Z}\chi_1\cong \mathbb{Z}$ & $0$ \\\hline
    \end{tabular}
    \caption{Invariants of each conjugacy class in Theorem \ref{thm:algorithm_CV} for $\G_2$.}
    \label{tab:G2_main}
\end{table}
\vspace{-0.4cm}
Now let $w$ be the chosen representative in Table \ref{tab:G2_main} of each conjugacy class in $D_6$. For a representative of each conjugacy class in the centralizer $C(w)$, we compute the data we need. 
\begin{itemize}
    \item[(i)] For the conjugacy class of $w=1$, the centralizer is $C(1)=D_6$. The free and torsion parts are $\Phi_w \cong \mathbb{Z}^2$ in the basis $\{\chi_1, \chi_2\}$, and $D_w = 0$. We compute in Table \ref{tab:G2_i} the rest of the information required in (\ref{eq:magic-formula}), accounting for 
\begin{align*}
(i) &=\frac{t^{2r}}{12}\big[(tuv+1)^{2r}+(tuv-1)^{2r}+2(t^2u^2v^2+tuv+1)^r+2(t^2u^2v^2-tuv+1)^r+6(tuv+1)^r(tuv-1)^r\big]
\end{align*}
\vspace{-0.3cm}
\begin{table}[H]
\tiny
    \centering
    \begin{tabular}{|c|c|c|c|c|c|c|}
    \hline
        $[g]\in\textup{Conj}(C(w))$ & $c_1(g)$ & $\tau_{w}(g):D_w\rightarrow D_w$ & $n_w(g)=\left|\frac{D_{w}}{\im(I_{D_w}-\tau_{w}(g))}\right|$ & $\phi_w(g):\Phi_w\rightarrow \Phi_w$ & $tuvI_{\Phi_w}+\phi_w(g)$ & $\det(tuvI_{\Phi_w}+\phi_w(g))$ \\\hline
$1$ & $1$ & $-$ & $1$ & $\begin{pmatrix} 1 & 0 \\ 0 & 1 \end{pmatrix}$ & $\begin{pmatrix} tuv+1 & 0 \\ 0 & tuv+1 \end{pmatrix}$ & $(tuv+1)^2$\\\hline
$\rho^3$ & $1$ & $-$ & $1$ & $\begin{pmatrix} -1 & 0 \\ 0 & -1 \end{pmatrix}$ & $\begin{pmatrix} tuv-1 & 0 \\ 0 & tuv-1 \end{pmatrix}$ & $(tuv-1)^2$ \\\hline
$\rho$ & $2$ & $-$ & $1$ & $\begin{pmatrix} -1 & 1 \\ -3 & 2 \end{pmatrix}$ & $\begin{pmatrix} tuv-1 & 1 \\ -3 & tuv+2 \end{pmatrix}$ & $(tuv)^2+tuv+1$ \\\hline
$\rho^2$ & $2$ & $-$ & $1$ & $\begin{pmatrix} -2 & 1 \\ -3 & 1 \end{pmatrix}$ & $\begin{pmatrix} tuv-2 & 1 \\ -3 & tuv+1 \end{pmatrix}$ & $(tuv)^2-tuv+1$ \\\hline
$\tau$ & $3$ & $-$ & $1$ & $\begin{pmatrix} -1 & 1 \\ 0 & 1 \end{pmatrix}$ & $\begin{pmatrix} tuv-1 & 1 \\ 0 & tuv+1 \end{pmatrix}$ & $(tuv+1)(tuv-1)$ \\\hline
$\tau\rho$ & $3$ & $-$ & $1$ & $\begin{pmatrix} 1 & 0 \\ 3 & -1 \end{pmatrix}$ & $\begin{pmatrix} tuv+1 & 0 \\ 3 & tuv-1 \end{pmatrix}$ & $(tuv+1)(tuv-1)$ \\\hline
    \end{tabular}
    \caption{Computations for $\G_2$ (i), conjugacy class of $w=1$ in the basis $\{\chi_1,\chi_2\}$ of $\Phi_w\cong \mathbb{Z}^2$.}
    \label{tab:G2_i}
\end{table}
\vspace{-0.3cm}
\item[(ii)] For the conjugacy class of $w=\rho^3$, the centralizer is $C(\rho^3)=D_6$. The free and torsion parts are $\Phi_w = 0$ and $D_w = \mathbb{Z}_2\oplus \mathbb{Z}_2$, in the basis $\{\chi_1,\chi_2\}$. We compute in Table \ref{tab:G2_ii} the rest of the information required in (\ref{eq:magic-formula}), accounting for 
\begin{align*}
(ii) &=\frac{t^{2r}(uv)^{r}}{6}\big[4^r + 2+3\cdot 2^r \big]
\end{align*}
\vspace{-0.3cm}
\begin{table}[H]
\tiny
    \centering
    \begin{tabular}{|c|c|c|c|c|c|c|}
    \hline
        $[g]\in\textup{Conj}(C(w))$ & $c_w(g)$ & $\tau_{w}(g):D_w\rightarrow D_w$ & $n_w(g)=\left|\frac{D_{w}}{\im(I_{D_w}-\tau_{w}(g))}\right|$ & $\phi_w(g):\Phi_w\rightarrow \Phi_w$ & $tuvI_{\Phi_w}+\phi_w(g)$ & $\det(tuvI_{\Phi_w}+\phi_w(g))$ \\\hline
$1$ & $1$ & $\begin{pmatrix} 1 & 0 \\ 0 & 1 \end{pmatrix} \pmod 2$ & $4$ & $-$ & $-$ & $1$\\\hline
$\rho^3$ & $1$ & $\begin{pmatrix} 1 & 0 \\ 0 & 1 \end{pmatrix}\pmod 2$ & $4$ & $-$ & $-$ & $1$ \\\hline
$\rho$ & $2$ & $\begin{pmatrix} 1 & 1 \\ 1 & 0 \end{pmatrix}\pmod 2$ & $1$ & $-$ & $-$ & $1$ \\\hline
$\rho^2$ & $2$ & $\begin{pmatrix} 0 & 1 \\ 1 & 1 \end{pmatrix}\pmod 2$ & $1$ & $-$ & $-$ & $1$ \\\hline
$\tau$ & $3$ & $\begin{pmatrix} 1 & 1 \\ 0 & 1 \end{pmatrix}\pmod 2$ & $2$ & $-$ & $-$ & $1$ \\\hline
$\tau\rho$ & $3$ & $\begin{pmatrix} 1 & 0 \\ 1 & 1 \end{pmatrix}\pmod 2$ & $2$ & $-$ & $-$ & $1$ \\\hline
    \end{tabular}
    \caption{Computations for $\G_2$ (ii), conjugacy class of $w=\rho^3$, in the basis $\{\chi_1,\chi_2\}$ of $D_w\cong (\mathbb{Z}_2)^2$.}
    \label{tab:G2_ii}
\end{table}
\vspace{-0.3cm}
\item[(iii)] For the conjugacy class of $w=\rho$, the centralizer is $C(\rho)=\langle \rho\rangle \cong \mathbb{Z}_6$. The free and torsion parts are $\Phi_w = 0$ and $D_w = 0$, hence the corresponding term in (\ref{eq:magic-formula}) is (Table \ref{tab:G2_iii}):
\begin{align*}
(iii) &= t^{2r}(uv)^r
\end{align*}
\vspace{-0.3cm}
\begin{table}[H]
\tiny
    \centering
    \begin{tabular}{|c|c|c|c|c|c|c|}
    \hline
        $[g]\in\textup{Conj}(C(w))$ & $c_w(g)$ & $\tau_{w}(g):D_w\rightarrow D_w$ & $n_w(g)=\left|\frac{D_{w}}{\im(I_{D_w}-\tau_{w}(g))}\right|$ & $\phi_w(g):\Phi_w\rightarrow \Phi_w$ & $tuvI_{\Phi_w}+\phi_w(g)$ & $\det(tuvI_{\Phi_w}+\phi_w(g))$ \\\hline
$1$ & $1$ & $-$ & $1$ & $-$ & $-$ & $1$ \\\hline
$\rho$ & $1$ & $-$ & $1$ & $-$ & $-$ & $1$ \\\hline
$\rho^2$ & $1$ & $-$ & $1$ & $-$ & $-$ & $1$ \\\hline
$\rho^3$ & $1$ & $-$ & $1$ & $-$ & $-$ & $1$ \\\hline
$\rho^4$ & $1$ & $-$ & $1$ & $-$ & $-$ & $1$ \\\hline
$\rho^5$ & $1$ & $-$ & $1$ & $-$ & $-$ & $1$ \\\hline
    \end{tabular}
    \caption{Computations for $\G_2$ (iii), conjugacy class of $w=\rho$.}
    \label{tab:G2_iii}
\end{table}
\vspace{-0.3cm}
\item[(iv)] For the conjugacy class of $w=\rho^2$, the centralizer is $C(\rho^2)=\langle \rho\rangle \cong \mathbb{Z}_6$, abelian. The free and torsion parts are $\Phi_w = 0$ and $D_w = \mathbb{Z}_3$, where we consider the basis given by $\{\chi_2\}$. Table \ref{tab:G2_iv} gives
\begin{align*}
(iv) &= \frac{t^{2r}(uv)^r}{2} \left( 3^r + 1\right)
\end{align*}
\vspace{-0.3cm}
\begin{table}[H]
\tiny
    \centering
    \begin{tabular}{|c|c|c|c|c|c|c|}
    \hline
        $[g]\in\textup{Conj}(C(w))$ & $c_w(g)$ & $\tau_{w}(g):D_w\rightarrow D_w$ & $n_w(g)=\left|\frac{D_{w}}{\im(I_{D_w}-\tau_{w}(g))}\right|$ & $\phi_w(g):\Phi_w\rightarrow \Phi_w$ & $tuvI_{\Phi_w}+\phi_w(g)$ & $\det(tuvI_{\Phi_w}+\phi_w(g))$ \\\hline
$1$ & $1$ & $(1)$ & $3$ & $-$ & $-$ & $1$ \\\hline
$\rho$ & $1$ & $(-1)$ & $1$ & $-$ & $-$ & $1$ \\\hline
$\rho^2$ & $1$ & $(1)$ & $3$ & $-$ & $-$ & $1$ \\\hline
$\rho^3$ & $1$ & $(-1)$ & $1$ & $-$ & $-$ & $1$ \\\hline
$\rho^4$ & $1$ & $(1)$ & $3$ & $-$ & $-$ & $1$ \\\hline
$\rho^5$ & $1$ & $(-1)$ & $1$ & $-$ & $-$ & $1$ \\\hline
    \end{tabular}
    \caption{Computations for $\G_2$ (iv), conjugacy class of $w=\rho^2$, in the basis $\{\chi_2\}$ of $D_w\cong\mathbb{Z}_3$}
    \label{tab:G2_iv}
\end{table}
\vspace{-0.4cm}
\item[(v)] For the conjugacy class of $w=\tau$, the centralizer is $C(\tau)=\langle \tau, \rho^3\rangle \cong \mathbb{Z}_2 \times \mathbb{Z}_2$, abelian. The free and torsion parts are $\Phi_w = \mathbb{Z}$ in the basis $\{\chi_2\}$ and $D_w = 0$. Table \ref{tab:G2_v} gives
\begin{align*}
(v) &= \frac{t^{2r}(uv)^{\frac{r}{2}}}{2}\big[ (tuv+1)^r + (tuv-1)^r \big]
\end{align*}
\vspace{-0.4cm}
\begin{table}[H]
\tiny
    \centering
    \begin{tabular}{|c|c|c|c|c|c|c|}
    \hline
        $[g]\in\textup{Conj}(C(w))$ & $c_w(g)$ & $\tau_{w}(g):D_w\rightarrow D_w$ & $n_w(g)=\left|\frac{D_{w}}{\im(I_{D_w}-\tau_{w}(g))}\right|$ & $\phi_w(g):\Phi_w\rightarrow \Phi_w$ & $tuvI_{\Phi_w}+\phi_w(g)$ & $\det(tuvI_{\Phi_w}+\phi_w(g))$ \\\hline
$1$ & $1$ & $-$ & $1$ & $(1)$ & $(tuv+1)$ & $tuv+1$ \\\hline
$\rho^3$ & $1$ & $-$ & $1$ & $(-1)$ & $(tuv-1)$ & $tuv-1$ \\\hline
$\tau$ & $1$ & $-$ & $1$ & $(1)$ & $(tuv+1)$ & $tuv+1$ \\\hline
$\tau\rho^3$ & $1$ & $-$ & $1$ & $(-1)$ & $(tuv-1)$ & $tuv-1$ \\\hline
    \end{tabular}
    \caption{Computations for $\G_2$ (v), conjugacy class of $w=\tau$ in the basis $\{\chi_2\}$ of $\Phi_w\cong \mathbb{Z}$.}
    \label{tab:G2_v}
\end{table}
\vspace{-0.4cm}
\item[(vi)] For the conjugacy class of $w=\tau\rho$, the centralizer is $C(\tau\rho)=\langle \tau\rho, \rho^3\rangle \cong \mathbb{Z}_2 \times \mathbb{Z}_2$, abelian. The free and torsion parts are $\Phi_w = \mathbb{Z}$ in the basis $\{\chi_1\}$ and $D_w = 0$. Table \ref{tab:G2_vi} gives
\begin{align*}
(vi) &= \frac{t^{2r}(uv)^{\frac{r}{2}}}{2}\big[ (tuv+1)^r + (tuv-1)^r \big]
\end{align*}
\vspace{-0.4cm}
\begin{table}[h!]
\tiny
    \centering
    \begin{tabular}{|c|c|c|c|c|c|c|}
    \hline
        $[g]\in\textup{Conj}(C(w))$ & $c_w(g)$ & $\tau_{w}(g):D_w\rightarrow D_w$ & $n_w(g)=\left|\frac{D_{w}}{\im(I_{D_w}-\tau_{w}(g))}\right|$ & $\phi_w(g):\Phi_w\rightarrow \Phi_w$ & $tuvI_{\Phi_w}+\phi_w(g)$ & $\det(tuvI_{\Phi_w}+\phi_w(g))$ \\\hline
$1$ & $1$ & $-$ & $1$ & $(1)$ & $(tuv+1)$ & $tuv+1$ \\\hline
$\rho^3$ & $1$ & $-$ & $1$ & $(-1)$ & $(tuv-1)$ & $tuv-1$ \\\hline
$\tau\rho$ & $1$ & $-$ & $1$ & $(1)$ & $(tuv+1)$ & $tuv+1$ \\\hline
$\tau\rho^4$ & $1$ & $-$ & $1$ & $(-1)$ & $(tuv-1)$ & $tuv-1$ \\\hline
    \end{tabular}
    \caption{Computations for $\G_2$ (vi), conjugacy class of $w=\tau\rho$ in the basis $\{\chi_1\}$ of $\Phi_w\cong \mathbb{Z}$.}
    \label{tab:G2_vi}
\end{table}

\end{itemize}
\vspace{-0.4cm}
Finally, the stringy compactly supported mixed Hodge polynomial of $T_{\G_2}^r/W$ is obtained as 
\vspace{-0.1cm}
 \begin{align*}
    \mu^{\mathrm{str}}_c(T_{\G_2}^r/W) &= (\textrm{i}) + (\textrm{ii}) + (\textrm{iii}) + (\textrm{iv}) + (\textrm{v}) + (\textrm{vi}) \\
    &= \frac{t^{2r}}{12} \bigg[ (tuv+1)^{2r} + (tuv-1)^{2r} + 2(t^2u^2v^2+tuv+1)^r + 2(t^2u^2v^2-tuv+1)^r \\
    & \; + 6(t^2u^2v^2-1)^r + (uv)^r \left( 2 \cdot 4^r + 6 \cdot 3^r +6 \cdot 2^{r} + 22 \right) + 12(uv)^{\frac{r}{2}} \left( (tuv+1)^r + (tuv-1)^r \right) \bigg]
\end{align*}

For $r=2$ and setting $uv=q$, the compactly supported mixed Hodge polynomial for the $\G_2$-character variety of the elliptic curve is:
\begin{align*}
\mu_c^{str}(T_{\G_2}^2/W) &= q^{4} t^{8} + 2 \, q^{3} t^{6} + q^{2} t^{6} + 11 \, q^{2} t^{4} + 2 \, q t^{4} + t^{4}\; .
\end{align*}
Evaluating at $t = -1$ yields the stringy $E$-polynomial of the $G_2$-character variety of the elliptic curve:
\begin{align*}
E^\str(T_{\G_2}^2/W) &= q^4 + 2\,q^3 + 12\,q^2 + 2\,q + 1\; .
\end{align*}

\subsection{Code and tables}

The formula from Theorem \ref{thm:algorithm_CV} and its compactly supported version from Corollary \ref{cor:compact-support} have been implemented in SageMath \cite{sagemath}. Using it, it is possible to compute the stringy mixed Hodge polynomials for all the groups of rank at most $4$ within seconds, and all the exceptional groups within less than one minute. The code has been posted in a public repository accessible at \cite{code_repository}.

In Table \ref{tab:summary}, we collect some of the results obtained with this code for the case $r = 2$, corresponding to an elliptic curve, for the identity component of the (normal, in this case) character variety $\mathfrak{X}_{2}^\circ(G) = T^2/W$. Concretely, we present all the stringy mixed Hodge polynomials and stringy Poincar\'e polynomials (with and without compact support), the stringy $E$-polynomial and the stringy Euler characteristic for some classical and exceptional groups up to rank $4$. The mixed Hodge polynomials for all the exceptional groups and arbitrary $r$ are displayed in Appendix \ref{app:calculations}.

\begin{table}[H]
\scriptsize
\begin{tabular}{|c||cccccl|}
\hline
\footnotesize $\bm{G}$                       & \multicolumn{6}{c|}{\textbf{\footnotesize Stringy Invariants}}                                                                                                                                                                                                         \\ \hline\hline
 & \multicolumn{6}{|c|}{\textbf{Rank $\bm{1}$}} \\ \hline
                         
\multirow{4}{*}{\begin{tabular}{@{}c@{}}$\SL_2$,\\ $\PGL_2$\end{tabular}} & $\bm{\mu^{\str}(T_{G}^2/W)}$   & \multicolumn{5}{|l|}{$q^{2} t^{2} + 4  q t^{2} + 1$}  \\ \cline{2-7} 
                         & $\bm{P^{\str}(T_{G}^2/W)}$      & \multicolumn{3}{|l|}{$5  t^{2} + 1$} & $\bm{\chi^{\str}(T_{G}^2/W)}$  & \multicolumn{1}{|l|}{$6$} \\ \cline{2-7} 
                         & $\bm{\mu^{\str}_c(T_{G}^2/W)}$  & \multicolumn{5}{|l|}{$q^{2} t^{4} + 4  q t^{2} + t^{2}$} \\ \cline{2-7} 
                         & $\bm{P^{\str}_c(T_{G}^2/W)}$    & \multicolumn{2}{|l|}{$t^{4} + 5  t^{2}$} & $\bm{E^{\str}(T_{G}^2/W)}$ & \multicolumn{2}{|l|}{$q^{2} + 4  q + 1$} \\ \hline\hline

 & \multicolumn{6}{|c|}{\textbf{Rank $\bm{2}$}} \\ \hline 

\multirow{4}{*}{$\GL_2$} & $\bm{\mu^{\str}(T_{G}^2/W)}$    & \multicolumn{5}{|l|}{$q^{4} t^{4} + q^{3} t^{4} + 2  q^{3} t^{3} + 2  q^{2} t^{3} + 2  q^{2} t^{2} + q t^{2} + 2  q t + 1$}  \\ \cline{2-7} 
                         & $\bm{P^{\str}(T_{G}^2/W)}$      & \multicolumn{3}{|l|}{$2  t^{4} + 4  t^{3} + 3  t^{2} + 2  t + 1$} & $\bm{\chi^{\str}(T_{G}^2/W)}$  & \multicolumn{1}{|l|}{$0$} \\ \cline{2-7} 
                         & $\bm{\mu^{\str}_c(T_{G}^2/W)}$  & \multicolumn{5}{|l|}{$q^{4} t^{8} + 2  q^{3} t^{7} + q^{3} t^{6} + 2  q^{2} t^{6} + 2  q^{2} t^{5} + 2  q t^{5} + q t^{4} + t^{4}$} \\ \cline{2-7} 
                         & $\bm{P^{\str}_c(T_{G}^2/W)}$    & \multicolumn{2}{|l|}{$t^{8} + 2  t^{7} + 3  t^{6} + 4  t^{5} + 2  t^{4}$} & $\bm{E^{\str}(T_{G}^2/W)}$ & \multicolumn{2}{|l|}{$q^{4} - q^{3} - q + 1$} \\ \hline\hline
                         
\multirow{4}{*}{\begin{tabular}{@{}c@{}}$\SL_3$,\\ $\PGL_3$\end{tabular}} & $\bm{\mu^{\str}(T_{G}^2/W)}$    & \multicolumn{5}{|l|}{$q^{4} t^{4} + q^{3} t^{4} + 9  q^{2} t^{4} + 2  q^{2} t^{3} + q^{2} t^{2} + q t^{2} + 1$}  \\ \cline{2-7} 
                         & $\bm{P^{\str}(T_{G}^2/W)}$      & \multicolumn{3}{|l|}{$11  t^{4} + 2  t^{3} + 2  t^{2} + 1$} & $\bm{\chi^{\str}(T_{G}^2/W)}$  & \multicolumn{1}{|l|}{$12$} \\ \cline{2-7} 
                         & $\bm{\mu^{\str}_c(T_{G}^2/W)}$  & \multicolumn{5}{|l|}{$q^{4} t^{8} + q^{3} t^{6} + q^{2} t^{6} + 2  q^{2} t^{5} + 9  q^{2} t^{4} + q t^{4} + t^{4}$} \\ \cline{2-7} 
                         & $\bm{P^{\str}_c(T_{G}^2/W)}$    & \multicolumn{2}{|l|}{$t^{8} + 2  t^{6} + 2  t^{5} + 11  t^{4}$} & $\bm{E^{\str}(T_{G}^2/W)}$ & \multicolumn{2}{|l|}{$q^{4} + q^{3} + 8  q^{2} + q + 1$} \\ \hline\hline  

\multirow{4}{*}{\begin{tabular}{@{}c@{}}$\SO_5$,\\ $\Sp_4$\end{tabular}} & $\bm{\mu^{\str}(T_{G}^2/W)}$    & \multicolumn{5}{|l|}{$q^{4} t^{4} + 5  q^{3} t^{4} + 14  q^{2} t^{4} + q^{2} t^{2} + 5  q t^{2} + 1$}  \\ \cline{2-7} 
                         & $\bm{P^{\str}(T_{G}^2/W)}$  & \multicolumn{3}{|l|}{$20  t^{4} + 6  t^{2} + 1$} & $\bm{\chi^{\str}(T_{G}^2/W)}$  & \multicolumn{1}{|l|}{$27$} \\ \cline{2-7} 
                         & $\bm{\mu^{\str}_c(T_{G}^2/W)}$  & \multicolumn{5}{|l|}{$q^{4} t^{8} + 5  q^{3} t^{6} + q^{2} t^{6} + 14  q^{2} t^{4} + 5  q t^{4} + t^{4}$} \\ \cline{2-7} 
                         & $\bm{P^{\str}_c(T_{G}^2/W)}$    & \multicolumn{2}{|l|}{$t^{8} + 6  t^{6} + 20  t^{4}$} & $\bm{E^{\str}(T_{G}^2/W)}$ & \multicolumn{2}{|l|}{$q^{4} + 5  q^{3} + 15  q^{2} + 5  q + 1$} \\ \hline\hline                            

\multirow{4}{*}{$\G_2$} & $\bm{\mu^{\str}(T_{G}^2/W)}$   & \multicolumn{5}{|l|}{$q^{4} t^{4} + 2  q^{3} t^{4} + 11  q^{2} t^{4} + q^{2} t^{2} + 2  q t^{2} + 1$}  \\ \cline{2-7} 
                         & $\bm{P^{\str}(T_{G}^2/W)}$      & \multicolumn{3}{|l|}{$14  t^{4} + 3  t^{2} + 1$} & $\bm{\chi^{\str}(T_{G}^2/W)}$  & \multicolumn{1}{|l|}{$18$} \\ \cline{2-7} 
                         & $\bm{\mu^{\str}_c(T_{G}^2/W)}$  & \multicolumn{5}{|l|}{$q^{4} t^{8} + 2  q^{3} t^{6} + q^{2} t^{6} + 11  q^{2} t^{4} + 2  q t^{4} + t^{4}$} \\ \cline{2-7} 
                         & $\bm{P^{\str}_c(T_{G}^2/W)}$    & \multicolumn{2}{|l|}{$t^{8} + 3  t^{6} + 14  t^{4}$} & $\bm{E^{\str}(T_{G}^2/W)}$ & \multicolumn{2}{|l|}{$q^{4} + 2  q^{3} + 12  q^{2} + 2  q + 1$} \\ \hline\hline

 & \multicolumn{6}{|c|}{\textbf{Rank $\bm{3}$}} \\ \hline

\multirow{4}{*}{$\GL_3$} & $\bm{\mu^{\str}(T_{G}^2/W)}$    & \multicolumn{5}{|l|}{\scalebox{0.8}{$q^{6} t^{6} + q^{5} t^{6} + 2  q^{5} t^{5} + q^{4} t^{6} + 4  q^{4} t^{5} + 2  q^{4} t^{4} + 2  q^{3} t^{5} + 6  q^{3} t^{4} + 2  q^{3} t^{3} + q^{2} t^{4} + 4  q^{2} t^{3} + 2  q^{2} t^{2} + q t^{2} + 2  q t + 1$}}  \\ \cline{2-7} 
                         & $\bm{P^{\str}(T_{G}^2/W)}$      & \multicolumn{3}{|l|}{$3  t^{6} + 8  t^{5} + 9  t^{4} + 6  t^{3} + 3  t^{2} + 2  t + 1$} & $\bm{\chi^{\str}(T_{G}^2/W)}$  & \multicolumn{1}{|l|}{$0$} \\ \cline{2-7} 
                         & $\bm{\mu^{\str}_c(T_{G}^2/W)}$  & \multicolumn{5}{|l|}{\scalebox{0.8}{$q^{6} t^{12} + 2  q^{5} t^{11} + q^{5} t^{10} + 2  q^{4} t^{10} + 4  q^{4} t^{9} + q^{4} t^{8} + 2  q^{3} t^{9} + 6  q^{3} t^{8} + 2  q^{3} t^{7} + 2  q^{2} t^{8} + 4  q^{2} t^{7} + q^{2} t^{6} + 2  q t^{7} + q t^{6} + t^{6}$}} \\ \cline{2-7} 
                         & $\bm{P^{\str}_c(T_{G}^2/W)}$    & \multicolumn{2}{|l|}{\tiny $t^{12} + 2  t^{11} + 3  t^{10} + 6  t^{9} + 9  t^{8} + 8  t^{7} + 3  t^{6}$} & $\bm{E^{\str}(T_{G}^2/W)}$ & \multicolumn{2}{|l|}{$q^{6} - q^{5} - q^{4} + 2  q^{3} - q^{2} - q + 1$} \\ \hline\hline   
                         
\multirow{4}{*}{\begin{tabular}{@{}c@{}}$\SL_4$,\\ $\PGL_4$\end{tabular}} & $\bm{\mu^{\str}(T_{G}^2/W)}$    & \multicolumn{5}{|l|}{$q^{6} t^{6} + q^{5} t^{6} + 5  q^{4} t^{6} + 2  q^{4} t^{5} + 16  q^{3} t^{6} + q^{4} t^{4} + 2  q^{3} t^{5} + 2  q^{3} t^{4} + 5  q^{2} t^{4} + 2  q^{2} t^{3} + q^{2} t^{2} + q t^{2} + 1$}  \\ \cline{2-7} 
                         & $\bm{P^{\str}(T_{G}^2/W)}$      & \multicolumn{3}{|l|}{$23  t^{6} + 4  t^{5} + 8  t^{4} + 2  t^{3} + 2  t^{2} + 1$} & $\bm{\chi^{\str}(T_{G}^2/W)}$  & \multicolumn{1}{|l|}{$28$} \\ \cline{2-7} 
                         & $\bm{\mu^{\str}_c(T_{G}^2/W)}$  & \multicolumn{5}{|l|}{$q^{6} t^{12} + q^{5} t^{10} + q^{4} t^{10} + 2  q^{4} t^{9} + 5  q^{4} t^{8} + 2  q^{3} t^{8} + 2  q^{3} t^{7} + q^{2} t^{8} + 16  q^{3} t^{6} + 2  q^{2} t^{7} + 5  q^{2} t^{6} + q t^{6} + t^{6}$} \\ \cline{2-7} 
                         & $\bm{P^{\str}_c(T_{G}^2/W)}$    & \multicolumn{2}{|l|}{$t^{12} + 2  t^{10} + 2  t^{9} + 8  t^{8} + 4  t^{7} + 23  t^{6}$} & $\bm{E^{\str}(T_{G}^2/W)}$ & \multicolumn{2}{|l|}{$q^{6} + q^{5} + 4  q^{4} + 16  q^{3} + 4  q^{2} + q + 1$} \\ \hline\hline

\multirow{4}{*}{\begin{tabular}{@{}c@{}}$\SO_7$,\\ $\Sp_6$\end{tabular} } & $\bm{\mu^{\str}(T_{G}^2/W)}$   & \multicolumn{5}{|l|}{$q^{6} t^{6} + 5  q^{5} t^{6} + 19  q^{4} t^{6} + 40  q^{3} t^{6} + q^{4} t^{4} + 6  q^{3} t^{4} + 19  q^{2} t^{4} + q^{2} t^{2} + 5  q t^{2} + 1$}  \\ \cline{2-7} 
                         & $\bm{P^{\str}(T_{G}^2/W)}$      & \multicolumn{3}{|l|}{$65  t^{6} + 26  t^{4} + 6  t^{2} + 1$} & $\bm{\chi^{\str}(T_{G}^2/W)}$  & \multicolumn{1}{|l|}{$98$} \\ \cline{2-7} 
                         & $\bm{\mu^{\str}_c(T_{G}^2/W)}$  & \multicolumn{5}{|l|}{$q^{6} t^{12} + 5  q^{5} t^{10} + q^{4} t^{10} + 19  q^{4} t^{8} + 6  q^{3} t^{8} + q^{2} t^{8} + 40  q^{3} t^{6} + 19  q^{2} t^{6} + 5  q t^{6} + t^{6}$} \\ \cline{2-7} 
                         & $\bm{P^{\str}_c(T_{G}^2/W)}$    & \multicolumn{2}{|l|}{$t^{12} + 6  t^{10} + 26  t^{8} + 65  t^{6}$} & $\bm{E^{\str}(T_{G}^2/W)}$ & \multicolumn{2}{|l|}{$q^{6} + 5  q^{5} + 20  q^{4} + 46  q^{3} + 20  q^{2} + 5  q + 1$} \\ \hline\hline

 & \multicolumn{6}{|c|}{\textbf{Rank $\bm{4}$}} \\ \hline

\multirow{5}{*}{$\GL_4$} & $\bm{\mu^{\str}(T_{G}^2/W)}$    & \multicolumn{5}{|l|}{\tiny $\begin{array}{l}q^{8} t^{8} + q^{7} t^{8} + 2  q^{7} t^{7} + 2  q^{6} t^{8} + 4  q^{6} t^{7} + q^{5} t^{8} + 2  q^{6} t^{6} + 6  q^{5} t^{7} + 7  q^{5} t^{6} + 2  q^{4} t^{7} + 2  q^{5} t^{5} + 8  q^{4} t^{6} \\+ 8  q^{4} t^{5} + q^{3} t^{6} + 2  q^{4} t^{4} + 6  q^{3} t^{5} + 7  q^{3} t^{4} + 2  q^{3} t^{3} + 2  q^{2} t^{4} + 4  q^{2} t^{3} + 2  q^{2} t^{2} + q t^{2} + 2  q t + 1\end{array}$}  \\ \cline{2-7} 
                         & $\bm{P^{\str}(T_{G}^2/W)}$      & \multicolumn{3}{|l|}{$5  t^{8} + 14  t^{7} + 18  t^{6} + 16  t^{5} + 11  t^{4} + 6  t^{3} + 3  t^{2} + 2  t + 1$} & $\bm{\chi^{\str}(T_{G}^2/W)}$  & \multicolumn{1}{|l|}{$0$} \\ \cline{2-7} 
                         & $\bm{\mu^{\str}_c(T_{G}^2/W)}$  & \multicolumn{5}{|l|}{\tiny $\begin{array}{l}q^{8} t^{16} + 2  q^{7} t^{15} + q^{7} t^{14} + 2  q^{6} t^{14} + 4  q^{6} t^{13} + 2  q^{6} t^{12} + 2  q^{5} t^{13} + 7  q^{5} t^{12} + 6  q^{5} t^{11} + 2  q^{4} t^{12} + q^{5} t^{10} + 8  q^{4} t^{11} \\+ 8  q^{4} t^{10} + 2  q^{3} t^{11} + 2  q^{4} t^{9} + 7  q^{3} t^{10} + 6  q^{3} t^{9} + 2  q^{2} t^{10} + q^{3} t^{8} + 4  q^{2} t^{9} + 2  q^{2} t^{8} + 2  q t^{9} + q t^{8} + t^{8}\end{array}$} \\ \cline{2-7} 
                         & $\bm{P^{\str}_c(T_{G}^2/W)}$    & \multicolumn{2}{|l|}{\tiny $\begin{array}{l}t^{16} + 2  t^{15} + 3  t^{14} + 6  t^{13} + 11  t^{12} \\ + 16  t^{11} + 18  t^{10} + 14  t^{9} + 5  t^{8}\end{array}$} & $\bm{E^{\str}(T_{G}^2/W)}$ & \multicolumn{2}{|l|}{$q^{8} - q^{7} - q + 1$} \\ \hline\hline    

\multirow{5}{*}{\begin{tabular}{@{}c@{}}$\SL_5$,\\ $\PGL_5$\end{tabular}} & $\bm{\mu^{\str}(T_{G}^2/W)}$    & \multicolumn{5}{|l|}{\tiny $\begin{array}{l}q^{8} t^{8} + q^{7} t^{8} + 2  q^{6} t^{8} + 2  q^{6} t^{7} + 2  q^{5} t^{8} + q^{6} t^{6} + 4  q^{5} t^{7} + 25  q^{4} t^{8} + 2  q^{5} t^{6} + 4  q^{4} t^{7} + 4  q^{4} t^{6} + 2  q^{4} t^{5} \\
+ 2  q^{3} t^{6} + q^{4} t^{4} + 4  q^{3} t^{5} + 2  q^{3} t^{4} + 2  q^{2} t^{4} + 2  q^{2} t^{3} + q^{2} t^{2} + q t^{2} + 1\end{array}$}  \\ \cline{2-7} 
                         & $\bm{P^{\str}(T_{G}^2/W)}$      & \multicolumn{3}{|l|}{$31  t^{8} + 10  t^{7} + 9  t^{6} + 6  t^{5} + 5  t^{4} + 2  t^{3} + 2  t^{2} + 1$} & $\bm{\chi^{\str}(T_{G}^2/W)}$  & \multicolumn{1}{|l|}{$30$} \\ \cline{2-7} 
                         & $\bm{\mu^{\str}_c(T_{G}^2/W)}$  & \multicolumn{5}{|l|}{\tiny $\begin{array}{l}q^{8} t^{16} + q^{7} t^{14} + q^{6} t^{14} + 2  q^{6} t^{13} + 2  q^{6} t^{12} + 2  q^{5} t^{12} + 4  q^{5} t^{11} + q^{4} t^{12} + 2  q^{5} t^{10} + 2  q^{4} t^{11} + 4  q^{4} t^{10} + 4  q^{4} t^{9} \\
                         + 2  q^{3} t^{10} + 25  q^{4} t^{8} + 4  q^{3} t^{9} + q^{2} t^{10} + 2  q^{3} t^{8} + 2  q^{2} t^{9} + 2 q^{2} t^{8} + q t^{8} + t^{8}\end{array}$} \\ \cline{2-7} 
                         & $\bm{P^{\str}_c(T_{G}^2/W)}$    & \multicolumn{2}{|l|}{\tiny $\begin{array}{l}t^{16} + 2  t^{14} + 2  t^{13} + 5  t^{12} + 6  t^{11} \\+ 9  t^{10} + 10  t^{9} + 31 t^{8}\end{array}$} & $\bm{E^{\str}(T_{G}^2/W)}$ & \multicolumn{2}{|l|}{$q^{8} + q^{7} + q^{6} + 24  q^{4} + q^{2} + q + 1$} \\ \hline\hline 

\multirow{5}{*}{\begin{tabular}{@{}c@{}}$\SO_9$,\\ $\Sp_8$\end{tabular}} & $\bm{\mu^{\str}(T_{G}^2/W)}$    & \multicolumn{5}{|l|}{\tiny $\begin{array}{l}q^{8} t^{8} + 5 q^{7} t^{8} + 20 q^{6} t^{8} + 59 q^{5} t^{8} + q^{6} t^{6} + 105 q^{4} t^{8} + 6 q^{5} t^{6} + 25  q^{4} t^{6} + 59  q^{3} t^{6} \\ + q^{4} t^{4} + 6  q^{3} t^{4} + 20  q^{2} t^{4} + q^{2} t^{2} + 5  q t^{2} + 1
\end{array}$}  \\ \cline{2-7} 
                         & $\bm{P^{\str}(T_{G}^2/W)}$      & \multicolumn{3}{|l|}{ $190  t^{8} + 91  t^{6} + 27  t^{4} + 6  t^{2} + 1$} & $\bm{\chi^{\str}(T_{G}^2/W)}$  & \multicolumn{1}{|l|}{$315$} \\ \cline{2-7} 
                         & $\bm{\mu^{\str}_c(T_{G}^2/W)}$  & \multicolumn{5}{|l|}{\tiny $\begin{array}{l}q^{8} t^{16} + 5  q^{7} t^{14} + q^{6} t^{14} + 20  q^{6} t^{12} + 6  q^{5} t^{12} + q^{4} t^{12} + 59  q^{5} t^{10} + 25  q^{4} t^{10} \\ + 6  q^{3} t^{10} + 105  q^{4} t^{8} + q^{2} t^{10} + 59  q^{3} t^{8} + 20  q^{2} t^{8} + 5  q t^{8} + t^{8}\end{array}$} \\ \cline{2-7} 
                         & $\bm{P^{\str}_c(T_{G}^2/W)}$    & \multicolumn{2}{|l|}{ $t^{16} + 6  t^{14} + 27  t^{12} + 91  t^{10} + 190  t^{8}$} & $\bm{E^{\str}(T_{G}^2/W)}$ & \multicolumn{2}{|l|}{\tiny $\begin{array}{l}q^{8} + 5  q^{7} + 21  q^{6} + 65  q^{5} + 131  q^{4} \\+ 65  q^{3} + 21  q^{2} + 5  q + 1\end{array}$} \\ \hline\hline   

\multirow{5}{*}{$\SO_8$} & $\bm{\mu^{\str}(T_{G}^2/W)}$    & \multicolumn{5}{|l|}{\tiny $\begin{array}{l}q^{8} t^{8} + q^{7} t^{8} + 13  q^{6} t^{8} + 28  q^{5} t^{8} + q^{6} t^{6} + 61  q^{4} t^{8} \\ + 3  q^{5} t^{6} + 18  q^{4} t^{6} + 28  q^{3} t^{6} + q^{4} t^{4} + 3  q^{3} t^{4} + 13  q^{2} t^{4} + q^{2} t^{2} + q t^{2} + 1\end{array}$}  \\ \cline{2-7} 
                         & $\bm{P^{\str}(T_{G}^2/W)}$      & \multicolumn{3}{|l|}{$104  t^{8} + 50  t^{6} + 17 t^{4} + 2  t^{2} + 1$} & $\bm{\chi^{\str}(T_{G}^2/W)}$  & \multicolumn{1}{|l|}{$174$} \\ \cline{2-7} 
                         & $\bm{\mu^{\str}_c(T_{G}^2/W)}$  & \multicolumn{5}{|l|}{\tiny $\begin{array}{l}q^{8} t^{16} + q^{7} t^{14} + q^{6} t^{14} + 13  q^{6} t^{12} + 3  q^{5} t^{12} + q^{4} t^{12} + 28  q^{5} t^{10} + 18  q^{4} t^{10} \\+ 3  q^{3} t^{10} + 61  q^{4} t^{8} + q^{2} t^{10} + 28  q^{3} t^{8} + 13  q^{2} t^{8} + q t^{8} + t^{8}\end{array}$} \\ \cline{2-7} 
                         & $\bm{P^{\str}_c(T_{G}^2/W)}$    & \multicolumn{2}{|l|}{$t^{16} + 2  t^{14} + 17  t^{12} + 50  t^{10} + 104  t^{8}$} & $\bm{E^{\str}(T_{G}^2/W)}$ & \multicolumn{2}{|l|}{\scalebox{0.8}{ $q^{8} + q^{7} + 14  q^{6} + 31  q^{5} + 80  q^{4} + 31  q^{3} + 14  q^{2} + q + 1$}} \\ \hline\hline 

\multirow{4}{*}{$\F_4$} & $\bm{\mu^{\str}(T_{G}^2/W)}$   & \multicolumn{5}{|l|}{\tiny $\begin{array}{l}q^{8} t^{8} + 2  q^{7} t^{8} + 8  q^{6} t^{8} + 22  q^{5} t^{8} + q^{6} t^{6} + 59  q^{4} t^{8} + 2  q^{5} t^{6} \\+ 9  q^{4} t^{6} + 22  q^{3} t^{6} + q^{4} t^{4} + 2  q^{3} t^{4} + 8  q^{2} t^{4} + q^{2} t^{2} + 2  q t^{2} + 1\end{array}$}  \\ \cline{2-7} 
                         & $\bm{P^{\str}(T_{G}^2/W)}$      & \multicolumn{3}{|l|}{ $92  t^{8} + 34  t^{6} + 11  t^{4} + 3  t^{2} + 1$} & $\bm{\chi^{\str}(T_{G}^2/W)}$  & \multicolumn{1}{|l|}{$141$} \\ \cline{2-7} 
                         & $\bm{\mu^{\str}_c(T_{G}^2/W)}$  & \multicolumn{5}{|l|}{\tiny {$\begin{array}{l}q^{8} t^{16} + 2  q^{7} t^{14} + q^{6} t^{14} + 8  q^{6} t^{12} + 2  q^{5} t^{12} + q^{4} t^{12} + 22  q^{5} t^{10} \\+ 9  q^{4} t^{10} + 2  q^{3} t^{10} + 59  q^{4} t^{8} + q^{2} t^{10} + 22  q^{3} t^{8} + 8  q^{2} t^{8} + 2  q t^{8} + t^{8}\end{array}$}} \\ \cline{2-7} 
                         & $\bm{P^{\str}_c(T_{G}^2/W)}$    & \multicolumn{2}{|l|}{ $t^{16} + 3  t^{14} + 11  t^{12} + 34  t^{10} + 92  t^{8}$} & $\bm{E^{\str}(T_{G}^2/W)}$ & \multicolumn{2}{|l|}{\scalebox{0.8}{ $q^{8} + 2  q^{7} + 9  q^{6} + 24  q^{5} + 69  q^{4} + 24  q^{3} + 9  q^{2} + 2  q + 1$}} \\ \hline
                   
\end{tabular}
\caption{Stringy cohomological invariants of the elliptic curve, for various classical and exceptional groups $G$ up to rank $4$. For simplicity, we set the variable $q = uv$.}\label{tab:summary}
\end{table}

\section{On the Batyrev positivity conjecture} \label{sec:batyrev}

As mentioned in Section \ref{sec:stringy-MH-poly}, the stringy $E$-polynomial was introduced by Batyrev \cite{Ba} to generalize classical Hodge--Deligne polynomials to algebraic varieties with mild singularities. Motivated by the search for an intrinsic cohomology theory for general singular spaces, Batyrev's positivity conjecture \cite[Conjecture 3.10]{Ba} claims that whenever the stringy $E$-function of a projective Gorenstein canonical variety $X$ is actually a polynomial, the coefficients, i.e.\ the corresponding stringy Hodge numbers, must be non-negative.

It is worth pointing out that for quotient singularities, as the ones in this paper, a modified non-negativity property is automatic. Indeed, the stringy mixed Hodge numbers, i.e.\ the coefficients of the stringy mixed Hodge polynomial, are automatically non-negative because they coincide with the dimensions of Chen--Ruan orbifold cohomology groups. While Chen--Ruan cohomology guarantees positivity in the orbifold context, general stringy Hodge numbers are defined via motivic integrals over log-resolutions rather than intrinsic vector spaces, so the general Batyrev's positivity conjecture is highly non-trivial in this more general context. 

However, even on our quotient singularity context, the non-negativity of the coefficients of the stringy $E$-polynomial is not clear, since it corresponds to setting $t = -1$ in the stringy mixed Hodge polynomial with compact support,
$$
    E^\str(T^r/W)(u,v) = \mu_c^\str(T^r/W)(-1, u, v).
$$
Inspired by Batyrev, we propose the following positivity conjecture, noting that $T^r/W$ is not projective, so it does not follow from Batyrev's.

\begin{conj}[Rank 2 character variety positivity conjecture]\label{conj:positivity}
    For any connected semisimple group $G$ with maximal torus $T$ and Weyl group $W$, the stringy $E$-polynomial of $T^2/W$ has non-negative coefficients.
\end{conj}

The examples computed in this work, including all the semisimple ones in Table \ref{tab:summary}, satisfy this conjecture, providing strong empirical evidence for it. On the other hand, the hypotheses of Conjecture \ref{conj:positivity} are sharp for the following reasons:
\begin{itemize}
    \item For $T^r/W$ in general, for $r = 4$ (corresponding to an abelian surface) and $G = \SL_3$, we have
$$
    E^\str(T_{\SL_3}^4/W) = q^{8} + 7 \, q^{6} - 8 \, q^{5} + 108 \, q^{4} - 8 \, q^{3} + 7 \, q^{2} + 1\; ,
$$
as computed with \cite{code_repository}. A similar phenomenon occurs for odd $r$. This implies that the conjecture must focus on character varieties of elliptic curves, corresponding to $r = 2d = 2$.
    \item For $G$ reductive non-semisimple, for $G = \GL_2$, from Table \ref{tab:summary} we get
$$
    E^\str(T_{\GL_2}^2/W) = q^{4} - q^{3} - q + 1\; ,
$$
which has negative coefficients, even on an elliptic curve. This implies that we must require $G$ to be semisimple, so that the center does not modify the coefficients.
\end{itemize}

A curious feature is that, although Conjecture \ref{conj:positivity} does not hold in general for arbitrarily large even $r$, it \emph{does hold} for the two cases computed in subsections \ref{ssec:SO7} and \ref{ssec:G2}. 

\begin{lem}\label{lem:polynomial-t2}
Let $d \geq 1$. For $G = \SO_7$ and $G = \G_2$, the compactly supported stringy mixed Hodge polynomial of $T_{G}^{2d}/W$ is a polynomial in the variable $t^2$, in other words, $\mu_c^\str(T_{G}^{2d}/W) \in \ZZ[t^2, u, v]$.
\end{lem}

\begin{proof}
Let us consider the case of $G = \G_2$. By the results of Section \ref{ssec:G2}, we have
\begin{align*}
   \mu^{\mathrm{str}}_c(T_{\G_2}^{2d}/W) = & \frac{t^{4d}}{12} \bigg[ 6(t^2u^2v^2-1)^{2d} + (uv)^{2d} \left( 2 \cdot 4^{2d} + 6 \cdot 3^{2d} +6 \cdot 2^{2d} + 22 \right)  \\
   & + \underbrace{(tuv+1)^{4d} + (tuv-1)^{4d}}_{\textrm{(i)}} + \underbrace{2(t^2u^2v^2+tuv+1)^{2d} + 2(t^2u^2v^2-tuv+1)^{2d}}_{\textrm{(ii)}} \\
    & + 12(uv)^{d} ( \underbrace{(tuv+1)^{2d} + (tuv-1)^{2d}}_{\textrm{(iii)}} ) \bigg]\; .
\end{align*}
The first two terms in the sum are obviously polynomials in $t^2$. For the three pairs of terms labeled as (i), (ii) and (iii), notice that they all have the form $(Q + tP)^{2m} + (Q - tP)^{2m}$, with $P, Q \in \ZZ[t^2, u, v]$. Expanding both terms through the binomial theorem, we get
$$
(Q + tP)^{2m} + (Q - tP)^{2m} = 2\sum_{i = 0}^m \left(\begin{matrix}
2m \\ 2i 
\end{matrix}\right) (tP)^{2i} Q^{2(m-i)}\; ,
$$
which are, for all such $P,Q$, obviously polynomials in $t^2$. The case of $G = \SO_7$ is completely analogous, taking into account that $(tuv+1)(tuv-1) = (t^2u^2v^2-1)$ is a polynomial in $t^2$.
\end{proof}

\begin{cor}
\label{cor:G2_SO7_conj}
    Let $d \geq 1$. For $G = \SO_7$ and $G = \G_2$, the stringy $E$-polynomial of $T_{G}^{2d}/W$ has non-negative coefficients.
\end{cor}

\begin{proof}
    By Lemma \ref{lem:polynomial-t2}, $\mu_c^\str(T_{G}^{2d}/W)$ is a polynomial in the variable $t^2$. Therefore, we have
    $$
        E^\str(T_{G}^{2d}/W)(u,v) = \mu_c^\str(T_{G}^{2d}/W)(-1, u, v) = \mu_c^\str(T_{G}^{2d}/W)(1, u, v)\; ,
    $$
    and the result follows from the non-negativity of the coefficients of the mixed Hodge polynomial $\mu_c^\str(T_{G}^{2d}/W)$.
\end{proof}

It is plausible that the same property of Lemma \ref{lem:polynomial-t2}, namely that the mixed Hodge polynomial is a polynomial in $t^2$, holds for other special orthogonal or semisimple exceptional ones, leading to a direct proof of Conjecture \ref{conj:positivity} for these groups. However, this feature does not hold for all semisimple groups, as can be seen in Table \ref{tab:summary}, with  $G = \SL_3$.

We finish this section by pointing out another curious property of the stringy $E$-polynomial that follows directly from Theorem \ref{thm:algorithm_CV}.

\begin{prop}\label{prop:palindromic}
For any connected reductive group $G$ with maximal torus $T$ and Weyl group $W$, the stringy $E$-polynomial of $T^{2d}/W$ is palindromic in the variable $q=uv$ for all $d \geq 1$.
\end{prop}

\begin{proof}
By Corollary \ref{cor:compact-support}, and writing the stringy $E$-polynomial of $T^{2d}/W$ as a polynomial in the variable $q := uv$, we have
    \[
        E^{\str}(T^{2d}/W)(q)=\sum_{[w] \in \textup{Conj}(W)}
        \frac{1}{|C(w)|}\left(\sum_{g \in C(w)} \, q^{2d\age(w)}n_{w}(g)^{2d} \det\left(\phi_w(g)-q\,I_{\Phi_w}\right)^{2d} \right),
    \]
where $2\age(w) = \rk(I_\Lambda - w)$.

Now, observe that a polynomial $P(q)$ of degree $n$ is palindromic if and only if $q^nP(q^{-1})= P(q)$. In the case of the stringy $E$-polynomial, its degree is $\dim T^{2d} = 2dr$, with $r = \dim T = \rk G=\rk \Lambda$. Therefore, if we show that such symmetry holds for each term of the form
$$
P_w(q) := \sum_{g \in C(w)} \, q^{2d\age(w)}n_{w}(g)^{2d} \det\left(\phi_w(g)-q\,I_{\Phi_w}\right)^{2d},
$$
for every $w \in W$, the result follows.

Indeed, recalling $s_w = \rk \coker(I_{\Lambda} - w) = \rk \Phi_w$, we have $r = \rk \coker(I_{\Lambda} - w) + \rk \image(I_{\Lambda} - w) = s_w + 2\age(w)$ or, equivalently, $2(r - \age(w)) = 2(\age(w) + s_w)$. Therefore, we have
\begin{align}
\begin{split}\label{eq:palindromic-term}
q^{2dr}P_w(q^{-1}) & = \sum_{g \in C(w)} \, q^{2dr-2d\age(w)}n_{w}(g)^{2d} \det\left(\phi_w(g)-q^{-1}\,I_{\Phi_w}\right)^{2d} \\
& = \sum_{g \in C(w)} \, q^{2d\age(w)}n_{w}(g)^{2d} \left(q^{s_w}\det\left(\phi_w(g)-q^{-1}\,I_{\Phi_w}\right)\right)^{2d}.
\end{split}
\end{align}
Now, since $s_w$ is exactly the dimension of the space on which the determinant is taken, we have
\begin{align*}
q^{s_w}\det\left(\phi_w(g)-q^{-1}\,I_{\Phi_w}\right) & = \det\left(q\phi_w(g)-I_{\Phi_w}\right) = \det\left(\phi_w(g)(qI_{\Phi_w}-\phi_w(g)^{-1})\right) \\
& = \det(\phi_w(g)) \det\left(qI_{\Phi_w}-\phi_w(g)^{-1}\right) = \pm \det\left(\phi_w(g^{-1}) - qI_{\Phi_w}\right),
\end{align*}
where in the last equality we have used that $\det(\phi_w(g)) = \pm1$ (since $g$ is an automorphism), and that $\phi_w(g)^{-1} = \phi_w(g^{-1})$. Additionally, recall that $n_w(g)$ can be understood as the number of fixed points of $g$ when acting on the torsion part of $\coker(I_\Lambda - w)$, and this number agrees with the number of fixed points of $g^{-1}$ on the same torsion part, so $n_w(g)= n_w(g^{-1})$.

Thus, returning to (\ref{eq:palindromic-term}), we get
\begin{align*}
q^{2dr}P_w(q^{-1}) & = \sum_{g \in C(w)} \, q^{2d\age(w)}n_{w}(g)^{2d} \left(q^{s_w}\det\left(\phi_w(g)-q^{-1}\,I_{\Phi_w}\right)\right)^{2d} \\
& = \sum_{g \in C(w)} \, q^{2d\age(w)}n_{w}(g^{-1})^{2d} \left(\det\left(\phi_w(g^{-1}) - qI_{\Phi_w}\right)\right)^{2d} = P_w(q)\; ,
\end{align*}
as we wanted to show.
\end{proof}

\begin{rem}
    By Poincar\'e duality (see Remark \ref{rem:mu-and-muc}), the previous proof of Proposition \ref{prop:palindromic} actually shows that $\mu_c^\str(T^{2d}/W)(-1,u,v) = \mu^\str(T^{2d}/W)(-1,u,v)$, so in this case the stringy $E$-polynomial can be computed by setting $t = -1$ in both the regular and the compactly supported stringy mixed Hodge polynomial.
\end{rem}

\begin{rem}
    Upon completion of this work, we noticed the papers \cite{SU} and \cite{dais2026unified} which study families of projective varieties and provide counterexamples to Batyrev's positivity conjecture. This highlights the subtleties of extending pure cohomological interpretations to stringy invariants on arbitrary Gorenstein singularities. These counterexamples do not include our character varieties, since $T^2/W$ is non-projective.
\end{rem}

\subsection*{AI Declaration} 
During the preparation of this work, models from OpenAI's GPT-5 family, including GPT-5.6, were used to suggest the use of the exponential map on formal completions as an ingredient in the proof of Theorem \ref{thm:resolution_ABC}, to generate an initial draft of the accompanying SageMath script, and to cross-check our methods against explicit examples of stringy mixed Hodge polynomials in the $\SL_n$ and $\GL_n$ cases, already present in the literature.
GPT-6 Astra was used in the final proofreading of the manuscript. 

The research design, formulation and proofs of the results, and writing of the manuscript were carried out by the authors, who take full responsibility for the content of the work. No funding from AI companies was received during the development of this work.

\bibliographystyle{amsalpha}
\bibliography{bibliography}

\newpage

\appendix

\section{Dual endomorphisms on dual lattices}\label{app:proof-mirror-symmetry}

The explicit description of the stringy mixed Hodge polynomial in terms of actions on the character lattice developed in this work will allow us to provide a more direct proof in terms of the elements appearing in Theorem \ref{thm:algorithm_CV}. Indeed, in this Appendix we reduce these calculations to some determinants and point-counting in the free and torsion parts of the character lattice, showing their agreement for an action and its dual.

Given a lattice $\Lambda$, we shall denote by $\Lambda^\vee = \Hom_{\Ab}(\Lambda, \ZZ)$ its dual lattice. A homomorphism $h: \Lambda_1 \to \Lambda_2$ induces a dual homomorphism $h^\vee: \Lambda_2^\vee \to \Lambda_1^\vee$ by $(h^\vee\alpha)(\lambda) = \alpha(h(\lambda))$, for $\alpha \in \Lambda_2^\vee$ and $\lambda \in \Lambda_1$.

Now, let us fix an endomorphism $f: \Lambda \to \Lambda$ and consider the respective cokernels
$$
    K:=\coker(f) = \Lambda/\image(f)\quad , \quad \hat{K}:=\coker(f^\vee) = \Lambda^\vee/\image(f^\vee)\; .
$$
Notice that, in general, there is no isomorphism between the dual lattice $K^\vee=\Hom(K, \ZZ)$ and $\hat{K}$, for instance, because the former is always torsion free.

Actually, the torsion and free parts of $\hat{K}$ can be identified explicitly. We recall here the procedure for completeness. Consider the morphism
$$
    \Psi: \hat{K} \to (\ker f)^\vee\; , 
$$
where a class $[\alpha] \in \Lambda^\vee/\image(f^\vee)$ is mapped to the restriction $\alpha|_{\ker f} \in (\ker f)^\vee$. Notice that $\Psi$ is well-defined since $\Psi([f^\vee\alpha]) = f^\vee\alpha|_{\ker f} = (\alpha \circ f)|_{\ker f} = 0$ for any $\alpha \in \Lambda^\vee$.

Actually, the kernel of $\Psi$ can be explicitly identified as the cokernel of the dual $\iota^\vee: \Lambda^\vee \to \image(f)^\vee$ of the inclusion map $\iota: \image(f) \hookrightarrow \Lambda$, that is $\ker(\Psi) = \coker(\iota^\vee) = \image(f)^\vee/\image(\iota^\vee)$.
This kernel is a finite group since, as $\QQ$-vector spaces, $\ker(\Psi) \otimes_\ZZ \QQ = \image(f_\QQ)^*/\image(\iota_\QQ^*) = 0$, where $f_\QQ: \Lambda \otimes_\ZZ \QQ \to \Lambda \otimes_\ZZ \QQ$ and $\iota_\QQ: \image(f_\QQ)\hookrightarrow \Lambda \otimes_\ZZ \QQ$ are the $\QQ$-extensions to a linear maps of rational vector spaces, and $\iota_\QQ^*: (\Lambda \otimes_\ZZ \QQ)^* \to \image(f_\QQ)^*$ its dual map, which is surjective.

In this manner, we get a short exact sequence
$$
    0 \longrightarrow \coker(\iota^\vee) \longrightarrow \hat{K} \stackrel{\Psi}{\longrightarrow} (\ker f)^\vee \longrightarrow 0\; .
$$
Furthermore, $(\ker f)^\vee$ is a free abelian group and $\hat{K} \otimes_\ZZ \QQ = \coker(f^*_\QQ) \cong (\ker f_\QQ)^*$, which shows that $\coker(\iota^\vee)$ is exactly the torsion part of $\hat{K}$, with free part $(\ker f)^\vee$, so we get a decomposition into free and torsion parts respectively
$$
    \hat{K} \cong (\ker f)^\vee \oplus \coker(\iota^\vee)\; .
$$

For the torsion part $\hat{D} :=\coker(\iota^\vee) = \image(f)^\vee/\image(\iota^\vee)$ of $\hat{K}$, there exists a pairing with the torsion part $D$ of $K$,
\begin{equation*}
    \langle-, -\rangle: D \otimes_\ZZ \hat{D} \to \QQ/\ZZ\; ,
\end{equation*}
given as follows. For $[\lambda] \in D$, let $n$ be its order, so that $n[\lambda] = 0$ in $K$, or equivalently, $n\lambda \in \image(f)$. In this manner, given $[\alpha] \in \coker(\iota^\vee)$ with $\alpha: \image(f) \to \ZZ$, we define
$$
    \langle[\lambda],[\alpha]\rangle := \frac{\alpha(n\lambda)}{n} \mod \ZZ\; .
$$

\begin{lem}\label{lem:pairing-well-defined}
The pairing $\langle-, -\rangle: D \otimes_\ZZ \hat{D} \to \QQ/\ZZ$ described above is well-defined and perfect.
\end{lem}

\begin{proof}
To prove that the pairing is well-defined, we must show that $\langle -, [\alpha]\rangle = 0$ for any $\alpha \in \image(\iota^\vee)$, as well as $\langle [\lambda], -\rangle = 0$ for any $\lambda \in \image(f)$. We will show the former condition, the latter being analogous. Indeed, if $\alpha \in \image(\iota^\vee) \subseteq \image(f)^\vee$, then $\alpha$ is defined on the whole $\Lambda$ and thus $\alpha(n\lambda)=n\alpha(\lambda)$, so $\alpha(n\lambda)/n$ is an integer and therefore it vanishes in $\QQ/\ZZ$. Using this description, in a normal basis we have
\begin{equation}\label{eq:pairing-perfect}
    \left\langle [\lambda], [\alpha]\right\rangle = \sum_i \frac{\lambda_i \alpha_i}{n_i} \mod \ZZ\; ,
\end{equation}
where $\lambda_i$ and $\alpha_i$ are the coordinates of $[\lambda]$ and $[\alpha]$ in the normal basis respectively, and $D \cong \hat{D} \cong \prod_i \ZZ_{n_i}$ is the normal factor decomposition. From (\ref{eq:pairing-perfect}), we directly see that the pairing is perfect.
\end{proof}

\begin{rem}
In a similar spirit to the proof of Lemma \ref{lem:pairing-well-defined}, it can be proven that $\langle[\lambda],[\alpha]\rangle$ does not depend on the choice of the integer $n \neq 0$ such that $n\lambda \in \image{f}$. In fact, given $[\lambda] \in D$ of order $n$, let $m \in \NN$ be another nonzero integer with $m\lambda \in \image(f)$. Then $m = ns$ for some integer $s$, and thus $\alpha(m\lambda)/m=\alpha(ns\lambda)/ns = \alpha(n\lambda)/n$.
\end{rem}

Now, consider another endomorphism $h: \Lambda \to \Lambda$ with $hf=fh$. Then $h$ descends to a morphism $K\to K$, also denoted by $h$. Similarly, the dual morphism $h^\vee: \Lambda^\vee \to \Lambda^\vee$ descends to a morphism $h^\vee: \hat{K} \to \hat{K}$. Denote their respective torsion parts by
$$
    \tau(h) = h|_D: D \to D\quad , \quad
    \hat{\tau}(h^\vee) = h^\vee|_{\hat{D}}: \hat{D} \to \hat{D}\; .
$$

\begin{prop}
    The morphisms $\tau(h): D \to D$ and $\hat{\tau}(h^\vee): \hat{D} \to \hat{D}$ are adjoint with respect to the pairing $\langle-, -\rangle: D \otimes_\ZZ \hat{D} \to \QQ/\ZZ$, that is
    $$
        \langle \tau(h)[\lambda], [\alpha]\rangle = \langle [\lambda], \hat{\tau}(h^\vee)[\alpha]\rangle\; ,
    $$
    for all $[\lambda] \in D$ and $[\alpha] \in \hat{D}$.
\end{prop}

\begin{proof}
Let $n \in \NN$ so that $n\lambda \in \image(f)$, say $n\lambda = f(\eta)$. Observe that, in that case, we also have $nh(\lambda) = h(n\lambda) = h(f(\eta)) = f(h(\eta))$, so $nh(\lambda)$ also belongs to $\image(f)$. In this manner, we have
$$
\langle [\lambda], h^\vee[\alpha]\rangle = \frac{h^\vee\alpha(n\lambda)}{n} = \frac{\alpha(h(n\lambda))}{n} = \frac{\alpha(nh(\lambda))}{n} = \langle h[\lambda], [\alpha]\rangle\; ,
$$
as we wanted to prove.
\end{proof}

\begin{cor}\label{cor:number-points}
    The number of points of $\coker(\tau(h)) = D/\image \tau(h)$ and $\coker(\hat{\tau}(h^\vee)) = \hat{D}/\image \hat{\tau}(h^\vee)$ agree.
\end{cor}

\begin{proof}
    Since $\tau(h)$ and $\hat{\tau}(h^\vee)$ are adjoint, the pairing between $D$ and $\hat{D}$ induces an isomorphism between $\ker(\tau(h))$ and the Pontryagin dual $ \coker(\hat{\tau}(h^\vee))^P := \Hom_{\Ab}(\coker(\hat{\tau}(h^\vee)), \QQ/\ZZ)$. Indeed, if $[\lambda] \in \ker(\tau(h))$, then $\langle [\lambda], -\rangle: \hat{D} \to \QQ/\ZZ$ is a functional vanishing on $\image(\hat{\tau}(h^\vee))$, so it descends to a morphism $\coker(\hat{\tau}(h^\vee))=\hat{D}/\image{\hat{\tau}(h^\vee)} \to \QQ/\ZZ$. Reciprocally, since the pairing is perfect, any $\varphi:  \hat{D} \to \QQ/\ZZ$ is of the form $\varphi = \langle [\lambda], -\rangle$ for some $[\lambda] \in D$, and if $\varphi$ vanishes on $\image(\hat{\tau}(h^\vee))$, then $[\lambda] \in \ker(\tau(h))$ since $\langle \tau(h)[\lambda], -\rangle = \langle [\lambda], \hat{\tau}(h^\vee)-\rangle = \varphi \circ \hat{\tau}(h^\vee) = 0$.

    Therefore, using that on any finite abelian group we have $|\coker(\hat{\tau}(h^\vee))^P|=|\coker(\hat{\tau}(h^\vee))|$ and $|\ker(\tau(h))| = |\coker(\tau(h))|$, the result follows.
\end{proof}

\section{Stringy mixed Hodge polynomials for exceptional groups}\label{app:calculations}

In this appendix, we show the stringy mixed Hodge polynomials of all the semisimple exceptional groups for the normalization of the identity component of the character variety $\hat{\mathfrak{X}}_{r}(G) = T_{G}^r/W$, for arbitrary $r$, as well as their $E$-polynomial and Euler characteristic. We also specialize these polynomials in the case $r = 2$, corresponding to an elliptic curve, where $\mathfrak{X}_{r}(G)$ is actually normal.

For simplicity, all the polynomials are written in the variable $q := uv$.

\subsection{Group $G = \G_2$} This case was already computed in full detail in Section \ref{ssec:G2}. We repeat it here for completeness. 

Applying Theorem \ref{thm:algorithm_CV}, the stringy mixed Hodge polynomial is the following.
\allowdisplaybreaks[4]
\begin{align*}
{\normalsize \mu^{\str}(T_{\G_2}^r/W)} =& \frac{1}{12} \, {\left(2 \cdot 4^{r} + 2 \cdot 3^{r + 1} + 3 \cdot 2^{r + 1} + 22\right)} q^{ r} t^{2r}+ \left(q t^{2}\right)^{\frac{1}{2} \, r} {\left({\left(q t + 1\right)}^{r} + {\left(-q t + 1\right)}^{r}\right)} \\
& + \frac{1}{12} \, {\left(q^{2} t^{2} + 2 \, q t + 1\right)}^{r} + \frac{1}{6} \, {\left(q^{2} t^{2} + q t + 1\right)}^{r} + \frac{1}{6} \, {\left(q^{2} t^{2} - q t + 1\right)}^{r} \\
& + \frac{1}{12} \, {\left(q^{2} t^{2} - 2 \, q t + 1\right)}^{r} + \frac{1}{2} \, {\left(-q^{2} t^{2} + 1\right)}^{r}.
\end{align*}
From this, setting $t = -1$ and $q = 1$, we can compute the stringy Euler characteristic to be
$$
\chi^{\str}(T_{\G_2}^r/W) = 4^{r - 1} + 2 \cdot 3^{r - 1} + 3 \cdot 2^{r - 1} + 2\; .
$$

Additionally, the stringy $E$-polynomial (with compact support) can be computed from \cite{code_repository} to give the following.
\begin{align*}
E^{\str}(T_{\G_2}^r/W) =& \frac{1}{6} \cdot 4^{r}  q^{r} + \frac{1}{2} \cdot 3^{r}  q^{r} + 2^{r - 1}  q^{r}  + \frac{1}{12} \, \big(12 \, q^{\frac{1}{2} \, r} {\left({\left(-q + 1\right)}^{r} + {\left(-q - 1\right)}^{r}\right)} + {\left(q^{2} + 2 \, q + 1\right)}^{r} \\
& + 2 \, {\left(q^{2} + q + 1\right)}^{r} + 2 \, {\left(q^{2} - q + 1\right)}^{r} + {\left(q^{2} - 2 \, q + 1\right)}^{r} + 6 \, {\left(q^{2} - 1\right)}^{r} + 22 \, q^{r}\big) \; .
\end{align*}

In the particular case that $r = 2$, corresponding to an elliptic curve, the mixed Hodge polynomial is
\begin{align*}
\mu^{\str}(T_{\G_2}^2/W) =& {\left(q^{4} + 2 \, q^{3} + 11 \, q^{2}\right)} t^{4} + {\left(q^{2} + 2 \, q\right)} t^{2} + 1\; .
\end{align*}

Additionally, the Poincar\'e polynomial is
\begin{align*}
    P^{\str}(T_{\G_2}^2/W) = 14 \, t^{4} + 3 \, t^{2} + 1\; ,
\end{align*}
whereas the $E$-polynomial is
\begin{align*}
E^{\str}(T_{\G_2}^2/W) =& q^{4} + 2 \, q^{3} + 12 \, q^{2} + 2 \, q + 1\; .
\end{align*}
This polynomial satisfies the character positivity conjecture \ref{conj:positivity} (case $d=1$ of Corollary \ref{cor:G2_SO7_conj}) and is palindromic (Proposition \ref{prop:palindromic}). 
The stringy Euler characteristic of the character variety is $\chi^{\str}(T_{\G_2}^2/W)=18$.

\subsection{Group $G = \F_4$}
Applying Theorem \ref{thm:algorithm_CV} and the script in \cite{code_repository}, the stringy mixed Hodge polynomial is the following.
\allowdisplaybreaks[4]
{\tiny
\begin{align*}
{\normalsize \mu^{\str}(T_{\F_4}^r/W)} =& \frac{1}{1152}  {\left(2 \cdot 16^{r} + 48 \cdot 9^{r} + 210 \cdot 8^{r} + 431 \cdot 4^{r + 1} + 128 \cdot 3^{r + 1} + 231 \cdot 2^{r + 4} + 4304\right)} \left(q t^{2}\right)^{2  r} \\
& + \frac{1}{24}  {\left(q^{3} t^{3} + 3  q^{2} t^{2} + 3  q t + 1\right)}^{r} q^{\frac{1}{2}  r} t^{r} + \frac{1}{4}  {\left(q^{3} t^{3} + q^{2} t^{2} + q t + 1\right)}^{r} q^{\frac{1}{2}  r} t^{r} + \frac{3}{8}  {\left(q^{3} t^{3} - q^{2} t^{2} - q t + 1\right)}^{r} q^{\frac{1}{2}  r} t^{r} \\
& + \frac{1}{3}  {\left(q^{3} t^{3} + 1\right)}^{r} q^{\frac{1}{2}  r} t^{r} + \frac{1}{24}  {\left(-q^{3} t^{3} + 3  q^{2} t^{2} - 3  q t + 1\right)}^{r} q^{\frac{1}{2}  r} t^{r} + \frac{1}{4}  {\left(-q^{3} t^{3} + q^{2} t^{2} - q t + 1\right)}^{r} q^{\frac{1}{2}  r} t^{r} \\
& + \frac{3}{8}  {\left(-q^{3} t^{3} - q^{2} t^{2} + q t + 1\right)}^{r} q^{\frac{1}{2}  r} t^{r} + \frac{1}{3}  {\left(-q^{3} t^{3} + 1\right)}^{r} q^{\frac{1}{2}  r} t^{r} + \frac{1}{8}  {\left(2  q^{2} t^{2} + 4  q t + 2\right)}^{r} q^{ r} t^{2r}\\
& + \frac{1}{8}  {\left(2  q^{2} t^{2} - 4  q t + 2\right)}^{r} q^{ r} t^{2r}+ \frac{1}{4}  {\left(2  q^{2} t^{2} + 2\right)}^{r} q^{ r} t^{2r}+ \frac{13}{24}  {\left(q^{2} t^{2} + 2  q t + 1\right)}^{r} q^{ r} t^{2r} + \frac{1}{3}  {\left(q^{2} t^{2} + q t + 1\right)}^{r} q^{ r} t^{2r}\\
& + \frac{1}{3}  {\left(q^{2} t^{2} - q t + 1\right)}^{r} q^{ r} t^{2r}+ \frac{13}{24}  {\left(q^{2} t^{2} - 2  q t + 1\right)}^{r} q^{ r} t^{2r}+ \frac{1}{4}  {\left(q^{2} t^{2} + 1\right)}^{r} q^{ r} t^{2r}+ 2  {\left(-q^{2} t^{2} + 1\right)}^{r} q^{ r} t^{2r}\\
& + \frac{1}{2}  {\left(-2  q^{2} t^{2} + 2\right)}^{r} q^{ r} t^{2r}+ \frac{1}{6}  q^{\frac{3}{2}  r} t^{3r} \big({\left(4  q t + 4\right)}^{r} + 9  {\left(2  q t + 2\right)}^{r} + 14  {\left(q t + 1\right)}^{r} + 14  {\left(-q t + 1\right)}^{r} \\
& + 9  {\left(-2  q t + 2\right)}^{r} + {\left(-4  q t + 4\right)}^{r}\big) + \frac{1}{1152}  {\left(q^{4} t^{4} + 4  q^{3} t^{3} + 6  q^{2} t^{2} + 4  q t + 1\right)}^{r} + \frac{1}{72}  {\left(q^{4} t^{4} + 2  q^{3} t^{3} + 3  q^{2} t^{2} + 2  q t + 1\right)}^{r} \\
& + \frac{1}{32}  {\left(q^{4} t^{4} + 2  q^{3} t^{3} + 2  q^{2} t^{2} + 2  q t + 1\right)}^{r} + \frac{1}{18}  {\left(q^{4} t^{4} + q^{3} t^{3} + q t + 1\right)}^{r} + \frac{1}{18}  {\left(q^{4} t^{4} - q^{3} t^{3} - q t + 1\right)}^{r} \\
& + \frac{1}{72}  {\left(q^{4} t^{4} - 2  q^{3} t^{3} + 3  q^{2} t^{2} - 2  q t + 1\right)}^{r} + \frac{1}{32}  {\left(q^{4} t^{4} - 2  q^{3} t^{3} + 2  q^{2} t^{2} - 2  q t + 1\right)}^{r} + \frac{1}{1152}  {\left(q^{4} t^{4} - 4  q^{3} t^{3} + 6  q^{2} t^{2} - 4  q t + 1\right)}^{r} \\
& + \frac{1}{96}  {\left(q^{4} t^{4} + 2  q^{2} t^{2} + 1\right)}^{r} + \frac{1}{12}  {\left(q^{4} t^{4} - q^{2} t^{2} + 1\right)}^{r} + \frac{5}{64}  {\left(q^{4} t^{4} - 2  q^{2} t^{2} + 1\right)}^{r} + \frac{1}{8}  {\left(q^{4} t^{4} + 1\right)}^{r} \\
& + \frac{1}{48}  {\left(-q^{4} t^{4} + 2  q^{3} t^{3} - 2  q t + 1\right)}^{r} + \frac{1}{6}  {\left(-q^{4} t^{4} + q^{3} t^{3} - q t + 1\right)}^{r} + \frac{1}{6}  {\left(-q^{4} t^{4} - q^{3} t^{3} + q t + 1\right)}^{r} \\
& + \frac{1}{48}  {\left(-q^{4} t^{4} - 2  q^{3} t^{3} + 2  q t + 1\right)}^{r} + \frac{1}{8}  {\left(-q^{4} t^{4} + 1\right)}^{r}.
\end{align*}}
From this, setting $t = -1$ and $q = 1$, we can compute the stringy Euler characteristic to be
$$
\chi^{\str}(T_{\F_4}^r/W) = \frac{1}{24} \cdot 16^{r - 1} + \frac{1}{2} \cdot 9^{r - 1} + 35 \cdot 8^{r - 2} + \frac{197}{3} \cdot 4^{r - 2} + 2 \cdot 3^{r - 1} + 25 \cdot 2^{r - 2} + \frac{25}{6}\; .
$$

Additionally, the stringy $E$-polynomial can be computed with the code \cite{code_repository} to give the following.
{\tiny
\begin{align*}
E^{\str}(T_{\F_4}^r/W) =& \frac{1}{36} \cdot 16^{r - 1}  q^{2  r} + \frac{1}{24} \cdot 9^{r}  q^{2  r} + \frac{35}{3} \cdot 8^{r - 2}  q^{2  r} + \frac{431}{18} \cdot 4^{r - 2}  q^{2  r} + 3^{r - 1}  q^{2  r}  + \frac{77}{3} \cdot 2^{r - 3}  q^{2  r} \\
& + \frac{1}{1152}  \big(48  {\left(-q^{3} + 3  q^{2} - 3  q + 1\right)}^{r} q^{\frac{1}{2}  r} + 432  {\left(-q^{3} + q^{2} + q - 1\right)}^{r} q^{\frac{1}{2}  r} + 288  {\left(-q^{3} + q^{2} - q + 1\right)}^{r} q^{\frac{1}{2}  r} \\
& + 432  {\left(-q^{3} - q^{2} + q + 1\right)}^{r} q^{\frac{1}{2}  r} + 288  {\left(-q^{3} - q^{2} - q - 1\right)}^{r} q^{\frac{1}{2}  r} + 48  {\left(-q^{3} - 3  q^{2} - 3  q - 1\right)}^{r} q^{\frac{1}{2}  r} \\
& + 384  {\left(-q^{3} + 1\right)}^{r} q^{\frac{1}{2}  r} + 384  {\left(-q^{3} - 1\right)}^{r} q^{\frac{1}{2}  r} + 144  {\left(2  q^{2} + 4  q + 2\right)}^{r} q^{r} + 144  {\left(2  q^{2} - 4  q + 2\right)}^{r} q^{r} + 288  {\left(2  q^{2} + 2\right)}^{r} q^{r} \\
& + 576  {\left(2  q^{2} - 2\right)}^{r} q^{r} + 624  {\left(q^{2} + 2  q + 1\right)}^{r} q^{r} + 384  {\left(q^{2} + q + 1\right)}^{r} q^{r} + 384  {\left(q^{2} - q + 1\right)}^{r} q^{r} + 624  {\left(q^{2} - 2  q + 1\right)}^{r} q^{r} \\
& + 288  {\left(q^{2} + 1\right)}^{r} q^{r} + 2304  {\left(q^{2} - 1\right)}^{r} q^{r} + {\left(q^{4} + 4  q^{3} + 6  q^{2} + 4  q + 1\right)}^{r} + 16  {\left(q^{4} + 2  q^{3} + 3  q^{2} + 2  q + 1\right)}^{r} \\
& + 192  q^{\frac{3}{2}  r} {\left(14  {\left(-q + 1\right)}^{r} + 14  {\left(-q - 1\right)}^{r} + 9  {\left(-2  q + 2\right)}^{r} + 9  {\left(-2  q - 2\right)}^{r} + {\left(-4  q + 4\right)}^{r} + {\left(-4  q - 4\right)}^{r}\right)} \\
& + 36  {\left(q^{4} + 2  q^{3} + 2  q^{2} + 2  q + 1\right)}^{r} + 24  {\left(q^{4} + 2  q^{3} - 2  q - 1\right)}^{r} + 64  {\left(q^{4} + q^{3} + q + 1\right)}^{r} + 192  {\left(q^{4} + q^{3} - q - 1\right)}^{r} \\
& + 192  {\left(q^{4} - q^{3} + q - 1\right)}^{r} + 64  {\left(q^{4} - q^{3} - q + 1\right)}^{r} + 16  {\left(q^{4} - 2  q^{3} + 3  q^{2} - 2  q + 1\right)}^{r} + 36  {\left(q^{4} - 2  q^{3} + 2  q^{2} - 2  q + 1\right)}^{r} \\
& + 24  {\left(q^{4} - 2  q^{3} + 2  q - 1\right)}^{r} + {\left(q^{4} - 4  q^{3} + 6  q^{2} - 4  q + 1\right)}^{r} + 12  {\left(q^{4} + 2  q^{2} + 1\right)}^{r} + 96  {\left(q^{4} - q^{2} + 1\right)}^{r} \\
& + 90  {\left(q^{4} - 2  q^{2} + 1\right)}^{r} + 144  {\left(q^{4} + 1\right)}^{r} + 144  {\left(q^{4} - 1\right)}^{r} + 4304  q^{2  r}\big) .
\end{align*}}

In the particular case that $r = 2$, corresponding to an elliptic curve, the mixed Hodge polynomial is
\begin{align*}
\mu^{\str}(T_{\F_4}^2/W) =& {\left(q^{8} + 2 \, q^{7} + 8 \, q^{6} + 22 \, q^{5} + 59 \, q^{4}\right)} t^{8} + {\left(q^{6} + 2 \, q^{5} + 9 \, q^{4} + 22 \, q^{3}\right)} t^{6} \\
& + {\left(q^{4} + 2 \, q^{3} + 8 \, q^{2}\right)} t^{4} + {\left(q^{2} + 2 \, q\right)} t^{2} + 1\; .
\end{align*}

Additionally, the Poincar\'e polynomial is
\begin{align*}
    P^{\str}(T_{\F_4}^2/W) = 92 \, t^{8} + 34 \, t^{6} + 11 \, t^{4} + 3 \, t^{2} + 1\; ,
\end{align*}
whereas the $E$-polynomial is
\begin{align*}
E^{\str}(T_{\F_4}^2/W) =& q^{8} + 2 \, q^{7} + 9 \, q^{6} + 24 \, q^{5} + 69 \, q^{4} + 24 \, q^{3} + 9 \, q^{2} + 2 \, q + 1\; .
\end{align*}
This polynomial satisfies the character positivity conjecture \ref{conj:positivity} and is palindromic (Proposition \ref{prop:palindromic}). 
The stringy Euler characteristic of the character variety is $\chi^{str}(T_{\F_4}^2/W)=141$.

\subsection{Group $G = \E_6$} In the case of root system $\E_6$, recall that the root lattice $Q$ and the weight lattice $P$ do not agree. Actually, their quotient is $P/Q = \ZZ_3$. Therefore, there exist two possible connected semisimple groups with root lattice $\E_6$, corresponding to character lattice $\Lambda = Q$ or $\Lambda = P$. However, these two groups are Langlands dual, so their stringy mixed Hodge polynomials are equal by Theorem \ref{thm:equality-terms}.
Applying Theorem \ref{thm:algorithm_CV}, the stringy mixed Hodge polynomial is the following.

\allowdisplaybreaks[4]
{\tiny
\begin{align*}
{\normalsize \mu^{\str}(T_{\E_6}^r/W)} =& \frac{1}{6}  {\left(4  q t + 4\right)}^{r} q^{\frac{5}{2}  r} t^{5  r} + {\left(3  q t + 3\right)}^{r} q^{\frac{5}{2}  r} t^{5  r} + \frac{1}{2}  {\left(2  q t + 2\right)}^{r} q^{\frac{5}{2}  r} t^{5  r} + {\left(-3  q t + 3\right)}^{r} q^{\frac{5}{2}  r} t^{5  r} + \frac{1}{18}  {\left(4  q^{2} t^{2} - 4  q t + 4\right)}^{r} q^{2  r} t^{4  r} \\
 & + \frac{1}{6}  {\left(3  q^{2} t^{2} + 3  q t + 3\right)}^{r} q^{2  r} t^{4  r} + \frac{1}{6}  {\left(3  q^{2} t^{2} - 3  q t + 3\right)}^{r} q^{2  r} t^{4  r} + \frac{1}{6}  {\left(2  q^{2} t^{2} - 2  q t + 2\right)}^{r} q^{2  r} t^{4  r} + \frac{3}{8}  {\left(-q^{2} t^{2} + 1\right)}^{r} q^{2  r} t^{4  r} \\
 & + \frac{1}{4}  {\left(q^{3} t^{3} + q^{2} t^{2} + q t + 1\right)}^{r} q^{\frac{3}{2}  r} t^{3  r} + \frac{1}{8}  {\left(q^{3} t^{3} - q^{2} t^{2} - q t + 1\right)}^{r} q^{\frac{3}{2}  r} t^{3  r} + \frac{1}{3}  {\left(q^{3} t^{3} + 1\right)}^{r} q^{\frac{3}{2}  r} t^{3  r} \\
 & + \frac{1}{6}  {\left(-q^{3} t^{3} + 2  q^{2} t^{2} - 2  q t + 1\right)}^{r} q^{\frac{3}{2}  r} t^{3  r} + \frac{3}{8}  {\left(q t + 1\right)}^{3  r} q^{\frac{3}{2}  r} t^{3  r} + \frac{5}{144}  {\left(q^{4} t^{4} + 4  q^{3} t^{3} + 6  q^{2} t^{2} + 4  q t + 1\right)}^{r} q^{r} t^{2  r} \\
 & + \frac{1}{8}  {\left(q^{4} t^{4} + 2  q^{3} t^{3} + 2  q^{2} t^{2} + 2  q t + 1\right)}^{r} q^{r} t^{2  r} + \frac{2}{9}  {\left(q^{4} t^{4} + q^{3} t^{3} + q t + 1\right)}^{r} q^{r} t^{2  r} + \frac{1}{18}  {\left(q^{4} t^{4} - 2  q^{3} t^{3} + 3  q^{2} t^{2} - 2  q t + 1\right)}^{r} q^{r} t^{2  r} \\
 & + \frac{1}{6}  {\left(q^{4} t^{4} + q^{2} t^{2} + 1\right)}^{r} q^{r} t^{2  r} + \frac{19}{48}  {\left(q^{4} t^{4} - 2  q^{2} t^{2} + 1\right)}^{r} q^{r} t^{2  r} + \frac{1}{48}  {\left(-q^{4} t^{4} + 2  q^{3} t^{3} - 2  q t + 1\right)}^{r} q^{r} t^{2  r} \\
 & + \frac{1}{6}  {\left(-q^{4} t^{4} + q^{3} t^{3} - q t + 1\right)}^{r} q^{r} t^{2  r} + \frac{1}{6}  {\left(-q^{4} t^{4} - q^{3} t^{3} + q t + 1\right)}^{r} q^{r} t^{2  r} + \frac{13}{48}  {\left(-q^{4} t^{4} - 2  q^{3} t^{3} + 2  q t + 1\right)}^{r} q^{r} t^{2  r} \\
 & + \frac{3}{8}  {\left(-q^{4} t^{4} + 1\right)}^{r} q^{r} t^{2  r} + \frac{1}{720}  {\left(q^{5} t^{5} + 5  q^{4} t^{4} + 10  q^{3} t^{3} + 10  q^{2} t^{2} + 5  q t + 1\right)}^{r} q^{\frac{1}{2}  r} t^{r} \\
 & + \frac{1}{18}  {\left(q^{5} t^{5} + 2  q^{4} t^{4} + q^{3} t^{3} + q^{2} t^{2} + 2  q t + 1\right)}^{r} q^{\frac{1}{2}  r} t^{r} + \frac{1}{16}  {\left(q^{5} t^{5} + q^{4} t^{4} - 2  q^{3} t^{3} - 2  q^{2} t^{2} + q t + 1\right)}^{r} q^{\frac{1}{2}  r} t^{r} \\
 & + \frac{1}{18}  {\left(q^{5} t^{5} - q^{4} t^{4} + q^{3} t^{3} + q^{2} t^{2} - q t + 1\right)}^{r} q^{\frac{1}{2}  r} t^{r} + \frac{1}{8}  {\left(q^{5} t^{5} - q^{4} t^{4} - q t + 1\right)}^{r} q^{\frac{1}{2}  r} t^{r} + \frac{1}{5}  {\left(q^{5} t^{5} + 1\right)}^{r} q^{\frac{1}{2}  r} t^{r} \\
 & + \frac{1}{48}  {\left(-q^{5} t^{5} + q^{4} t^{4} + 2  q^{3} t^{3} - 2  q^{2} t^{2} - q t + 1\right)}^{r} q^{\frac{1}{2}  r} t^{r} + \frac{1}{6}  {\left(-q^{5} t^{5} + q^{4} t^{4} - q^{3} t^{3} + q^{2} t^{2} - q t + 1\right)}^{r} q^{\frac{1}{2}  r} t^{r} \\
 & + \frac{1}{8}  {\left(-q^{5} t^{5} - q^{4} t^{4} + q t + 1\right)}^{r} q^{\frac{1}{2}  r} t^{r} + \frac{1}{48}  {\left(-q^{5} t^{5} - 3  q^{4} t^{4} - 2  q^{3} t^{3} + 2  q^{2} t^{2} + 3  q t + 1\right)}^{r} q^{\frac{1}{2}  r} t^{r} \\
 & + \frac{1}{6}  {\left(-q^{5} t^{5} + q^{3} t^{3} - q^{2} t^{2} + 1\right)}^{r} q^{\frac{1}{2}  r} t^{r} + \frac{1}{18}  {\left(3  {\left(q t + 1\right)}^{r} q^{\frac{3}{2}  r} t^{3  r} + 14  q^{2  r} t^{4  r}\right)} {\left(q^{2} t^{2} - q t + 1\right)}^{r} \\
 & + \frac{1}{324}  {\left(3^{r + 3} q^{2  r} t^{4  r} + 81  {\left(q t + 1\right)}^{r} q^{\frac{3}{2}  r} t^{3  r}\right)} {\left(q^{2} t^{2} - 2  q t + 1\right)}^{r} \\
 & + \frac{1}{10368}  {\left(13824  {\left(-q t + 1\right)}^{r} q^{\frac{3}{2}  r} t^{3  r} + {\left(18 \cdot 4^{r + 2} q^{2  r} + 32 \cdot 3^{r + 3} q^{2  r} + 27 \cdot 2^{r + 5} q^{2  r} + 19584  q^{2  r}\right)} t^{4  r}\right)} {\left(q t + 1\right)}^{2  r} \\
 & + \frac{1}{10368}  {\left({\left(54 \cdot 4^{r + 2} q^{2  r} + 64 \cdot 3^{r + 4} q^{2  r} + 81 \cdot 2^{r + 5} q^{2  r}\right)} {\left(1-q t\right)}^{r} t^{4  r} + 23760  {\left(1-q t\right)}^{r} q^{2  r} t^{4  r} + 34560  q^{\frac{5}{2}  r} t^{5  r}\right)} {\left(q t + 1\right)}^{r} \\
 & + \frac{1}{10368}  {\left(48 \cdot 27^{r} q^{3  r} + 12^{r + 3} q^{3  r} + 1536 \cdot 9^{r} q^{3  r} + 4 \cdot 6^{r + 4} q^{3  r} + 4816 \cdot 3^{r + 2} q^{3  r}\right)} t^{6  r} \\
 & + \frac{1}{51840}  {\left(q^{6} t^{6} + 6  q^{5} t^{5} + 15  q^{4} t^{4} + 20  q^{3} t^{3} + 15  q^{2} t^{2} + 6  q t + 1\right)}^{r} + \frac{1}{216}  {\left(q^{6} t^{6} + 3  q^{5} t^{5} + 3  q^{4} t^{4} + 2  q^{3} t^{3} + 3  q^{2} t^{2} + 3  q t + 1\right)}^{r} \\
 & + \frac{1}{96}  {\left(q^{6} t^{6} + 2  q^{5} t^{5} + 3  q^{4} t^{4} + 4  q^{3} t^{3} + 3  q^{2} t^{2} + 2  q t + 1\right)}^{r} + \frac{1}{192}  {\left(q^{6} t^{6} + 2  q^{5} t^{5} - q^{4} t^{4} - 4  q^{3} t^{3} - q^{2} t^{2} + 2  q t + 1\right)}^{r} \\
 & + \frac{1}{72}  {\left(q^{6} t^{6} + q^{5} t^{5} + 2  q^{4} t^{4} + q^{3} t^{3} + 2  q^{2} t^{2} + q t + 1\right)}^{r} + \frac{1}{36}  {\left(q^{6} t^{6} + q^{5} t^{5} - q^{4} t^{4} - 2  q^{3} t^{3} - q^{2} t^{2} + q t + 1\right)}^{r} \\
 & + \frac{1}{10}  {\left(q^{6} t^{6} + q^{5} t^{5} + q t + 1\right)}^{r} + \frac{1}{24}  {\left(q^{6} t^{6} - q^{5} t^{5} - q^{4} t^{4} + 2  q^{3} t^{3} - q^{2} t^{2} - q t + 1\right)}^{r} + \frac{1}{12}  {\left(q^{6} t^{6} - q^{5} t^{5} + q^{3} t^{3} - q t + 1\right)}^{r} \\
 & + \frac{1}{36}  {\left(q^{6} t^{6} - 2  q^{5} t^{5} + 2  q^{4} t^{4} - 2  q^{3} t^{3} + 2  q^{2} t^{2} - 2  q t + 1\right)}^{r} + \frac{1}{1152}  {\left(q^{6} t^{6} - 2  q^{5} t^{5} - q^{4} t^{4} + 4  q^{3} t^{3} - q^{2} t^{2} - 2  q t + 1\right)}^{r} \\
 & + \frac{1}{648}  {\left(q^{6} t^{6} - 3  q^{5} t^{5} + 6  q^{4} t^{4} - 7  q^{3} t^{3} + 6  q^{2} t^{2} - 3  q t + 1\right)}^{r} + \frac{1}{16}  {\left(q^{6} t^{6} - q^{4} t^{4} - q^{2} t^{2} + 1\right)}^{r} + \frac{1}{108}  {\left(q^{6} t^{6} + 2  q^{3} t^{3} + 1\right)}^{r} \\
 & + \frac{1}{9}  {\left(q^{6} t^{6} - q^{3} t^{3} + 1\right)}^{r} + \frac{1}{96}  {\left(-q^{6} t^{6} + 2  q^{5} t^{5} - q^{4} t^{4} + q^{2} t^{2} - 2  q t + 1\right)}^{r} + \frac{1}{36}  {\left(-q^{6} t^{6} + 2  q^{5} t^{5} - 2  q^{4} t^{4} + 2  q^{2} t^{2} - 2  q t + 1\right)}^{r} \\
 & + \frac{1}{10}  {\left(-q^{6} t^{6} + q^{5} t^{5} - q t + 1\right)}^{r} + \frac{1}{36}  {\left(-q^{6} t^{6} - q^{5} t^{5} + q^{4} t^{4} - q^{2} t^{2} + q t + 1\right)}^{r} + \frac{1}{12}  {\left(-q^{6} t^{6} - q^{5} t^{5} - q^{4} t^{4} + q^{2} t^{2} + q t + 1\right)}^{r} \\
 & + \frac{1}{32}  {\left(-q^{6} t^{6} - 2  q^{5} t^{5} - q^{4} t^{4} + q^{2} t^{2} + 2  q t + 1\right)}^{r} + \frac{1}{1440}  {\left(-q^{6} t^{6} - 4  q^{5} t^{5} - 5  q^{4} t^{4} + 5  q^{2} t^{2} + 4  q t + 1\right)}^{r} \\
 & + \frac{1}{96}  {\left(-q^{6} t^{6} + 3  q^{4} t^{4} - 3  q^{2} t^{2} + 1\right)}^{r} + \frac{1}{8}  {\left(-q^{6} t^{6} + q^{4} t^{4} - q^{2} t^{2} + 1\right)}^{r} + \frac{1}{12}  {\left(-q^{6} t^{6} + 1\right)}^{r}.
\end{align*}
}
From this, setting $t = -1$ and $q = 1$, we can compute the stringy Euler characteristic to be
$$
\chi^{\str}(T_{\E_6}^r/W) = \frac{1}{6} \cdot 27^{r - 1} + \frac{1}{3} \cdot 12^{r} + \frac{10}{3} \cdot 9^{r - 1} + \frac{1}{3} \cdot 6^{r + 1} + \frac{11}{2} \cdot 3^{r}\; .
$$

Additionally, the stringy $E$-polynomial can be computed from the code in \cite{code_repository} to give the following.
\allowdisplaybreaks[4]
{\tiny
\begin{align*}
E^{\str}(T_{\E_6}^r/W) =& \frac{1}{18}   {\left(4  q^{2} + 4  q + 4\right)}^{r} q^{2  r} + \frac{1}{36}   {\left(4  q^{2} - 8  q + 4\right)}^{r} q^{2  r} + \frac{1}{12}   {\left(4  q^{2} - 4\right)}^{r} q^{2  r} + \frac{1}{12}   {\left(3  q^{2} + 6  q + 3\right)}^{r} q^{2  r} + \frac{1}{6}   {\left(3  q^{2} + 3  q + 3\right)}^{r} q^{2  r} \\
& + \frac{1}{6}   {\left(3  q^{2} - 3  q + 3\right)}^{r} q^{2  r}  + \frac{1}{12}   {\left(3  q^{2} - 6  q + 3\right)}^{r} q^{2  r} + \frac{1}{2}   {\left(3  q^{2} - 3\right)}^{r} q^{2  r} + \frac{1}{6}   {\left(2  q^{2} + 2  q + 2\right)}^{r} q^{2  r}  + \frac{1}{12}   {\left(2  q^{2} - 4  q + 2\right)}^{r} q^{2  r} \\
& + \frac{1}{4}   {\left(2  q^{2} - 2\right)}^{r} q^{2  r} + \frac{7}{9}   {\left(q^{2} + q + 1\right)}^{r} q^{2  r}  + \frac{17}{9}   {\left(q^{2} - 2  q + 1\right)}^{r} q^{2  r} + \frac{8}{3}   {\left(q^{2} - 1\right)}^{r} q^{2  r} + \frac{3}{8}   {\left(-q^{3} + 3  q^{2} - 3  q + 1\right)}^{r} q^{\frac{3}{2}  r}  \\
& + \frac{4}{3}   {\left(-q^{3} + q^{2} + q - 1\right)}^{r} q^{\frac{3}{2}  r} + \frac{1}{4}   {\left(-q^{3} + q^{2} - q + 1\right)}^{r} q^{\frac{3}{2}  r} + \frac{3}{8}   {\left(-q^{3} - q^{2} + q + 1\right)}^{r} q^{\frac{3}{2}  r}  + \frac{1}{6}   {\left(-q^{3} - 2  q^{2} - 2  q - 1\right)}^{r} q^{\frac{3}{2}  r} \\
& + \frac{1}{2}   {\left(-q^{3} + 1\right)}^{r} q^{\frac{3}{2}  r} + \frac{1}{720}   {\left(-q^{5} + 5  q^{4} - 10  q^{3} + 10  q^{2} - 5  q + 1\right)}^{r} q^{\frac{1}{2}  r} + \frac{1}{48}   {\left(-q^{5} + 3  q^{4} - 2  q^{3} - 2  q^{2} + 3  q - 1\right)}^{r} q^{\frac{1}{2}  r} \\
& + \frac{1}{18}   {\left(-q^{5} + 2  q^{4} - q^{3} + q^{2} - 2  q + 1\right)}^{r} q^{\frac{1}{2}  r} + \frac{1}{16}   {\left(-q^{5} + q^{4} + 2  q^{3} - 2  q^{2} - q + 1\right)}^{r} q^{\frac{1}{2}  r} + \frac{1}{8}   {\left(-q^{5} + q^{4} + q - 1\right)}^{r} q^{\frac{1}{2}  r} \\
& + \frac{1}{48}   {\left(-q^{5} - q^{4} + 2  q^{3} + 2  q^{2} - q - 1\right)}^{r} q^{\frac{1}{2}  r} + \frac{1}{18}   {\left(-q^{5} - q^{4} - q^{3} + q^{2} + q + 1\right)}^{r} q^{\frac{1}{2}  r} + \frac{1}{6}   {\left(-q^{5} - q^{4} - q^{3} - q^{2} - q - 1\right)}^{r} q^{\frac{1}{2}  r}  \\
& + \frac{1}{8}   {\left(-q^{5} - q^{4} + q + 1\right)}^{r} q^{\frac{1}{2}  r} + \frac{1}{6}   {\left(-q^{5} + q^{3} + q^{2} - 1\right)}^{r} q^{\frac{1}{2}  r} + \frac{1}{5}   {\left(-q^{5} + 1\right)}^{r} q^{\frac{1}{2}  r} + \frac{1}{18}   {\left(q^{4} + 2  q^{3} + 3  q^{2} + 2  q + 1\right)}^{r} q^{r} \\
& + \frac{1}{48}   {\left(q^{4} + 2  q^{3} - 2  q - 1\right)}^{r} q^{r} + \frac{1}{6}   {\left(q^{4} + q^{3} - q - 1\right)}^{r} q^{r} + \frac{1}{6}   {\left(q^{4} - q^{3} + q - 1\right)}^{r} q^{r} + \frac{2}{9}   {\left(q^{4} - q^{3} - q + 1\right)}^{r} q^{r} \\
& + \frac{1}{8}   {\left(q^{4} - 2  q^{3} + 2  q^{2} - 2  q + 1\right)}^{r} q^{r} + \frac{13}{48}   {\left(q^{4} - 2  q^{3} + 2  q - 1\right)}^{r} q^{r} + \frac{5}{144}   {\left(q^{4} - 4  q^{3} + 6  q^{2} - 4  q + 1\right)}^{r} q^{r} + \frac{1}{6}   {\left(q^{4} + q^{2} + 1\right)}^{r} q^{r} \\
& + \frac{19}{48}   {\left(q^{4} - 2  q^{2} + 1\right)}^{r} q^{r} + \frac{3}{8}   {\left(q^{4} - 1\right)}^{r} q^{r} + \frac{10}{3}   q^{\frac{5}{2}  r} {\left(-q + 1\right)}^{r} + \frac{1}{2}   q^{\frac{5}{2}  r} {\left(-2  q + 2\right)}^{r} +  q^{\frac{5}{2}  r} {\left(-3  q + 3\right)}^{r} +  q^{\frac{5}{2}  r} {\left(-3  q - 3\right)}^{r} \\
& + \frac{1}{6}   q^{\frac{5}{2}  r} {\left(-4  q + 4\right)}^{r} + \frac{1}{10368}  {\left(48 \cdot 27^{r} q^{3  r} + 12^{r + 3} q^{3  r} + 1536 \cdot 9^{r} q^{3  r} + 4 \cdot 6^{r + 4} q^{3  r} + 4816 \cdot 3^{r + 2} q^{3  r}\right)}  \\
& + \frac{1}{648}   {\left(q^{6} + 3  q^{5} + 6  q^{4} + 7  q^{3} + 6  q^{2} + 3  q + 1\right)}^{r}  + \frac{1}{36}   {\left(q^{6} + 2  q^{5} + 2  q^{4} + 2  q^{3} + 2  q^{2} + 2  q + 1\right)}^{r} \\
& + \frac{1}{36}   {\left(q^{6} + 2  q^{5} + 2  q^{4} - 2  q^{2} - 2  q - 1\right)}^{r} + \frac{1}{96}   {\left(q^{6} + 2  q^{5} + q^{4} - q^{2} - 2  q - 1\right)}^{r} + \frac{1}{1152}   {\left(q^{6} + 2  q^{5} - q^{4} - 4  q^{3} - q^{2} + 2  q + 1\right)}^{r} \\
& + \frac{1}{24}   {\left(q^{6} + q^{5} - q^{4} - 2  q^{3} - q^{2} + q + 1\right)}^{r} + \frac{1}{12}   {\left(q^{6} + q^{5} - q^{3} + q + 1\right)}^{r} + \frac{1}{10}   {\left(q^{6} + q^{5} - q - 1\right)}^{r} \\
& + \frac{1}{72}   {\left(q^{6} - q^{5} + 2  q^{4} - q^{3} + 2  q^{2} - q + 1\right)}^{r} + \frac{1}{12}   {\left(q^{6} - q^{5} + q^{4} - q^{2} + q - 1\right)}^{r} + \frac{1}{36}   {\left(q^{6} - q^{5} - q^{4} + 2  q^{3} - q^{2} - q + 1\right)}^{r} \\
& + \frac{1}{36}   {\left(q^{6} - q^{5} - q^{4} + q^{2} + q - 1\right)}^{r} + \frac{1}{10}   {\left(q^{6} - q^{5} - q + 1\right)}^{r} + \frac{1}{96}   {\left(q^{6} - 2  q^{5} + 3  q^{4} - 4  q^{3} + 3  q^{2} - 2  q + 1\right)}^{r} \\
& + \frac{1}{32}   {\left(q^{6} - 2  q^{5} + q^{4} - q^{2} + 2  q - 1\right)}^{r} + \frac{1}{192}   {\left(q^{6} - 2  q^{5} - q^{4} + 4  q^{3} - q^{2} - 2  q + 1\right)}^{r} \\
& + \frac{1}{216}   {\left(q^{6} - 3  q^{5} + 3  q^{4} - 2  q^{3} + 3  q^{2} - 3  q + 1\right)}^{r} + \frac{1}{1440}   {\left(q^{6} - 4  q^{5} + 5  q^{4} - 5  q^{2} + 4  q - 1\right)}^{r} \\
& + \frac{1}{51840}   {\left(q^{6} - 6  q^{5} + 15  q^{4} - 20  q^{3} + 15  q^{2} - 6  q + 1\right)}^{r} + \frac{1}{8}   {\left(q^{6} - q^{4} + q^{2} - 1\right)}^{r} + \frac{1}{16}   {\left(q^{6} - q^{4} - q^{2} + 1\right)}^{r} \\
& + \frac{1}{96}   {\left(q^{6} - 3  q^{4} + 3  q^{2} - 1\right)}^{r} + \frac{1}{9}   {\left(q^{6} + q^{3} + 1\right)}^{r} + \frac{1}{108}   {\left(q^{6} - 2  q^{3} + 1\right)}^{r} + \frac{1}{12}   {\left(q^{6} - 1\right)}^{r}.
\end{align*}
}

In the particular case that $r = 2$, corresponding to an elliptic curve, the mixed Hodge polynomial is
\begin{align*}
\mu^{\str}(T_{\E_6}^2/W) =& {\left(q^{12} + q^{11} + 2 \, q^{10} + 3 \, q^{9} + 19 \, q^{8} + 26 \, q^{7} + 95 \, q^{6}\right)} t^{12} \\
& + 2 \, {\left(q^{9} + 2 \, q^{8} + 3 \, q^{7} + 8 \, q^{6}\right)} t^{11} + {\left(q^{10} + q^{9} + 3 \, q^{8} + 6 \, q^{7} + 22 \, q^{6} + 26 \, q^{5}\right)} t^{10} \\
& + 2 \, {\left(q^{7} + 2 \, q^{6} + 3 \, q^{5}\right)} t^{9} + {\left(q^{8} + q^{7} + 4 \, q^{6} + 6 \, q^{5} + 19 \, q^{4}\right)} t^{8} + 2 \, q^{3} t^{5} \\
& + 2 \, {\left(q^{5} + 2 \, q^{4}\right)} t^{7} + {\left(q^{6} + q^{5} + 3 \, q^{4} + 3 \, q^{3}\right)} t^{6} \\
& + {\left(q^{4} + q^{3} + 2 \, q^{2}\right)} t^{4} + {\left(q^{2} + q\right)} t^{2} + 1\; .
\end{align*}

Additionally, the Poincar\'e polynomial is
\begin{align*}
    P^{\str}(T_{\E_6}^2/W) = 147 \, t^{12} + 28 \, t^{11} + 59 \, t^{10} + 12 \, t^{9} + 31 \, t^{8} + 6 \, t^{7} + 8 \, t^{6} + 2 \, t^{5} + 4 \, t^{4} + 2 \, t^{2} + 1\; ,
\end{align*}
whereas the $E$-polynomial is
\begin{align*}
E^{\str}(T_{\E_6}^2/W) =& q^{12} + q^{11} + 3 \, q^{10} + 2 \, q^{9} + 19 \, q^{8} + 25 \, q^{7} + 102 \, q^{6} \\
& + 25 \, q^{5} + 19 \, q^{4} + 2 \, q^{3} + 3 \, q^{2} + q + 1\; .
\end{align*}
This polynomial satisfies the character positivity conjecture \ref{conj:positivity} and is palindromic (Proposition \ref{prop:palindromic}). The stringy Euler characteristic of the character variety is $\chi^{\str}(T_{\E_6}^2/W)=204$.

\subsection{Group $G = \E_7$}
In the case of root system $\E_7$, recall that the root lattice $Q$ and the weight lattice $P$ do not agree. Actually, their quotient is $P/Q = \ZZ_2$. Therefore, there exist two possible connected semisimple groups with root lattice $\E_7$, corresponding to character lattice $\Lambda = Q$ or $\Lambda = P$. However, these two groups are Langlands dual, so their stringy mixed Hodge polynomials are equal by Theorem \ref{thm:equality-terms}.

{\tiny
}

From this, setting $t = -1$ and $q = 1$, we can compute the stringy Euler characteristic to be
\begin{align*}
\chi^{\str}(T_{\E_7}^r/W) &= \frac{1}{7560} \cdot 128^{r - 1} + \frac{1}{120} \cdot 64^{r - 1} + \frac{529}{720} \cdot 32^{r - 1} + \frac{1}{12} \cdot 18^{r} \\
& + \frac{427}{3} \cdot 16^{r - 2} + \frac{11929}{45} \cdot 8^{r - 2} + 6^{r} + \frac{1381}{30} \cdot 4^{r - 1} + \frac{44777}{945} \cdot 2^{r - 2}\; .
\end{align*}

Additionally, the stringy $E$-polynomial can be computed from the code in \cite{code_repository} to give the following.
{\tiny
\begin{align*}
E^{\str}(T_{\E_7}^r/W) = & \frac{1}{11340} \cdot 128^{r - 1} \left(-1\right)^{ r} q^{\frac{7}{2} r} + \frac{1}{360} \cdot 64^{r - 1} \left(-1\right)^{ r} q^{\frac{7}{2} r} + \frac{137}{540} \cdot 32^{r - 1} \left(-1\right)^{ r} q^{\frac{7}{2} r} + \frac{1}{24} \cdot 18^{r} \left(-1\right)^{ r} q^{\frac{7}{2} r} \\
& + \frac{101}{72} \cdot 16^{r - 1} \left(-1\right)^{ r} q^{\frac{7}{2} r} + \frac{4759}{540} \cdot 8^{r - 1} \left(-1\right)^{ r} q^{\frac{7}{2} r} + \frac{1}{3} \cdot 6^{r} \left(-1\right)^{ r} q^{\frac{7}{2} r} + \frac{533}{45} \cdot 4^{r - 1} \left(-1\right)^{ r} q^{\frac{7}{2} r} \\
& + \frac{169343}{2835} \cdot 2^{r - 3} \left(-1\right)^{ r} q^{\frac{7}{2} r} + \frac{1}{2903040} \Biggl(15120 {\left(8 q^{2} + 16 q + 8\right)}^{r} q^{\frac{5}{2} r} + 15120 {\left(8 q^{2} - 16 q + 8\right)}^{r} q^{\frac{5}{2} r} \\
& + 30240 {\left(8 q^{2} + 8\right)}^{r} q^{\frac{5}{2} r} + 60480 {\left(8 q^{2} - 8\right)}^{r} q^{\frac{5}{2} r} + 257040 {\left(4 q^{2} + 8 q + 4\right)}^{r} q^{\frac{5}{2} r} + 257040 {\left(4 q^{2} - 8 q + 4\right)}^{r} q^{\frac{5}{2} r} \\
& + 272160 {\left(4 q^{2} + 4\right)}^{r} q^{\frac{5}{2} r} + 786240 {\left(4 q^{2} - 4\right)}^{r} q^{\frac{5}{2} r} + 2509920 {\left(2 q^{2} + 4 q + 2\right)}^{r} q^{\frac{5}{2} r} + 967680 {\left(2 q^{2} + 2 q + 2\right)}^{r} q^{\frac{5}{2} r} \\
& + 967680 {\left(2 q^{2} - 2 q + 2\right)}^{r} q^{\frac{5}{2} r} + 2509920 {\left(2 q^{2} - 4 q + 2\right)}^{r} q^{\frac{5}{2} r} + 1874880 {\left(2 q^{2} + 2\right)}^{r} q^{\frac{5}{2} r} + 8830080 {\left(2 q^{2} - 2\right)}^{r} q^{\frac{5}{2} r} \\
& + 5322240 {\left(q^{2} + 2 q + 1\right)}^{r} q^{\frac{5}{2} r} + 5322240 {\left(q^{2} - 2 q + 1\right)}^{r} q^{\frac{5}{2} r} + 7741440 {\left(q^{2} - 1\right)}^{r} q^{\frac{5}{2} r} \\
& + 1229760 {\left(-q^{3} + 3 q^{2} - 3 q + 1\right)}^{r} q^{2 r} + 241920 {\left(-q^{3} + 2 q^{2} - 2 q + 1\right)}^{r} q^{2 r} + 5019840 {\left(-q^{3} + q^{2} + q - 1\right)}^{r} q^{2 r} \\
& + 846720 {\left(-q^{3} + q^{2} - q + 1\right)}^{r} q^{2 r} + 5019840 {\left(-q^{3} - q^{2} + q + 1\right)}^{r} q^{2 r} + 846720 {\left(-q^{3} - q^{2} - q - 1\right)}^{r} q^{2 r} \\
& + 241920 {\left(-q^{3} - 2 q^{2} - 2 q - 1\right)}^{r} q^{2 r} + 1229760 {\left(-q^{3} - 3 q^{2} - 3 q - 1\right)}^{r} q^{2 r} + 1854720 {\left(-q^{3} + 1\right)}^{r} q^{2 r} \\
& + 1854720 {\left(-q^{3} - 1\right)}^{r} q^{2 r} + 151200 {\left(-2 q^{3} + 6 q^{2} - 6 q + 2\right)}^{r} q^{2 r} + 1360800 {\left(-2 q^{3} + 2 q^{2} + 2 q - 2\right)}^{r} q^{2 r} \\
& + 907200 {\left(-2 q^{3} + 2 q^{2} - 2 q + 2\right)}^{r} q^{2 r} + 1360800 {\left(-2 q^{3} - 2 q^{2} + 2 q + 2\right)}^{r} q^{2 r} + 907200 {\left(-2 q^{3} - 2 q^{2} - 2 q - 2\right)}^{r} q^{2 r} \\
& + 151200 {\left(-2 q^{3} - 6 q^{2} - 6 q - 2\right)}^{r} q^{2 r} + 1209600 {\left(-2 q^{3} + 2\right)}^{r} q^{2 r} + 1209600 {\left(-2 q^{3} - 2\right)}^{r} q^{2 r} \\
& + 10080 {\left(-4 q^{3} + 12 q^{2} - 12 q + 4\right)}^{r} q^{2 r} + 90720 {\left(-4 q^{3} + 4 q^{2} + 4 q - 4\right)}^{r} q^{2 r} + 60480 {\left(-4 q^{3} + 4 q^{2} - 4 q + 4\right)}^{r} q^{2 r} \\
& + 90720 {\left(-4 q^{3} - 4 q^{2} + 4 q + 4\right)}^{r} q^{2 r} + 60480 {\left(-4 q^{3} - 4 q^{2} - 4 q - 4\right)}^{r} q^{2 r} + 10080 {\left(-4 q^{3} - 12 q^{2} - 12 q - 4\right)}^{r} q^{2 r} \\
& + 80640 {\left(-4 q^{3} + 4\right)}^{r} q^{2 r} + 80640 {\left(-4 q^{3} - 4\right)}^{r} q^{2 r} + 2520 {\left(2 q^{4} + 8 q^{3} + 12 q^{2} + 8 q + 2\right)}^{r} q^{\frac{3}{2} r} \\
& + 40320 {\left(2 q^{4} + 4 q^{3} + 6 q^{2} + 4 q + 2\right)}^{r} q^{\frac{3}{2} r} + 90720 {\left(2 q^{4} + 4 q^{3} + 4 q^{2} + 4 q + 2\right)}^{r} q^{\frac{3}{2} r} + 60480 {\left(2 q^{4} + 4 q^{3} - 4 q - 2\right)}^{r} q^{\frac{3}{2} r} \\
& + 161280 {\left(2 q^{4} + 2 q^{3} + 2 q + 2\right)}^{r} q^{\frac{3}{2} r} + 483840 {\left(2 q^{4} + 2 q^{3} - 2 q - 2\right)}^{r} q^{\frac{3}{2} r} + 483840 {\left(2 q^{4} - 2 q^{3} + 2 q - 2\right)}^{r} q^{\frac{3}{2} r} \\
& + 161280 {\left(2 q^{4} - 2 q^{3} - 2 q + 2\right)}^{r} q^{\frac{3}{2} r} + 40320 {\left(2 q^{4} - 4 q^{3} + 6 q^{2} - 4 q + 2\right)}^{r} q^{\frac{3}{2} r} + 90720 {\left(2 q^{4} - 4 q^{3} + 4 q^{2} - 4 q + 2\right)}^{r} q^{\frac{3}{2} r} \\
& + 60480 {\left(2 q^{4} - 4 q^{3} + 4 q - 2\right)}^{r} q^{\frac{3}{2} r} + 2520 {\left(2 q^{4} - 8 q^{3} + 12 q^{2} - 8 q + 2\right)}^{r} q^{\frac{3}{2} r} + 30240 {\left(2 q^{4} + 4 q^{2} + 2\right)}^{r} q^{\frac{3}{2} r} \\
& + 241920 {\left(2 q^{4} - 2 q^{2} + 2\right)}^{r} q^{\frac{3}{2} r} + 226800 {\left(2 q^{4} - 4 q^{2} + 2\right)}^{r} q^{\frac{3}{2} r} + 362880 {\left(2 q^{4} + 2\right)}^{r} q^{\frac{3}{2} r} + 362880 {\left(2 q^{4} - 2\right)}^{r} q^{\frac{3}{2} r} \\
& + 120960 {\left(q^{4} + 4 q^{3} + 6 q^{2} + 4 q + 1\right)}^{r} q^{\frac{3}{2} r} + 362880 {\left(q^{4} + 2 q^{3} + 2 q^{2} + 2 q + 1\right)}^{r} q^{\frac{3}{2} r} + 967680 {\left(q^{4} + 2 q^{3} - 2 q - 1\right)}^{r} q^{\frac{3}{2} r} \\
& + 967680 {\left(q^{4} + q^{3} + q + 1\right)}^{r} q^{\frac{3}{2} r} + 483840 {\left(q^{4} + q^{3} - q - 1\right)}^{r} q^{\frac{3}{2} r} + 483840 {\left(q^{4} - q^{3} + q - 1\right)}^{r} q^{\frac{3}{2} r} \\
& + 967680 {\left(q^{4} - q^{3} - q + 1\right)}^{r} q^{\frac{3}{2} r} + 362880 {\left(q^{4} - 2 q^{3} + 2 q^{2} - 2 q + 1\right)}^{r} q^{\frac{3}{2} r} + 967680 {\left(q^{4} - 2 q^{3} + 2 q - 1\right)}^{r} q^{\frac{3}{2} r} \\
& + 120960 {\left(q^{4} - 4 q^{3} + 6 q^{2} - 4 q + 1\right)}^{r} q^{\frac{3}{2} r} + 1451520 {\left(q^{4} - 2 q^{2} + 1\right)}^{r} q^{\frac{3}{2} r} + 1451520 {\left(q^{4} - 1\right)}^{r} q^{\frac{3}{2} r} \\
& + 126 {\left(q^{6} + 6 q^{5} + 15 q^{4} + 20 q^{3} + 15 q^{2} + 6 q + 1\right)}^{r} q^{\frac{1}{2} r} + 3780 {\left(q^{6} + 4 q^{5} + 5 q^{4} - 5 q^{2} - 4 q - 1\right)}^{r} q^{\frac{1}{2} r} \\
& + 20160 {\left(q^{6} + 3 q^{5} + 3 q^{4} + 2 q^{3} + 3 q^{2} + 3 q + 1\right)}^{r} q^{\frac{1}{2} r} + 22680 {\left(q^{6} + 2 q^{5} + 3 q^{4} + 4 q^{3} + 3 q^{2} + 2 q + 1\right)}^{r} q^{\frac{1}{2} r} \\
& + 105840 {\left(q^{6} + 2 q^{5} + q^{4} - q^{2} - 2 q - 1\right)}^{r} q^{\frac{1}{2} r} + 24570 {\left(q^{6} + 2 q^{5} - q^{4} - 4 q^{3} - q^{2} + 2 q + 1\right)}^{r} q^{\frac{1}{2} r} \\
& + 120960 {\left(q^{6} + q^{5} + q^{4} - q^{2} - q - 1\right)}^{r} q^{\frac{1}{2} r} + 60480 {\left(q^{6} + q^{5} - q^{4} - 2 q^{3} - q^{2} + q + 1\right)}^{r} q^{\frac{1}{2} r} \\
& + 120960 {\left(q^{6} + q^{5} - q^{4} + q^{2} - q - 1\right)}^{r} q^{\frac{1}{2} r} + 290304 {\left(q^{6} + q^{5} + q + 1\right)}^{r} q^{\frac{1}{2} r} + 120960 {\left(q^{6} - q^{5} + q^{4} - q^{2} + q - 1\right)}^{r} q^{\frac{1}{2} r} \\
& + 60480 {\left(q^{6} - q^{5} - q^{4} + 2 q^{3} - q^{2} - q + 1\right)}^{r} q^{\frac{1}{2} r} + 120960 {\left(q^{6} - q^{5} - q^{4} + q^{2} + q - 1\right)}^{r} q^{\frac{1}{2} r} \\
& + 290304 {\left(q^{6} - q^{5} - q + 1\right)}^{r} q^{\frac{1}{2} r} + 22680 {\left(q^{6} - 2 q^{5} + 3 q^{4} - 4 q^{3} + 3 q^{2} - 2 q + 1\right)}^{r} q^{\frac{1}{2} r} \\
& + 105840 {\left(q^{6} - 2 q^{5} + q^{4} - q^{2} + 2 q - 1\right)}^{r} q^{\frac{1}{2} r} + 24570 {\left(q^{6} - 2 q^{5} - q^{4} + 4 q^{3} - q^{2} - 2 q + 1\right)}^{r} q^{\frac{1}{2} r} \\
& + 20160 {\left(q^{6} - 3 q^{5} + 3 q^{4} - 2 q^{3} + 3 q^{2} - 3 q + 1\right)}^{r} q^{\frac{1}{2} r} + 3780 {\left(q^{6} - 4 q^{5} + 5 q^{4} - 5 q^{2} + 4 q - 1\right)}^{r} q^{\frac{1}{2} r} \\
& + 126 {\left(q^{6} - 6 q^{5} + 15 q^{4} - 20 q^{3} + 15 q^{2} - 6 q + 1\right)}^{r} q^{\frac{1}{2} r} + 181440 {\left(q^{6} + q^{4} + q^{2} + 1\right)}^{r} q^{\frac{1}{2} r} \\
& + 45360 {\left(q^{6} + q^{4} - q^{2} - 1\right)}^{r} q^{\frac{1}{2} r} + 181440 {\left(q^{6} - q^{4} + q^{2} - 1\right)}^{r} q^{\frac{1}{2} r} + 272160 {\left(q^{6} - q^{4} - q^{2} + 1\right)}^{r} q^{\frac{1}{2} r} \\
& + 37800 {\left(q^{6} - 3 q^{4} + 3 q^{2} - 1\right)}^{r} q^{\frac{1}{2} r} + 80640 {\left(q^{6} + 2 q^{3} + 1\right)}^{r} q^{\frac{1}{2} r} + 80640 {\left(q^{6} - 2 q^{3} + 1\right)}^{r} q^{\frac{1}{2} r} + 483840 {\left(q^{6} - 1\right)}^{r} q^{\frac{1}{2} r} \\
& + 5796 {\left(-q^{5} + 5 q^{4} - 10 q^{3} + 10 q^{2} - 5 q + 1\right)}^{r} q^{r} + 94500 {\left(-q^{5} + 3 q^{4} - 2 q^{3} - 2 q^{2} + 3 q - 1\right)}^{r} q^{r} \\
& + 45360 {\left(-q^{5} + 3 q^{4} - 4 q^{3} + 4 q^{2} - 3 q + 1\right)}^{r} q^{r} + 201600 {\left(-q^{5} + 2 q^{4} - q^{3} + q^{2} - 2 q + 1\right)}^{r} q^{r} \\
& + 120960 {\left(-q^{5} + 2 q^{4} - q^{3} - q^{2} + 2 q - 1\right)}^{r} q^{r} + 340200 {\left(-q^{5} + q^{4} + 2 q^{3} - 2 q^{2} - q + 1\right)}^{r} q^{r} \\
& + 241920 {\left(-q^{5} + q^{4} - q^{3} + q^{2} - q + 1\right)}^{r} q^{r} + 80640 {\left(-q^{5} + q^{4} - q^{3} - q^{2} + q - 1\right)}^{r} q^{r} \\
& + 45360 {\left(-q^{5} + q^{4} - 2 q^{3} + 2 q^{2} - q + 1\right)}^{r} q^{r} + 771120 {\left(-q^{5} + q^{4} + q - 1\right)}^{r} q^{r} + 181440 {\left(-q^{5} + q^{4} - q + 1\right)}^{r} q^{r} \\
& + 340200 {\left(-q^{5} - q^{4} + 2 q^{3} + 2 q^{2} - q - 1\right)}^{r} q^{r} + 80640 {\left(-q^{5} - q^{4} - q^{3} + q^{2} + q + 1\right)}^{r} q^{r} \\
& + 241920 {\left(-q^{5} - q^{4} - q^{3} - q^{2} - q - 1\right)}^{r} q^{r} + 45360 {\left(-q^{5} - q^{4} - 2 q^{3} - 2 q^{2} - q - 1\right)}^{r} q^{r} \\
& + 771120 {\left(-q^{5} - q^{4} + q + 1\right)}^{r} q^{r} + 181440 {\left(-q^{5} - q^{4} - q - 1\right)}^{r} q^{r} + 120960 {\left(-q^{5} - 2 q^{4} - q^{3} + q^{2} + 2 q + 1\right)}^{r} q^{r} \\
& + 201600 {\left(-q^{5} - 2 q^{4} - q^{3} - q^{2} - 2 q - 1\right)}^{r} q^{r} + 94500 {\left(-q^{5} - 3 q^{4} - 2 q^{3} + 2 q^{2} + 3 q + 1\right)}^{r} q^{r} \\
& + 45360 {\left(-q^{5} - 3 q^{4} - 4 q^{3} - 4 q^{2} - 3 q - 1\right)}^{r} q^{r} + 5796 {\left(-q^{5} - 5 q^{4} - 10 q^{3} - 10 q^{2} - 5 q - 1\right)}^{r} q^{r} \\
& + 483840 {\left(-q^{5} + q^{3} + q^{2} - 1\right)}^{r} q^{r} + 483840 {\left(-q^{5} + q^{3} - q^{2} + 1\right)}^{r} q^{r} + 290304 {\left(-q^{5} + 1\right)}^{r} q^{r} + 290304 {\left(-q^{5} - 1\right)}^{r} q^{r} \\
& + 126 q^{3 r} \Biggl(68640 {\left(-q + 1\right)}^{r} + 68640 {\left(-q - 1\right)}^{r} + 107104 {\left(-2 q + 2\right)}^{r} + 107104 {\left(-2 q - 2\right)}^{r} + 3840 {\left(-3 q + 3\right)}^{r} \\
& + 3840 {\left(-3 q - 3\right)}^{r} + 22800 {\left(-4 q + 4\right)}^{r} + 22800 {\left(-4 q - 4\right)}^{r} + 4450 {\left(-8 q + 8\right)}^{r} + 4450 {\left(-8 q - 8\right)}^{r} + 480 {\left(-9 q + 9\right)}^{r} \\
& + 480 {\left(-9 q - 9\right)}^{r} + 45 {\left(-16 q + 16\right)}^{r} + 45 {\left(-16 q - 16\right)}^{r} + {\left(-32 q + 32\right)}^{r} + {\left(-32 q - 32\right)}^{r}\Biggr) \\
& + {\left(-q^{7} + 7 q^{6} - 21 q^{5} + 35 q^{4} - 35 q^{3} + 21 q^{2} - 7 q + 1\right)}^{r} + 63 {\left(-q^{7} + 5 q^{6} - 9 q^{5} + 5 q^{4} + 5 q^{3} - 9 q^{2} + 5 q - 1\right)}^{r} \\
& + 672 {\left(-q^{7} + 4 q^{6} - 6 q^{5} + 5 q^{4} - 5 q^{3} + 6 q^{2} - 4 q + 1\right)}^{r} + 945 {\left(-q^{7} + 3 q^{6} - q^{5} - 5 q^{4} + 5 q^{3} + q^{2} - 3 q + 1\right)}^{r} \\
& + 7560 {\left(-q^{7} + 3 q^{6} - 3 q^{5} + q^{4} + q^{3} - 3 q^{2} + 3 q - 1\right)}^{r} + 3780 {\left(-q^{7} + 3 q^{6} - 5 q^{5} + 7 q^{4} - 7 q^{3} + 5 q^{2} - 3 q + 1\right)}^{r} \\
& + 48384 {\left(-q^{7} + 2 q^{6} - q^{5} + q^{2} - 2 q + 1\right)}^{r} + 60480 {\left(-q^{7} + 2 q^{6} - 2 q^{5} + q^{4} + q^{3} - 2 q^{2} + 2 q - 1\right)}^{r} \\
& + 20160 {\left(-q^{7} + 2 q^{6} - 3 q^{5} + 3 q^{4} - 3 q^{3} + 3 q^{2} - 2 q + 1\right)}^{r} + 2240 {\left(-q^{7} + 2 q^{6} - 3 q^{5} + q^{4} + q^{3} - 3 q^{2} + 2 q - 1\right)}^{r} \\
& + 10080 {\left(-q^{7} + 2 q^{6} - q^{4} - q^{3} + 2 q - 1\right)}^{r} + 10080 {\left(-q^{7} + 2 q^{6} - 3 q^{4} + 3 q^{3} - 2 q + 1\right)}^{r} \\
& + 4095 {\left(-q^{7} + q^{6} + 3 q^{5} - 3 q^{4} - 3 q^{3} + 3 q^{2} + q - 1\right)}^{r} + 52920 {\left(-q^{7} + q^{6} + q^{5} - q^{4} + q^{3} - q^{2} - q + 1\right)}^{r} \\
& + 90720 {\left(-q^{7} + q^{6} + q^{5} - q^{4} - q^{3} + q^{2} + q - 1\right)}^{r} + 11340 {\left(-q^{7} + q^{6} - q^{5} + q^{4} + q^{3} - q^{2} + q - 1\right)}^{r} \\
& + 90720 {\left(-q^{7} + q^{6} - q^{5} + q^{4} - q^{3} + q^{2} - q + 1\right)}^{r} + 96768 {\left(-q^{7} + q^{6} - q^{5} - q^{2} + q - 1\right)}^{r} \\
& + 13440 {\left(-q^{7} + q^{6} + 2 q^{4} - 2 q^{3} - q + 1\right)}^{r} + 161280 {\left(-q^{7} + q^{6} - q^{4} + q^{3} - q + 1\right)}^{r} \\
& + 40320 {\left(-q^{7} + q^{6} - 2 q^{4} + 2 q^{3} - q + 1\right)}^{r} + 161280 {\left(-q^{7} + q^{6} + q - 1\right)}^{r} \\
& + 4095 {\left(-q^{7} - q^{6} + 3 q^{5} + 3 q^{4} - 3 q^{3} - 3 q^{2} + q + 1\right)}^{r} + 52920 {\left(-q^{7} - q^{6} + q^{5} + q^{4} + q^{3} + q^{2} - q - 1\right)}^{r} \\
& + 90720 {\left(-q^{7} - q^{6} + q^{5} + q^{4} - q^{3} - q^{2} + q + 1\right)}^{r} + 11340 {\left(-q^{7} - q^{6} - q^{5} - q^{4} + q^{3} + q^{2} + q + 1\right)}^{r} \\
& + 90720 {\left(-q^{7} - q^{6} - q^{5} - q^{4} - q^{3} - q^{2} - q - 1\right)}^{r} + 96768 {\left(-q^{7} - q^{6} - q^{5} + q^{2} + q + 1\right)}^{r} \\
& + 40320 {\left(-q^{7} - q^{6} + 2 q^{4} + 2 q^{3} - q - 1\right)}^{r} + 161280 {\left(-q^{7} - q^{6} + q^{4} + q^{3} - q - 1\right)}^{r} \\
& + 13440 {\left(-q^{7} - q^{6} - 2 q^{4} - 2 q^{3} - q - 1\right)}^{r} + 161280 {\left(-q^{7} - q^{6} + q + 1\right)}^{r} + 48384 {\left(-q^{7} - 2 q^{6} - q^{5} - q^{2} - 2 q - 1\right)}^{r} \\
& + 60480 {\left(-q^{7} - 2 q^{6} - 2 q^{5} - q^{4} + q^{3} + 2 q^{2} + 2 q + 1\right)}^{r} + 2240 {\left(-q^{7} - 2 q^{6} - 3 q^{5} - q^{4} + q^{3} + 3 q^{2} + 2 q + 1\right)}^{r} \\
& + 20160 {\left(-q^{7} - 2 q^{6} - 3 q^{5} - 3 q^{4} - 3 q^{3} - 3 q^{2} - 2 q - 1\right)}^{r} + 10080 {\left(-q^{7} - 2 q^{6} + 3 q^{4} + 3 q^{3} - 2 q - 1\right)}^{r} \\
& + 10080 {\left(-q^{7} - 2 q^{6} + q^{4} - q^{3} + 2 q + 1\right)}^{r} + 945 {\left(-q^{7} - 3 q^{6} - q^{5} + 5 q^{4} + 5 q^{3} - q^{2} - 3 q - 1\right)}^{r} \\
& + 7560 {\left(-q^{7} - 3 q^{6} - 3 q^{5} - q^{4} + q^{3} + 3 q^{2} + 3 q + 1\right)}^{r} + 3780 {\left(-q^{7} - 3 q^{6} - 5 q^{5} - 7 q^{4} - 7 q^{3} - 5 q^{2} - 3 q - 1\right)}^{r} \\
& + 672 {\left(-q^{7} - 4 q^{6} - 6 q^{5} - 5 q^{4} - 5 q^{3} - 6 q^{2} - 4 q - 1\right)}^{r} + 63 {\left(-q^{7} - 5 q^{6} - 9 q^{5} - 5 q^{4} + 5 q^{3} + 9 q^{2} + 5 q + 1\right)}^{r} \\
& + {\left(-q^{7} - 7 q^{6} - 21 q^{5} - 35 q^{4} - 35 q^{3} - 21 q^{2} - 7 q - 1\right)}^{r} + 30240 {\left(-q^{7} + 2 q^{5} + q^{4} - q^{3} - 2 q^{2} + 1\right)}^{r} \\
& + 30240 {\left(-q^{7} + 2 q^{5} - q^{4} - q^{3} + 2 q^{2} - 1\right)}^{r} + 120960 {\left(-q^{7} + q^{5} + q^{4} - q^{3} - q^{2} + 1\right)}^{r} \\
& + 120960 {\left(-q^{7} + q^{5} - q^{4} - q^{3} + q^{2} - 1\right)}^{r} + 145152 {\left(-q^{7} + q^{5} + q^{2} - 1\right)}^{r} + 145152 {\left(-q^{7} + q^{5} - q^{2} + 1\right)}^{r} \\
& + 60480 {\left(-q^{7} + q^{4} + q^{3} - 1\right)}^{r} + 60480 {\left(-q^{7} - q^{4} + q^{3} + 1\right)}^{r} + 207360 {\left(-q^{7} + 1\right)}^{r} + 207360 {\left(-q^{7} - 1\right)}^{r}\Biggr) \left(-1\right)^{ r}.
\end{align*}}

In the particular case that $r = 2$, corresponding to an elliptic curve, the mixed Hodge polynomial is
\begin{align*}
\mu^{\str}(T_{\E_7}^2/W) =& {\left(q^{14} + q^{13} + 2 \, q^{12} + 7 \, q^{11} + 19 \, q^{10} + 42 \, q^{9} + 107 \, q^{8} + 204 \, q^{7}\right)} t^{14} \\
& + {\left(q^{12} + q^{11} + 3 \, q^{10} + 10 \, q^{9} + 25 \, q^{8} + 61 \, q^{7} + 107 \, q^{6}\right)} t^{12} \\
& + {\left(q^{10} + q^{9} + 3 \, q^{8} + 10 \, q^{7} + 25 \, q^{6} + 42 \, q^{5}\right)} t^{10} \\
& + {\left(q^{8} + q^{7} + 3 \, q^{6} + 10 \, q^{5} + 19 \, q^{4}\right)} t^{8} + {\left(q^{6} + q^{5} + 3 \, q^{4} + 7 \, q^{3}\right)} t^{6} \\
& + {\left(q^{4} + q^{3} + 2 \, q^{2}\right)} t^{4} + {\left(q^{2} + q\right)} t^{2} + 1\; .
\end{align*}

Additionally, the Poincar\'e polynomial is
\begin{align*}
    P^{\str}(T_{\E_7}^2/W) = 383 \, t^{14} + 208 \, t^{12} + 82 \, t^{10} + 34 \, t^{8} + 12 \, t^{6} + 4 \, t^{4} + 2 \, t^{2} + 1\; ,
\end{align*}
whereas the $E$-polynomial is
\begin{align*}
    E^{\str}(T_{\E_7}^2/W) =& q^{14} + q^{13} + 3 \, q^{12} + 8 \, q^{11} + 23 \, q^{10} + 53 \, q^{9} + 136 \, q^{8} + 276 \, q^{7} \\
    & + 136 \, q^{6} + 53 \, q^{5} + 23 \, q^{4} + 8 \, q^{3} + 3 \, q^{2} + q + 1\; .
\end{align*}
This polynomial satisfies the character positivity conjecture \ref{conj:positivity} and  is palindromic (Proposition \ref{prop:palindromic}). 
The stringy Euler characteristic of the character variety is $\chi^{\str}(T_{\E_7}^2/W)=726$.

\subsection{Group $G = \E_8$}

Applying Theorem \ref{thm:algorithm_CV}, the stringy mixed Hodge polynomial is the following.
\allowdisplaybreaks[4]
{\tiny
}
From this, setting $t = -1$ and $q = 1$, we can compute the stringy Euler characteristic to be
\begin{align*}
\chi^{\str}(T_{\E_8}^r/W) = & \frac{1}{907200} \cdot 256^{r - 1} + \frac{1}{7560} \cdot 128^{r - 1} + \frac{1}{480} \cdot 81^{r - 1} + \frac{2279}{810} \cdot 64^{r - 2} + \frac{1}{2} \cdot 36^{r - 1} + \frac{1829}{45} \cdot 32^{r - 2} \\
& + \frac{1}{4} \cdot 27^{r - 1} + \frac{1}{4} \cdot 25^{r - 1} + \frac{1}{12} \cdot 18^{r} + \frac{828841}{5400} \cdot 16^{r - 2} + \frac{1}{6} \cdot 12^{r} + \frac{191}{48} \cdot 9^{r - 1} \\
& + \frac{44491}{180} \cdot 8^{r - 2} + 6^{r} + \frac{3}{2} \cdot 5^{r - 1} + \frac{4591463}{5670} \cdot 4^{r - 3} + \frac{11}{4} \cdot 3^{r} + \frac{323819}{945} \cdot 2^{r - 4} + \frac{1135319}{64800}.
\end{align*}

Additionally, the stringy $E$-polynomial can be computed from the code in \cite{code_repository} to give the following.
{\tiny
}

In the particular case that $r = 2$, corresponding to an elliptic curve, the mixed Hodge polynomial is
\begin{align*}
\mu^{\str}(T_{\E_8}^2/W) =& q^{16} t^{16} + q^{15} t^{16} + 2 \, q^{14} t^{16} + 3 \, q^{13} t^{16} + q^{14} t^{14} + 12 \, q^{12} t^{16} + q^{13} t^{14} + 20 \, q^{11} t^{16} + 2 \, q^{12} t^{14} \\
& + 60 \, q^{10} t^{16} + 4 \, q^{11} t^{14} + 123 \, q^{9} t^{16} + q^{12} t^{12} + 14 \, q^{10} t^{14} + 286 \, q^{8} t^{16} + q^{11} t^{12} + 26 \, q^{9} t^{14} \\
& + 2 \, q^{10} t^{12} + 72 \, q^{8} t^{14} + 4 \, q^{9} t^{12} + 123 \, q^{7} t^{14} + q^{10} t^{10} + 15 \, q^{8} t^{12} + q^{9} t^{10} + 26 \, q^{7} t^{12} + 2 \, q^{8} t^{10} \\
& + 60 \, q^{6} t^{12} + 4 \, q^{7} t^{10} + q^{8} t^{8} + 14 \, q^{6} t^{10} + q^{7} t^{8} + 20 \, q^{5} t^{10} + 2 \, q^{6} t^{8} + 4 \, q^{5} t^{8} + q^{6} t^{6} + 12 \, q^{4} t^{8} \\
& + q^{5} t^{6} + 2 \, q^{4} t^{6} + 3 \, q^{3} t^{6} + q^{4} t^{4} + q^{3} t^{4} + 2 \, q^{2} t^{4} + q^{2} t^{2} + q t^{2} + 1\; .
\end{align*}

Additionally, the Poincar\'e polynomial is
\begin{align*}
    P^{\str}(T_{\E_8}^2/W) = 508 \, t^{16} + 243 \, t^{14} + 109 \, t^{12} + 42 \, t^{10} + 20 \, t^{8} + 7 \, t^{6} + 4 \, t^{4} + 2 \, t^{2} + 1\; ,
\end{align*}
whereas the $E$-polynomial is
\begin{align*}
E^{\str}(T_{\E_8}^2/W) =& q^{16} + q^{15} + 3 \, q^{14} + 4 \, q^{13} + 15 \, q^{12} + 25 \, q^{11} + 77 \, q^{10} + 154 \, q^{9} + 376 \, q^{8} + 154 \, q^{7} + 77 \, q^{6} \\
& + 25 \, q^{5} + 15 \, q^{4} + 4 \, q^{3} + 3 \, q^{2} + q + 1\; .
\end{align*}
This polynomial satisfies the character positivity conjecture \ref{conj:positivity} and is palindromic (Proposition \ref{prop:palindromic}). 
The stringy Euler characteristic of the character variety is $\chi^{\str}(T_{\E_8}^2/W)=936$.

\end{document}